\documentclass[11pt,reqno]{amsart}
\numberwithin{equation}{section}
\usepackage{amsmath,amsthm}
\usepackage{mathptmx}  
\usepackage{amssymb}
\usepackage{multicol}
\DeclareFontFamily{U}{BOONDOX-calo}{\skewchar\font=45 }
\DeclareFontShape{U}{BOONDOX-calo}{m}{n}{
  <-> s*[1.05] BOONDOX-r-calo}{}
\DeclareFontShape{U}{BOONDOX-calo}{b}{n}{
  <-> s*[1.05] BOONDOX-b-calo}{}
\DeclareMathAlphabet{\mathcalboondox}{U}{BOONDOX-calo}{m}{n}
\SetMathAlphabet{\mathcalboondox}{bold}{U}{BOONDOX-calo}{b}{n}
\DeclareMathAlphabet{\mathbcalboondox}{U}{BOONDOX-calo}{b}{n}

\usepackage{mathtools}
\usepackage{xfrac}

\newcommand{\mcb}[1]{{\mathcalboondox #1}}
\usepackage{amssymb}
\usepackage{amsmath}
\usepackage{amsthm}
\usepackage{diagbox}
\usepackage{booktabs}
\usepackage{enumitem}
\usepackage[sans]{dsfont} 
\usepackage{bbm}
\usepackage{changebar}
\usepackage[colorlinks]{hyperref}
\usepackage{a4wide}
\usepackage[usenames,dvipsnames,svgnames,table]{xcolor}
\usepackage{xcolor}

\usepackage{tikz}
\usetikzlibrary{arrows,shapes,automata,petri,positioning,calc}
\tikzset{
    place/.style={
        circle,
        thick,
        draw=black,
        fill=gray!50,
        minimum size=20mm,
    },
        state/.style={
        circle,
        thick,
        draw=blue!75,
        fill=blue!20,
        minimum size=20mm,
    },
}

\tikzset{
    cross/.pic = {
    \draw[rotate = 45] (-0.2,0) -- (0.2,0);
    \draw[rotate = 45] (0,-0.2) -- (0, 0.2);
    }
}

\usepackage{ulem}
\usepackage{setspace}

\newtheorem{thm}{Theorem}[section]

\theoremstyle{definition}
\newtheorem{prop}[thm]{Proposition}
\newtheorem{lem}[thm]{Lemma}
\newtheorem{definition}[thm]{Definition}
\newtheorem{rem}[thm]{Remark}

\newcommand\ve{\varepsilon}

\usepackage{comment}

\title[{Hydrodynamics of a $d$-dimensional long jumps symmetric exclusion with a slow barrier}
]{Hydrodynamics of a $d$-dimensional long jumps diffusive\\ symmetric exclusion with a slow barrier}
\author{Pedro Cardoso, Patr\'icia   Gon\c calves, Byron Jim\'enez-Oviedo}
\newcommand{\Addresses}{{
		\footnotesize
		Pedro Cardoso, \textsc{\noindent Institute for Applied Mathematics,  University of Bonn \\
	Endenicher Allee, no. 60, 53115 Bonn, Germany}\par\nopagebreak
		\textit{E-mail address}: \texttt{pgondimc@uni-bonn.de}
		
		\medskip
		
		Patr\'icia   Gon\c calves, \textsc{\noindent Center for Mathematical Analysis,  Geometry and Dynamical Systems \\
Instituto Superior T\'ecnico, Universidade de Lisboa\\
Av. Rovisco Pais, no. 1, 1049-001 Lisboa, Portugal}\par\nopagebreak
		\textit{E-mail address}: \texttt{pgoncalves@tecnico.ulisboa.pt}
		
		\medskip
		
		Byron Jim\'enez-Oviedo, \textsc{\noindent Escuela de Matem\'atica,  \\
Faculdad de Ciencias Exactas y Naturales, Universidad Nacional de Costa Rica\\
Heredia, Costa Rica}\par\nopagebreak
		\textit{E-mail address}: \texttt{byron.jimenez.oviedo@una.cr}
}}
\begin{document}
\begin{abstract} 
We obtain the hydrodynamic limit of symmetric long-jumps exclusion in $\mathbb{Z}^d$ (for $d \geq 1$), where the jump rate is inversely proportional to a power of the jump's length with exponent $\gamma+1$, where $\gamma \geq 2$. Moreover, movements between $\mathbb{Z}^{d-1} \times \mathbb{Z}_{-}^{*}$ and $\mathbb{Z}^{d-1} \times \mathbb N$ are slowed down by a factor $\alpha n^{-\beta}$ (with $\alpha>0$ and $\beta\geq 0$). In the hydrodynamic limit we  obtain the heat equation in $\mathbb{R}^d$ without boundary conditions or with {homogeneous} Neumann boundary conditions at both sides of the hyperplane $\mathcal{H}:=\mathbb{R}^{d-1} \times \{0\}$, depending on the values of $\beta$ and $\gamma$. {The homogeneous Neumann boundary conditions correspond to a decoupling across $\mathcal{H}$.} The (rather restrictive) condition in \cite{casodif} (for $d=1$) about the initial distribution satisfying an  entropy bound with respect to a Bernoulli product measure with constant parameter is weakened or completely dropped.
\end{abstract}
\maketitle

\section{Introduction}

One of the most famous challenges in  {thermodynamics} is to describe the space/time evolution of a physical quantity of interest in a fluid, such as the density of a gas. However, since the number of molecules is extremely large (of the order of Avogadro's number), a purely deterministic approach applying Newton's laws is not feasible. Fortunately, it is possible to go around this issue by using concepts from  {statistical mechanics}, where the macroscopic behavior of a fluid is analyzed from the rules governing the microscopic movements of its molecules.

This was the motivation of Spitzer when he introduced  the Interacting Particle Systems (IPS) to the mathematical community in \cite{spitzer} as a possible  research direction. In many cases, one studies models with a huge number of particles, that move through the sites of a lattice and whose evolution is described by random rules. One remarkable example of an IPS is the exclusion process, where the \textit{exclusion rule} ensures that every site is occupied by at most one particle. Despite its {simplicity, it} is a model that has been extensively studied in the literature because it captures many interesting phenomena that {are} shared by many other more complicated dynamics.  

In this work, we study the exclusion process evolving on $\mathbb{Z}^d$, whose dynamics swaps the position of particles according to some transition {probability, therefore} the number of particles is the conserved quantity in the system. This motivates the investigation of the space/time evolution of the density of particles, since the total mass is conserved by the dynamics. This description is given by the derivation of a partial differential equation (PDE), known in the literature of IPSs as the \textit{hydrodynamic limit}.

In this article, we {combine} the main features of the models presented in \cite{casodif} and \cite{tertumariana}. More specifically, we consider a symmetric exclusion process where  particles move on a $d-$dimensional lattice and there exists a slow barrier hindering some of the (possibly long) jumps. This is quite an interesting feature, since multidimensional IPSs are much less common in the literature than the unidimensional ones. More precisely, particles move according to a {transition probability} $p: \mathbb{Z}^d \rightarrow [0,1]$, given by
\begin{equation} \label{prob}
p( \hat{x} ) =
\begin{dcases}
\frac{c_{\gamma} |\hat{x}|^{-\gamma-1}}{d}, & \quad \text{if} \; \; \hat{x} \neq \hat{0} \; \; \text{and} \; \; \exists j \in \{1, \ldots, d\}: \hat{x}= \pm |\hat{x}| \hat{e}_j; \\
0,  & \quad \text{otherwise}.
\end{dcases}
\end{equation}
Above, $\gamma \geq 2$ (the inclusion of the case $\gamma =2$ is an improvement on the conditions imposed upon the model in \cite{casodif}, where it was assumed $\gamma > 2$), $c_\gamma$ is a normalizing constant and $\{ \hat{e}_1, \ldots, \hat{e}_d \}$ is the canonical basis of $\mathbb{R}^d$. From \eqref{prob}, we see that \textit{diagonal jumps} are forbidden. For instance, when $d=2$ it is not possible for a particle to move from $(1,3)$ directly to $(4,8)$ with only one jump; it must move first to an intermediate site, such as $(4,3)$ or $(1,8)$. The reason for this choice of $p(\cdot)$ is due to the fact that in this work we want to obtain the heat equation, written in terms of the Laplacian operator in dimension $d$, which is given by $\Delta = \sum_{j=1}^d \partial_{\hat{e}_j} \partial_{\hat{e}_j}$. Indeed, for every $j$ fixed, the effect of only considering jumps in the direction determined by $\hat{e}_j$ leads to $\partial_t \rho= \partial_{\hat{e}_j} \partial_{\hat{e}_j} \rho$ (analogously to the results in \cite{casodif}), see \eqref{deftilknj} and Lemma \ref{lemconvprinc} for more details.  

Next, we add a barrier  hindering the jumps between $\mathbb{Z}^{d-1} \times \mathbb{Z}_{-}^{*}$ and $\mathbb{Z}^{d-1} \times \mathbb N$, where $\mathbb{N}:= \{0, 1, \ldots\}$ and $\mathbb{Z}_{-}^{*}:= \mathbb{Z} - \mathbb{N} = \{-1, -2, \ldots\}$. In the same way as it is done in \cite{casodif} and \cite{tertumariana}, the slowing factor is $\alpha n^{-\beta}$, with $n \geq 1$, $\alpha >0$ and $\beta \geq 0$. We stress that only (some of the) movements affecting the last coordinate are affected. For instance, for $d=2$, the jump rate from $(-2,-2)$ to $(-2,2)$ is multiplied by $\alpha n^{-\beta}$; on the other hand jumps from $(-2,-3)$ to $(-2,-4)$, from $(-1,-3)$ to $(2,-3)$ and from $(2,1)$ to $(2,-1)$ are not realized through slow bonds. This is illustrated in Figure \ref{fig:dynamics} below. The slow barrier leads to a decoupling across the hyperplane $\mathbb{R}^{d-1} \times \{0\}$, preventing the flow of mass between $\mathbb{R}^{d-1} \times (-\infty,0)$ and $\mathbb{R}^{d-1} \times [0, \infty)$ when $\beta$ is large enough. This is represented by homogeneous Neumann boundary conditions at both sides of $\mathbb{R}^{d-1} \times \{0\}$.
\begin{figure}[h]
    \centering
    \includegraphics{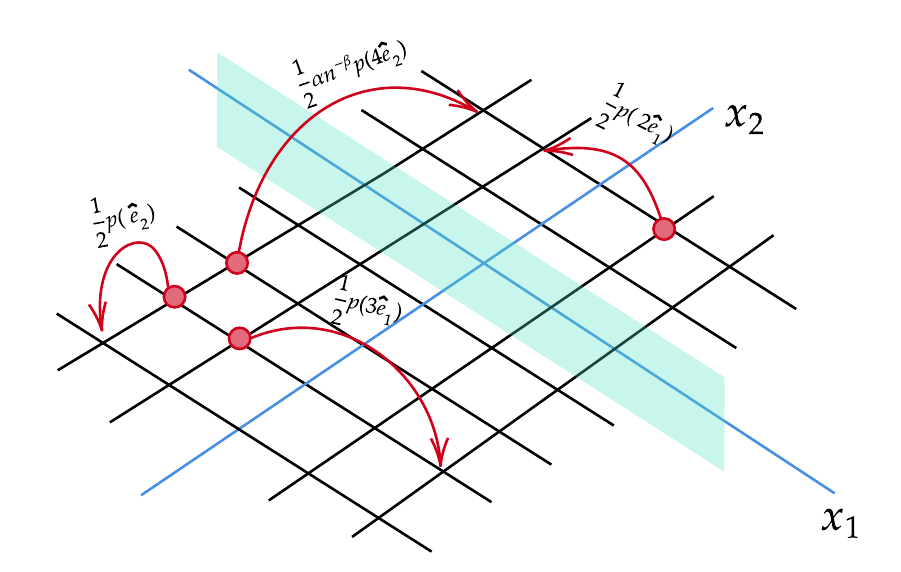}\caption{Dynamics of the long-jumps symmetric exclusion  with slow bonds in $\mathbb{Z}^2$. The factor $1/2$ comes from $d=2$. Recall from \eqref{prob} that $p(2 \hat{e}_1) = p(-2 \hat{e}_1)$ and $p(\hat{e}_2) = p(- \hat{e}_2)$.}\label{fig:dynamics}
\end{figure}

We observe that this work is a non-trivial generalization of the results in \cite{casodif, tertumariana}. Indeed, in \cite{casodif, tertumariana}, the initial distribution needs to satisfy an entropy bound with respect to a reference measure, the Bernoulli  product measure of constant parameter  (this is analogous to the case where $h$ in Definition \ref{defberprod} is taken as being constant). This assumption simplified the proofs and  {it was fundamental in \cite{casodif}} to derive \textit{global} energy estimates, since the spatial domain  {in \cite{casodif} was} \textit{unbounded}.  

In previous works (see \cite{casodif, superdif}), those global energy estimates could not be dispensed, since they were crucial to prove the uniqueness of weak solutions according to Oleinik's procedure; this uniqueness is, in turn, essential to obtain the hydrodynamic limit by using the tools of the entropy method introduced in \cite{GPV}. In this work, by dropping the assumption that $h$ in Definition \ref{defberprod} is constant, we can treat a much larger class of initial profiles in our system, therefore Theorem \ref{hydlim1} is much more general than the main result in \cite{casodif} for most of the values for $(\beta,\gamma)$. Furthermore, for a slightly smaller range of values of $(\beta,\gamma)$, we can avoid any kind of entropy bounds, see Theorem \ref{hydlim0}.

Here, the setting is also substantially different from \cite{tertumariana}, since our spatial domain is unbounded, in opposition to {the one in \cite{tertumariana}. A bounded} spatial domain  {(such as the one in \cite{tertumariana, MSV}) is} particularly simpler in a multidimensional context, since it allows us to define the density at the boundary in the sense given by the Trace Theorem (Theorem 1 in Section 5.5 of \cite{evans1998partial}). However, we do not have this possibility in this work and we only present some definition of the density at the boundary for the particular case when $d=1$. This can be done since \textit{local} energy estimates ensure that the density has a continuous representative in the neighborhood of the origin, see Proposition \ref{holderrepdif} for more details.  

The only drawback of this work, in comparison with \cite{casodif}, is to treat the very particular case where $d=1$, $\gamma > 2$ and $\beta = 1$. In order to apply the entropy method in this case, we would have to prove the uniqueness of weak solutions with a space of test functions satisfying Robin boundary conditions, and unfortunately we were not able to obtain this particular result. So far, we were only able to treat this case by assuming an entropy bound with respect to the product measure given by Definition \ref{defberprod} when $h$ is constant. Since this assumption is exactly the same as in \cite{casodif}, we  did not treat this case here. {We also note that the case where the transition probability $p(\cdot)$ is asymmetric and allowing  long-jumps has been explored in \cite{SS18} and the case of nearest-neighbor jumps has been treated in \cite{B12,X22a, X22b}.}

Next we describe the outline of this article. In Section \ref{sec:model} we give the details of our model, we define the notions of weak solutions that are deduced and we state the hydrodynamic limit. In Section \ref{secheur} we motivate the definitions of weak solutions presented in the previous section.
In Section \ref{sectight} we present the tightness of the sequence of empirical measures. In Section \ref{seccharac} we characterize the limiting point by showing that it is concentrated on a Dirac measure of a trajectory of measures which are absolutely continuous with respect to the Lebesgue measure, whose density is a weak solution of the corresponding hydrodynamic equation. Section \ref{secheurwithout} is devoted to the proof of some estimates which are applied in the proofs of the results of the previous section. Finally in Section \ref{secenerest} we treat the case $d=1$ and $\gamma >2$ in order to be able to define the density at the origin. Sections \ref{secheurwithout} and \ref{secenerest} both make use of Proposition \ref{bound}, whose proof is postponed to Appendix \ref{estdirfor}.

We complement the article with {five} appendices: Appendix \ref{secuniq} is devoted to the proof of uniqueness of weak solutions for our hydrodynamic equations, Appendix \ref{secdiscconv} is focused on the convergence from the discrete operators to the continuous ones,  Appendix \ref{estdirfor} presents a variety of bounds which were useful to show Proposition \ref{bound} and Appendix \ref{prooftight} is dedicated to the proof of the tightness criterium that we use since the process is evolving in infinite volume. Finally, in Appendix \ref{miscdynk} we obtain Proposition \ref{dynkform}, which extends Dynkin's formula as written in Appendix 1 of  \cite{kipnis1998scaling}, since our test functions are only of class $C^1$ regarding the temporal variable, instead of class $C^2$. {Up to the best of our knowledge, this extension has always been taken for granted in the literature, but was never proved rigorously.}

\section{Statement of results} \label{sec:model}

\subsection{The model} 

For every $k \geq 1$ and $u_1, \ldots u_k \in \mathbb{R}$, we will denote the \textit{maximum} norm of a vector $(u_1, \ldots, u_k) \in \mathbb{R}^k$ by
\begin{align*}
| (u_1, \ldots, u_k) |: = \max \{ |u_1|, \ldots, |u_k| \}.
\end{align*}
 Hereinafter, we fix a dimension $d \in \mathbb{Z}^*_{+}=\{1, 2,3, \ldots,   \}$ and we represent elements of $\mathbb{R}^d$ with a hat, e.g., $\hat{u}=(u_1, \ldots, u_d)$, unless we assume that $d=1$. We also make use of the canonical basis $\{\hat{e}_1, \ldots, \hat{e}_d\}$ of $\mathbb{R}^d$. Usually we denote elements of $\mathbb{R}^d$ by $\hat{u}$ and $\hat{v}$.

 Our goal is to study the space-time evolution of the density of particles in the symmetric exclusion process with long jumps evolving on the lattice $\mathbb{Z}^d$. This is a stochastic interacting particle system which allows at most one particle per site, therefore our space state is $\Omega:=\{0,1\}^{\mathbb{Z}^d}$. The elements of the lattice $\mathbb{Z}^d$ are called \textit{sites} and will be denoted by Latin letters, such as $\hat{x}$, $\hat{y}$ and $\hat{z}$. The elements of $\Omega$ are called \textit{configurations} and will be denoted by Greek letters, such as $\eta$. Moreover, we denote the number of particles at a site $\hat{x}$ on a configuration $\eta$ by $\eta(\hat{x})$; this means that the site $\hat{x}$ is \textit{empty} (resp. \textit{occupied}) if $\eta(\hat{x})=0$ (resp.  $\eta(\hat{x})=1$). 

The transition probability is denoted by $p: \mathbb{Z}^d \rightarrow [0,1]$ and it is given by \eqref{prob}.
Since we want to observe a diffusive behavior, $ \gamma \in [2, \infty)$ and $c_{\gamma}$ is a constant that turns $p(\cdot)$ into a probability. For $d=1$ and $\gamma > 2$, we recover the long-jumps form of the probability distribution given in \cite{casodif}. This motivates us to the denote (for the case $\gamma >2$) $\sigma^2:= d \sum_{r} r^2 p(r \hat{e}_d) = {2} \sum_{r =1 }^{\infty} c_{\gamma} r^{1-\gamma} < \infty$.

We observe that our dynamics \textit{does not allow diagonal jumps}: therefore, after a direction determined by $\hat{e}_j$ is chosen, a particle can only jump from $\hat{x}=(x_1, \ldots, x_{j-1},x_j,x_{j+1}, \ldots, x_d)$ to $\hat{y}=(x_1, \ldots, x_{j-1},y_j,x_{j+1}, \ldots, x_d)$, for some $y_j \neq x_j$. This motivates us to write $\hat{u}=(\hat{u}_{\star,j}, u_j)$, where $\hat{u}_{\star,j}$ is the element of $\mathbb{R}^{d-1}$ constructed by removing the $j$-th coordinate of $\hat{u}$. In the particular case $j=d$, we use the notation $\hat{u}_{\star}:=\hat{u}_{\star,d}$.

According to the \textit{exclusion rule}, given two sites $\hat{x},\hat{y}$, a particle only goes from $\hat{x}$ to $\hat{y}$ if the later one is empty and the former one is occupied, i.e. $\eta(\hat{y})=0$ and $\eta(\hat{x})=1$. This means that a movement between $\hat{x}$ and $\hat{y}$ only is possible if $\eta(\hat{x})[1-\eta(\hat{y})]+ \eta(\hat{y})[1-\eta(\hat{x})]=1$. This operation acts on an initial configuration $\eta \in \Omega$ and transforms it into $\eta^{\hat{x},\hat{y}} \in \Omega$, that is defined by
\begin{equation*}
\eta^{\hat{x},\hat{y}}(\hat{z})=
\begin{cases}
\eta(\hat{y}), & \quad \hat{z}= \hat{x}; \\
\eta(\hat{x}), & \quad \hat{z}= \hat{y}; \\
\eta(\hat{z}), & \quad \hat{z} \notin \{\hat{x}, \hat{y}\}.
\end{cases}
\end{equation*}
The elements of $\mcb {B}:=\big\{ \{\hat{x}, \hat{y} \}: \hat{x} \neq \hat{y} \in \mathbb{Z}^d \big\} $ are called \textit{bonds}. We denote $\mathbb{N}:=\{0, 1, 2, \ldots \}$ $\mathbb{Z}_{-}^{*}:= \mathbb{Z} - \mathbb{N} = \{-1, -2, \ldots \}$, and $\mcb{S}:=\big\{ \{\hat{x}, \hat{y} \} \in \mcb B: \hat{x} \in \mathbb{Z}^{d-1} \times \mathbb{Z}_{-}^{*}, \, \hat{y} \in \mathbb{Z}^{d-1} \times \mathbb{N}  \big\}$. Moreover, we denote $\mcb{F}:= \mcb{B} - \mcb{S}$. 

Next we introduce a slow barrier in the system. Here we fix $\alpha > 0$ and $\beta \geq 0$. Starting from a configuration $\eta$, a particle jumps from $\hat{x}$ to $\hat{y}$ at rate  $\eta(\hat{x}) [ 1 - \eta (\hat{y}) ] r_{\hat{x},\hat{y}}^n {p( \hat{y} - \hat{x} )}$, where 
\begin{equation}
r_{\hat{x},\hat{y}}^n:=
\begin{cases}
\alpha n^{-\beta}, & \quad \{\hat{x}, \hat{y}\} \in \mcb{S}; \\
1, & \quad \{\hat{x}, \hat{y}\} \in \mcb{F}.
\end{cases}
\end{equation}
When $\beta >0$, we have $\lim_{n \rightarrow \infty} r_{\hat{x},\hat{y}}^n=0$ for every $\{\hat{x},\hat{y}\}$ in $\mcb{S}$, creating a physical barrier between $\mathbb{Z}^{d-1} \times \mathbb{Z}_{-}^{*}$ and $\mathbb{Z}^{d-1} \times \mathbb{N}$. Then $\mcb{S}$ and $\mcb {F}$ are the sets of slow and fast bonds, respectively.  

We say that a function $f: \Omega \rightarrow \mathbb{R}$ is  \textit{local} if there exists a finite set $\Lambda \subset \mathbb{Z}^d$ such that
\begin{align*}
\forall \hat{x} \in \Lambda, \quad \eta_1(\hat{x}) = \eta_2(\hat{x})  \Rightarrow f(\eta_1) = f(\eta_2).
\end{align*}
Finally, the exclusion process with slow bonds is defined by the infinitesimal generator $ \mcb{L}_n$, given by
\begin{align*}
( \mcb{L}_n f ) (\eta) := \sum_{\hat{x},\hat{y}} p( \hat{y} - \hat{x}) \eta(\hat{x}) [1 - \eta (\hat{y}) ] r_{\hat{x},\hat{y}}^n \nabla_{\hat{x}, \hat{y}} f(\eta) =  \frac{1}{2} \sum_{\hat{x},\hat{y}} p( \hat{y} - \hat{x})  r_{\hat{x},\hat{y}}^n \nabla_{\hat{x}, \hat{y}} f(\eta),
\end{align*}
where $\nabla_{\hat{x}, \hat{y}} f(\eta) :=[ f ( \eta^{\hat{x},\hat{y}} ) - f (\eta ) ]$ {and} $f$ is a local function. In {the} last line and in what follows, unless it is stated differently, we will assume that our discrete variables in a summation always range over $\mathbb{Z}^d$. 

\subsection{Notation}

In this subsection, we begin by presenting the notation for some spaces of test functions.  For every $r \in \mathbb{N}$, we denote by $C_c^r(\mathbb{R}^d)$ the space of functions {$G: \mathbb{R}^d \rightarrow \mathbb{R}$} which are in $C^r(\mathbb{R}^d)$ and have a compact support; this means that there exists $b_1>0$ such that $G(\hat{u})=0$ for $|\hat{u}| \geq b_1$. We also denote the space of functions $G: [0, \infty) \times \mathbb{R}^d \rightarrow \mathbb{R}$ with compact support $K \subset [0, \infty) \times {\mathbb{R}^d}$ which are of class $r_1$ regarding the temporal variable and of class $r_2$ regarding the spatial variable by $C_c^{r_1,r_2} ( [0, \infty) \times \mathbb{R}^d )$. 

For every $j \in \{1, \ldots, d\}$, we denote by $\partial_{\hat{e}_j} G$ the partial derivative of $G$ in the direction of $\hat{e}_j$, i.e.
\begin{align*}
\partial_{\hat{e}_j} G (\hat{u}):= \lim_{h \rightarrow 0} \frac{G(\hat{u}+h \hat{e}_j) - G(\hat{u})}{h}
\end{align*}
when the above limit exists. We are aware of the more usual notation $\partial_{x_j} G$, but we will not use the later notation in order to avoid confusion with $x_j$, the $j$-th coordinate of the site $\hat{x} \in \mathbb{Z}^d$. Similarly, we denote the second derivative of $G$ in the direction determined by $\hat{e}_j$ by $\partial_{\hat{e}_j \hat{e}_j} G$.  Naturally, the Laplacian of $G$ is $\Delta G:= \sum_{j=1}^d \partial_{\hat{e}_j \hat{e}_j} G$. In the particular case $d=1$, we denote $\partial_{\hat{e}_d}$ by $\nabla$. 

When $\beta \in [0,1)$, the slow bonds are not strong enough to have a macroscopic effect in our system, and this motivates the choice of $\mcb S_{\textrm{Dif}}:=C_c^{1,2} ( [0, \infty) \times \mathbb{R}^d )$ as the space of test functions. On the other hand, when $\beta \geq 1$, the slow bonds will hinder the flow of mass between $\mathbb{R}_{-}^{d*}:=\mathbb{R}^{d-1} \times (-\infty,0)$ and $\mathbb{R}_{+}^{d*}:=\mathbb{R}^{d-1} \times [0,\infty)$, leading to functions which \textit{may be} discontinuous at $\mathcal{H}:= \{\hat{u} \in \mathbb{R}^d: u_d = 0\}$. This motivates the following definition:
	\begin{align*}
		\mcb{S}_{\textrm {Disc}}:=\{G: [0, \infty) \times \mathbb{R}^d  \rightarrow \mathbb{R}: \exists G^{-}, G^{+} \in \mcb S_{\textrm{Dif}} \; \; \text{s.t.} \; \; G=\mathbbm{1}_{\{ [0, \infty)  \times \mathbb{R}_{-}^{d*} \} } G^{-} + \mathbbm{1}_{ \{[0, \infty) \times { \mathbb{R}_{+}^{d*} } \}} G^{+} \}.
	\end{align*} 
In particular, by choosing $G^{-} = G^{+}$, we get $\mcb S_{\textrm{Dif}} \subsetneq \mcb{S}_{\textrm {Disc}}$. 
\begin{rem} \label{remdisc}
\textit{	
Let $j \in \{1, \ldots, d-1\}$. Then, $\partial_{\hat{e}_j \hat{e}_j} G$ is well-defined on $[0, \infty) \times \mathbb{R}^d$, for every} $G \in \mcb{S}_{\textrm {Disc}}$\textit{, with
\begin{align*}
\partial_{\hat{e}_j \hat{e}_j} G(s, u_1, \ldots, u_d)=
\begin{cases}
\partial_{\hat{e}_j \hat{e}_j} G^{-}(s, u_1, \ldots, u_d), \quad u_d <0, \; u_1, \ldots u_{d-1} \in \mathbb{R}, \; s \geq 0; \\
\partial_{\hat{e}_j \hat{e}_j} G^{+}(s, u_1, \ldots, u_d), \quad u_d \geq 0, \; u_1, \ldots u_{d-1} \in \mathbb{R}, \; s \geq 0.
\end{cases}
\end{align*}
On the other hand, for} $G \in \mcb{S}_{\textrm {Disc}}$ \textit{, $\partial_{\hat{e}_d \hat{e}_d} G$ may not be defined on $[0, \infty) \times \mathcal{H}$. Nevertheless, on $[0, \infty) \times (\mathbb{R}^d-\mathcal{H})$, it holds
\begin{align*}
	\partial_{\hat{e}_d \hat{e}_d} G(s, u_1, \ldots, u_d)=
	\begin{cases}
		\partial_{\hat{e}_d \hat{e}_d} G^{-}(s, u_1, \ldots, u_d), \quad u_d <0, \; u_1, \ldots u_{d-1} \in \mathbb{R}, \; s \geq 0; \\
		\partial_{\hat{e}_d \hat{e}_d} G^{+}(s, u_1, \ldots, u_d), \quad u_d > 0, \; u_1, \ldots u_{d-1} \in \mathbb{R}, \; s \geq 0.
	\end{cases}
\end{align*}
In particular, $\Delta G$ is well-defined on $[0, \infty) \times (\mathbb{R}^d-\mathcal{H})$, but may not be defined on $[0, \infty) \times  \mathcal{H}$, when} $G \in \mcb{S}_{\textrm {Disc}}$\textit{.}
\end{rem}
As it was illustrated in Remark \ref{remdisc}, the operators $\partial_{\hat{e}_d \hat{e}_d}$ and $\Delta$ are inadequate on $G \in \mcb{S}_{\textrm {Disc}}$, due to the possible discontinuities of $G(s, \cdot)$ at $\mathcal{H}$ for any $s \geq 0$ . Keeping this in mind, we define two extra operators $\mathbb{L}_d$ and $\mathbb{L}_{\Delta}$ in order to avoid those issues. More exactly, for any $G \in \mcb{S}_{\textrm {Disc}}$, $\mathbb{L}_d G$ and $\mathbb{L}_{\Delta} G$ are given by
\begin{equation} \label{defLd}
\mathbb{L}_d G(s, u_1, \ldots, u_d)=
	\begin{cases}
		\partial_{\hat{e}_d \hat{e}_d} G^{-}(s, u_1, \ldots, u_d), \quad u_d <0, \; u_1, \ldots u_{d-1} \in \mathbb{R}, \; s \geq 0; \\
		\partial_{\hat{e}_d \hat{e}_d} G^{+}(s, u_1, \ldots, u_d), \quad u_d \geq 0, \; u_1, \ldots u_{d-1} \in \mathbb{R}, \; s \geq 0,
	\end{cases}
\end{equation}
and $\mathbb{L}_{\Delta} G:=\mathbb{L}_{d} G + \sum_{j=1}^{d-1} \partial_{\hat{e}_j \hat{e}_j} G$. Thus, $\mathbb{L}_d G \equiv \partial_{\hat{e}_d \hat{e}_d} G$ and $\mathbb{L}_{\Delta} \equiv \Delta G$ on $[0, \infty) \times \mathbb{R}^d$, for any $G \in \mcb S_{\textrm{Dif}}$. Furthermore, for every $G \in \mcb{S}_{\textrm {Disc}}$, $\mathbb{L}_d G \equiv \partial_{\hat{e}_d \hat{e}_d} G$ and $\mathbb{L}_{\Delta} \equiv \Delta G$ on $[0, \infty) \times (\mathbb{R}^d - \mathcal{H})$. Moreover, we observe that $\mathbb{L}_d G$ and $\mathbb{L}_{\Delta}$ are well-defined on $[0, \infty) \times \mathbb{R}^d$, for any $G \in \mcb{S}_{\textrm {Disc}}$.

Given a measure space $(X, \mu)$, for every $p \in [1, \infty]$ let $L^p(X):=L^p(X, \mu)$. In this work, $\mu$ is the Lebesgue measure on $X$. We also denote the norm in $L^p(X)$ by $\| \cdot \|_{p,X}$ for $p< \infty$ and by $\| \cdot \|_{\infty}$ for $p = \infty$.  

\begin{rem} \label{remL1disc}
\textit{
Let} $G \in \mcb{S}_{\textrm {Disc}}$ \textit{and fix} $G^{-}, G^{+} \in \mcb S_{\textrm{Dif}}$ \textit{such that $G=\mathbbm{1}_{\{ [0, \infty)  \times \mathbb{R}_{-}^{d*} \} } G^{-} + \mathbbm{1}_{ \{[0, \infty) \times \mathbb{R}_{+}^{d*} \}} G^{+}$. Then
\begin{align*}
\sup_{s \geq 0} \| G (s, \cdot)  \|_{1, \mathbb{R}^{d}} \leq \sup_{s \geq 0} \| G^{-}(s, \cdot) \|_{1, \mathbb{R}^{d}} + \sup_{s \geq 0} \| G^{+}(s, \cdot) \|_{1, \mathbb{R}^{d}} < \infty.
\end{align*}
Reasoning exactly in the same way, we conclude that $\sup_{s \geq 0} \| \partial_sG(s, \cdot) \|_{1, \mathbb{R}^{d}}$ and $\sup_{s \geq 0} \| \mathbb{L}_{\Delta} G(s, \cdot) \|_{1, \mathbb{R}^{d}}$ are also finite. Furthermore, for $d=1$, it holds
\begin{align*}
	\sup_{s \geq 0} \big\{ |\nabla G(s, 0^{-})| + |\nabla G(s, 0^{+})| \} = \sup_{s \geq 0} \big\{ |\nabla G^{-}(s, 0)| + |\nabla G^{+}(s, 0)| \} \leq \|  \nabla G^{-} \|_{\infty} + \|  \nabla G^{+} \|_{\infty} < \infty.
\end{align*}
}
\end{rem}

In order to observe a macroscopic influence of the slow bonds in our system, our Markov process will be accelerated in time by a factor of $\Theta(n)$, given by
\begin{equation} \label{timescale}
\Theta(n):=
\begin{cases}
n^2, & \quad \gamma > 2; \\
n^2 / \log(n), & \quad \gamma = 2.
\end{cases}
\end{equation}
We make this choice in order to obtain a convergence (as $n \rightarrow \infty$) of a discrete operator to the Laplacian operator in $L^1(\mathbb{R}^d)$, see Proposition \ref{convprincdif} for more details. We study the process $\eta_t^n(\cdot):=\eta_{t \Theta(n)}(\cdot)$, whose infinitesimal generator is given by $\Theta(n) \mcb {L}_n$. We fix $T>0$, which leads to a finite time horizon $[0,T]$. We always assume that we start from a configuration in $\Omega$, therefore the exclusion rule ensures that  $(\eta^n_t)_{t \in [0,T]} \in \mcb{D}([0,T], \Omega)$,  the space of (right-continuous and with left limits) trajectories in $\Omega$. 
\begin{definition} \label{radon}
\textit{	
We denote the space of non-negative Radon measures on $\mathbb{R}^d$ endowed with the vague topology by $\mcb {M}^{+}$. }
\end{definition}
The \textit{empirical measure}, denoted by $\pi^n$, is defined by
\begin{equation}
\pi^n(\eta, \, d \hat{u}) := \frac{1}{n^d} \sum_{\hat{x}} \eta(\hat{x}) \delta_{\frac{\hat{x}}{n}}(d \hat{u}),
\end{equation}
where $\delta_{\hat{v}}$ is the Dirac measure on $\hat{v} \in \mathbb{R}^d$.  Denoting $\pi_t^n:= \pi^n (\eta_t^n, \, d \hat{u})$, we produce a Markov process $(\pi^n_t)_{t \in [0,T]} \in \mcb {D} ([0,T],  \mcb{M}^{+} )$, the space of the c\`adl\`ag (right-continuous and with left limits) trajectories on $\mcb{M}^{+}$. 
For any $G \in C^0_c(\mathbb{R}^d)$, we define
$
\int G d \pi^n(\eta) =  \langle \pi^n(\eta), G \rangle : = \frac{1}{n^d} \sum_{\hat{x}} \eta(\hat{x}) G (\tfrac{\hat{x}}{n} ). 
$
Here, we observe that the elements of $\mcb {M}^{+}$ typically act in $C^0_c(\mathbb{R}^d)$. However, keeping Remark \ref{remL1disc} in mind, we will be interested in applying the empirical measure to  functions that are in the larger space $L^1(\mathbb{R}^d)$. In order to do so, given $G$ in $L^1(\mathbb{R}^d)$ and $(G_j)_{j \geq 1} \subset C^0_c(\mathbb{R}^d)$ such that $(G_j)_{j \geq 1}$ converges to $G$ in $L^1(\mathbb{R}^d)$, we get
\begin{align*} 
\forall \varepsilon >0, \exists \; J(\varepsilon) \geq 1: \forall \eta \in \Omega, \forall j_1, j_2 > J(\varepsilon), \quad \big| \langle \pi^n(\eta), G_{j_1} \rangle - \langle \pi^n(\eta), G_{j_2} \rangle \big| < \varepsilon. 
\end{align*}
In {the} last line, we used the fact that $| \eta(\hat{x})| \leq 1$, for any $\eta \in \Omega$ and any $\hat{x} \in \mathbb{Z}^d$. From the last display, we get that for any fixed $\eta \in \Omega$, $\big( \langle \pi^n(\eta), G_{j} \rangle \big)_{j \geq 1}$ is a Cauchy sequence in $\mathbb{R}$, and in particular it is convergent. Therefore, for any $\eta \in \Omega$, we define
\begin{align} \label{intmedemp}
\forall G \in L^1(\mathbb{R}^d), \quad \langle \pi^n(\eta), G \rangle : =  \lim_{j \rightarrow \infty} \langle \pi^n(\eta), G_{j} \rangle.
\end{align}
The limit in \eqref{intmedemp} does not depend on the sequence $(G_j)_{j \geq 1} \subset C^0_c(\mathbb{R}^d)$ which approximates $G$ in $L^1(\mathbb{R}^d)$, therefore it is well-defined.

\begin{definition}
\textit{	
We say that the sequence $(\mu_n)_{n \geq 1}$ of probability measures on $\Omega$ is \textit{ associated } to a measurable profile $g: \mathbb{R}^d \rightarrow [0,1]$ if
\begin{equation} \label{assoc}
	\lim_{n \rightarrow \infty} \mu_n \big( \eta \in \Omega: \Big| \langle \pi^n, G \rangle - { \int_{\mathbb{R}^{d}} g(\hat{u}) G(\hat{u}) \; d \hat{u} }  \Big| > \delta    \big)=0, 
\end{equation}
for every $\delta>0$ and every $G \in C_c^0(\mathbb{R}^d)$. 
This means that the sequence of random empirical measures $\pi^n (\eta, d \hat{u})$ converges weakly to the deterministic measure $g(\hat{u}) d \hat{u}$ in $\mcb{M}^{+}$, which can be interpreted as a weak Law of Large Numbers. }
\end{definition}
We denote by $( \mathbb{P}_{\mu_n} )_{n \geq 1}$ (resp. $ (\mathbb{Q}_n)_{n \geq 1} $) the sequence of probability measures on $\mcb{D} ([0,T], \Omega)$ (resp. $ \mcb{D} ([0,T], \mcb{M}^{+})$) induced by a given sequence $(\mu_n)_{n \geq 1}$ of probability measures on $\Omega$. We also denote the expectation with respect to $\mathbb{P}_{\mu_n}$ by $\mathbb{E}_{\mu_n}$.

In this work we fix a measurable profile $g: \mathbb{R}^d \rightarrow [0,1]$ and a sequence $(\mu_n)_{n \geq 1}$ of probability measures on $\Omega$  associated  to $g$. This means that \eqref{assoc} holds for $t=0$ and {we} want to prove that $\pi^n (\eta_t^n, d\hat{u})$ converges weakly to a deterministic measure $\rho(t,\hat{u}) d\hat{u}$, in some sense, for every $t \in [0,T]$. We will also prove that $\rho:[0,T] \times \mathbb{R}^d \rightarrow [0,1]$ is a weak solution of a partial differential equation (PDE), the \textit{hydrodynamic equation}. In this work, the PDE will be the heat equation, since we want to observe a diffusive behavior, and for this reason we made the choice $\gamma \geq 2$. 

\subsubsection{Product measures}

Next we introduce a class of product measures defined on the state space $\Omega$ which will be relevant later on. {In order to do so, we define $\textrm{Ref}$, a space of reference functions.}
\begin{definition} \label{defrefprof}
\textit{	
We say that} $h: \mathbb{R}^d \rightarrow [0,1] \in \textrm{Ref}$ \textit{if it satisfies the following conditions:
\begin{enumerate}
\item
$h$ is bounded away from the boundary of $[0,1]$, i.e., there exist $a_h < b_h$ in $(0,1)$ such that $h(\hat{u}) \in [a_h, b_h]$, for every $\hat{u} \in \mathbb{R}^d$;
\item
$h$ is Lipschitz, i.e., there exists $L_h>0$ such that $|h(\hat{v}) - h(\hat{u})| \leq L_h |\hat{v}-\hat{u}|$, for every $\hat{u}, \hat{v} \in \mathbb{R}^d$;
\item
$h$ is constant far from the origin, i.e., there exist $A_h \in [a_h, b_h]$ and $K_h > 0$ such that $h(\hat{u})=A_h$ when $|\hat{u}| \geq K_h$.
\end{enumerate}
}
\end{definition}  
An element of $\textrm{Ref}$ is analogous to the function $h$ introduced in \cite{byrondif, byronsdif}. Moreover (in the same way as it was done in \cite{byrondif, byronsdif}), for every $h \in \textrm{Ref}$ and $n \geq 1$ we define a probability measure $\nu_h^n$ on $\Omega$.
\begin{definition} \label{defberprod}
\textit{Let} $h \in \textrm{Ref}$ \textit{and $n \geq 1$. We denote by $\nu_h^n$ the Bernoulli product measure on $\Omega$ with marginals given by 
\begin{align*}
\forall \hat{x} \in  \mathbb{Z}^d, \quad   1 - \nu_h^n \big( \eta(\hat{x}) = 0 \big) = \nu_h^n \big( \eta(\hat{x}) = 1 \big) = h ( \tfrac{\hat{x}}{n} ).
\end{align*}
}
\end{definition}
In order to use a procedure analogous to the one in \cite{tertuaihp, franco2015phase, byrondif, stefano_patricia, byronsdif, stefano}, in many cases we will require the following assumption:
\begin{equation}  \label{1entbound}
\exists h \in \textrm{Ref}: \quad H ( \mu_n | \nu_h^n ) \lesssim n^d.
\end{equation} 
Above the notation $f(n) \lesssim g(n)$ means that there exists $C >0$ satisfying $0 \leq f(n) \leq C g(n)$, $n \geq 1$.

\subsubsection{Sobolev spaces}

In this subsubsection, we assume $d=1$. Moreover, in this work we always assume that $I \subset \mathbb{R}$ is an open interval and $\bar{I}$ is the closure of $I$. Next we introduce the notation for Sobolev spaces, following Chapter 8 of \cite{brezis2010functional}. We say that $G \in \mcb{H}^1(I)$ if $G \in L^2(I)$ and there exists $g \in L^2(I)$ such that
	\begin{align*}
		\forall \phi \in C_c^{\infty}(I), \quad \int_{\mathbb{R}} G(u) {\nabla} \phi(u) \, du = - \int_{\mathbb{R}} g(u) \phi(u) \, du.  
	\end{align*}
We denote $g$ by $\partial_u G$, and it is the \textit{weak derivative} of $G$. We remark that $\mcb{H}^1(I)$ is a Hilbert space with norm given by $ \| G \|^2_{\mcb{H}^1(I)}  :=  \| G \|^2_{2,I}   +  \| \partial_u G \|^2_{2,I}  $. 
Next we present a classical result stated in \cite{brezis2010functional}, whose proof comes from a direct application of H\"older's inequality.  
\begin{prop} \label{holderrepdif}
\textit{	
For any $f \in \mcb H^{1}(I)$, there exists $\tilde{f}: \bar{I} \rightarrow \mathbb{R}$ such that $f = \tilde{f}$ almost everywhere on $I$ and $| \tilde{f}(v_2) - \tilde{f}(v_1) | \leq \| f \|_{\mcb H^{1}(I)} |v_2-v_1|^{ 1/2 }$, for {every} ${v_1}, v_2 \in \bar{I}$.
}
\end{prop}
We say $f \in \mcb H_{loc}^{1}(\mathbb{R}^{*})$ if there exist $a_1, a_2 >0$ and two intervals $I_1=(-a_1,0)$ and $I_2=(0, a_2)$ such that $f|_{I_j} \in \mcb{H}^1(I_j)$, $j=1,2$. {Due to} Proposition \ref{holderrepdif}, every $f \in \mcb{H}^1(I)$ has an unique continuous representative $\tilde{f}$ defined on $\bar{I}$, the closure of $I$, in the sense that $\tilde{f}(u)=f(u)$, almost everywhere in $I$. In particular, if $f \in \mcb H_{loc}^{1}(\mathbb{R}^{*})$, then $f$ has unique continuous representatives in $[-a_1,0]$ and $[0,a_2]$, and the limits
\begin{align*}
f(0^{-}):= \lim_{\varepsilon \rightarrow 0^{+}} \int_{-\varepsilon}^{0} f(u) \, du  ;  \quad f(0^{+}):=\lim_{\varepsilon \rightarrow 0^{+}} \int_{0}^{\varepsilon} f(u) \, du; 
\end{align*}   
are well-defined. In {the} next subsection, we describe the hydrodynamic equations that can we obtain; they depend on whether we are assuming \eqref{1entbound} or not, and  on the values of $\alpha$, $\beta$ and $\gamma$. 

\subsection{Hydrodynamic equations} \label{sechydeq}

We begin with the hydrodynamic equation obtained when $(\alpha, \beta)=(1,0)$ (i.e., $\alpha=1$ and $\beta=0$). Hereinafter, for every time-dependent function $f$ we denote its temporal derivative by {$\partial_r f$, $\partial_s f$ or} $\partial_t f$. Moreover, for every $s \in [0,T]$, we denote $f_s:=f(s, \cdot): \mathbb{R}^d \rightarrow \mathbb{R}$. 

{In this work, the hydrodynamic equations will be defined in terms of a weak formulation, i.e., its solutions must satisfy some integral equation. Keeping this in mind, for every $t \geq 0$, every $\rho \in L^{\infty}([0, T] \times \mathbb{R}^d)$,  every $G \in \mcb{S}_{\textrm {Disc}}$ and every $\kappa \in \mathbb{R}$, we define  
	\begin{align*}
		F_{\textrm{Dif}}(t, \rho,G, g, \kappa):= & \int_{\mathbb{R}^d} \rho_t(\hat{u}) G_t(\hat{u}) \, d \hat{u} - \int_{\mathbb{R}^d}  g(\hat{u}) G_0( \hat{u}) \, d \hat{u} 
		-  \int_0^t \int_{\mathbb{R}^d} \rho_s( \hat{u}) [ \kappa \mathbb{L}_{\Delta} + \partial_s ] G_s(\hat{u}) \, d \hat{u} \, ds.
	\end{align*}
Combining Remark \ref{remL1disc} with H\"older's inequality, we conclude that $F_{\textrm{Dif}}(t, \rho,G, g, \kappa)$ is well-defined.
}
 \begin{definition} 
 \textit{	
Let $\kappa > 0$ and $g: \mathbb{R}^d \rightarrow [0,1]$ be a measurable function. Then $\rho :[0,T] \times \mathbb{R}^d \rightarrow [0,1]$ is a weak solution of the heat equation with initial condition $g$
\begin{equation} \label{ehd}
\begin{cases}
\partial_t \rho_t ( \hat{u}) = \kappa \Delta \rho_t ( \hat{u}), &\quad (t,\hat{u}) \in [0,T] \times \mathbb{R}^d, \\
\rho_0(\hat{u}) = g(\hat{u}), &\quad \hat{u} \in \mathbb{R}^d, 
\end{cases}
\end{equation}
if for every $t \in [0,T]$ and for every} $G \in \mcb{S}_{\textrm{Dif}} \subsetneq  \mcb{S}_{\textrm{Disc}}$\textit{, we have} $F_{\textrm{Dif}}(t, \rho,G, g,\kappa)=0$\textit{.}
\end{definition}
Next we present the heat equation with Neumann boundary conditions; we observe that the elements of the space of test functions satisfy \textit{Neumann boundary conditions} themselves. Observe that $\partial_{\hat{e}_d}$ is the partial derivative in the direction determined by $\hat{e}_d$, which is normal to the hyperplane $\mathcal{H}$.
\begin{definition} \label{defehdn}
\textit{	
Let $\kappa>0$ and $g: \mathbb{R}^d \rightarrow [0,1]$ be a measurable function. Then $\rho:[0,T] \times \mathbb{R}^d \rightarrow [0,1]$ is a weak solution of the heat equation with Neumann boundary conditions and initial condition $g$ 
	\begin{equation} \label{ehdn}
		\begin{cases}
			\partial_t \rho_t( \hat{u}) = \kappa \Delta \rho_t( \hat{u}), &\quad (t, \hat{u}) \in [0,T] \times { \mathbb{R}^d - \mathcal{H}}, \\
			\partial_{\hat{e}_d} \rho_t(\hat{u}_{\star}, 0^+) = \partial_{\hat{e}_d} \rho_t(\hat{u}_{\star}, 0^-)= 0, &\quad t \in (0,T], \; \hat{u}_{\star} \in {\mathbb{R}^{d-1}}, \\
			\rho_0( \hat{u}) =  g(\hat{u}), &\quad \hat{u} \in \mathbb{R}^d, 
		\end{cases}
	\end{equation}
if for every} $G \in \mcb{S}_{\textrm {Neu}}:=\{ G \in \mcb{S}_{\textrm {Disc}}: \forall s \in [0,T], \forall \hat{u}_{\star} \in {\mathbb{R}^{d-1}}, \quad \partial_{\hat{e}_d} G^{-}_s(\hat{u}_{\star},0)=\partial_{\hat{e}_d} G^{+}_s(\hat{u}_{\star},0)=0  \}$ \textit{and every $t \in [0,T]$, we have} $F_{\textrm{Dif}}(t, \rho,G, g,\kappa)=0$\textit{.}
\end{definition}
\begin{rem} \label{remwelldef1}
\textit{We observe that if $\rho$ is a weak solution of \eqref{ehd} or \eqref{ehdn}, then $0 \leq \rho \leq 1$ on $[0, T] \times \mathbb{R}^{d}$. Combining this with Remark \ref{remL1disc}, we conclude that} $F_{\textrm{Dif}}(t, \rho,G, g, \kappa)$ \textit{is well-defined for every $t \in [0, T]$, every} $G \in \mcb{S}_{\textrm {Disc}}$\textit{, every measurable $g: \mathbb{R}^d \rightarrow [0,1]$ and every $\kappa >0$. }
\end{rem}

When \eqref{1entbound} holds, some technical results (known in the literature as \text{Replacement Lemmas}) can be derived. Under that assumption, for $d=1$ and $\gamma >2$ we can treat a wider space of test functions, such as $\mcb{S}_{\textrm {Disc}}$ instead of $\mcb{S}_{\textrm {Neu}}$. This wider space leads to a new integral equation, given by $F_{\textrm{Neu}}$ below. 
\begin{definition} \label{defehdnlee}
\textit{Let $\kappa>0$ and $g: \mathbb{R} \rightarrow [0,1]$ be a measurable function. Then $\rho:[0,T] \times \mathbb{R} \rightarrow [0,1]$ is a weak solution of the heat equation with Neumann boundary conditions and initial condition $g$
\begin{equation} \label{ehdnlee}
\begin{cases}
\partial_t \rho_t(u) = \kappa \Delta \rho_t(u), &\quad (t,u) \in [0,T] \times \mathbb{R}^{*},  \\
\nabla \rho_t(0^+) = \nabla \rho_t(0^-)= 0, &\quad t \in (0,T], \\
\rho_0(u) =  g(u), &\quad u \in \mathbb{R}, 
\end{cases}
\end{equation}
if the following two conditions hold:}
\begin{enumerate}
\item
 \textit{for every} $G \in \mcb{S}_{\textrm {Disc}}$ \textit{and every $t \in [0,T]$,  we have} $F_{\textrm{Neu}}(t, \rho,G, g, \kappa)=0$\textit{, where}
 	\begin{align*}
		F_{\textrm{Neu}}(t, \rho,G, g,\kappa):=& \int_{\mathbb{R}} \rho_t(u) G_t(u)\, du - \int_{\mathbb{R}} g(u) G_0(u)\, du - \int_0^t \int_{\mathbb{R}} \rho_s(u) [ \kappa { \mathbb{L}_{\Delta}} + \partial_s ] G_s(u)\, du \, ds \\
		+ & \kappa \int_0^t [ \nabla G_s(0^{-}) \rho_s(0^{-}) - \nabla G_s(0^{+}) \rho_s(0^{+}) ]\, ds;
	\end{align*}´
\item
\textit{$\rho_s \in \mcb H_{loc}^1(\mathbb{R}^{*})$, for a.e. $s \in [0,T]$.}
\end{enumerate}
\end{definition}
	\begin{rem} \label{remwelldef2}
		\textit{
		We observe that if $\rho$ is a weak solution of \eqref{ehdnlee}, then $0 \leq \rho \leq 1$ on $[0, T] \times \mathbb{R}$. Combining this with Remark \ref{remL1disc}, we conclude that} $F_{\textrm{Neu}}(t, \rho,G, g, \kappa)$ \textit{is well-defined for every $t \in [0, T]$, every} $G \in \mcb{S}_{\textrm {Disc}}$\textit{, every measurable $g: \mathbb{R} \rightarrow [0,1]$ and every $\kappa >0$. }
	\end{rem}

Due to $(2)$ in Definition \ref{defehdnlee}, $\rho_s(0^{-})$ and $\rho_s(0^{+})$ are well-defined for a.e. $s \in [0,T]$. Combining this with some Replacement Lemmas (which require \eqref{1entbound}), we are allowed to deal with $\mcb{S}_{\textrm {Disc}}$ {as the space} of test functions, which is significantly larger than $\mcb{S}_{\textrm {Neu}}$, the space of test functions of \eqref{ehdn}. Since all weak solutions of \eqref{ehdnlee} are also weak solutions of \eqref{ehdn}, the uniqueness of weak solutions of \eqref{ehdnlee} is a direct consequence of the uniqueness of weak solutions of \eqref{ehdn}. 

The following result is proved in Appendix \ref{secuniq}. {We stress that the boundedness of weak solutions for \eqref{ehd}, \eqref{ehdn} and \eqref{ehdnlee} is crucial to obtain next lemma.}
\begin{lem} \label{lemuniq}
\textit{	
There exists at most one weak solution of \eqref{ehd}, \eqref{ehdn} and \eqref{ehdnlee}.
}
\end{lem}
Now we state the main results of this work. We do not present the results for the case $\beta=1$ and $\gamma > 2$, because this case was already treated in Theorem 2.8 of \cite{casodif}.

\subsection{Main results}

We present two theorems, depending on whether we assume \eqref{1entbound} or not. We recall that $c_{2}$ is the normalizing constant $c_{\gamma}$ in \eqref{prob} for $\gamma=2$. The following notation will be useful to state the main results in a clean way. Let
\begin{equation} \label{defkdifgamma}
R_0:= \big( [1,\infty) \times\{2\} \big) \cup \big((1, \infty) \times (2, \infty) \big); \quad \kappa_{\gamma}:=
\begin{cases}
c_2/d, & \quad  \gamma=2; \\
\sigma^2 / 2d, & \quad  \gamma > 2.
\end{cases}
\end{equation}
 \begin{thm} \textbf{(Hydrodynamic Limit without an entropy bound)} \label{hydlim0}
Let $g: \mathbb{R}^d \rightarrow [0,1]$ be a measurable function and $(\mu_n)_{n \geq 1}$ be a sequence of probability measures on $\Omega$ associated to the initial profile $g$. Then, for every $t \in [0,T]$, every $G \in C_c^0(\mathbb{R}^d)$ and every $\delta >0$,
\begin{align*}
\lim_{n \rightarrow \infty} \mathbb{P}_{\mu_{n}} \Big( \eta^n_\cdot \in  \mcb D([0,T],\Omega) : \Big| \langle \pi^n_t, G \rangle - \int_{\mathbb{R}^d} G( \hat{u} ) \rho_t(\hat{u}) \, d \hat{u} \Big| > \delta \Big) =0, 
\end{align*}
where $\rho$ is the unique weak solution of
\begin{align*}
\begin{cases}
\eqref{ehd}, \quad  \text{with} \quad \kappa=\kappa_{\gamma}, &\quad  \text{if} \quad (\alpha, \beta)=(1,0)  \quad \text{and} \quad \gamma \geq 2; \\
\eqref{ehdn}, \quad \text{with} \quad \kappa=\kappa_{\gamma}, &\quad \text{if} \quad \alpha >0  \quad \text{and} \quad  (\beta,\gamma) \in R_0.
 \end{cases}
\end{align*}
\end{thm}
When we assume \eqref{1entbound}, we are able to obtain a result for a wider range of values of $(\alpha,\beta,\gamma)$. 
\begin{thm} \textbf{(Hydrodynamic Limit with an entropy bound)} \label{hydlim1}
Let $g: \mathbb{R}^d \rightarrow [0,1]$ be a measurable function and $(\mu_n)_{n \geq 1}$ be a sequence of probability measures on $\Omega$ associated to the initial profile $g$ satisfying \eqref{1entbound}. Then, for every $t \in [0,T]$, every $G \in C_c^0(\mathbb{R}^d)$ and every $\delta >0$,
\begin{align*}
\lim_{n \rightarrow \infty} \mathbb{P}_{\mu_{n}} \Big( \eta^n_\cdot \in  \mcb D([0,T],\Omega) : \Big| \langle \pi^n_t, G \rangle - \int_{\mathbb{R}^d} G(\hat{u}) \rho_t(\hat{u}) \, d \hat{u} \Big| > \delta \Big) =0, 
\end{align*}
where $\rho$ is the unique weak solution of
\begin{align*}
\begin{cases}
\eqref{ehd},\quad \text{with} \quad \kappa=\kappa_{\gamma},  &\quad  \text{if} \quad (\alpha,\beta,\gamma) \in (0, \infty) \times [0,1) \times [2, \infty) \quad \text{and} \quad d \geq 1; \\
\eqref{ehdn}, \quad \text{with} \quad \kappa=\kappa_{\gamma},  &\quad  \text{if} \quad \alpha >0 \quad \text{and} \quad (\beta,\gamma) \in R_0 \quad \text{and} \quad d \geq 2; \\
 \eqref{ehdn}, \quad \text{with} \quad \kappa=\kappa_{\gamma},  &\quad  \text{if} \quad (\alpha,\beta,\gamma) \in (0, \infty) \times [1,\infty) \times \{2\} \quad \text{and} \quad d = 1; \\
 \eqref{ehdnlee}, \quad \text{with} \quad \kappa=\kappa_{\gamma},  &\quad  \text{if} \quad (\alpha,\beta,\gamma) \in (0, \infty) \times (1,\infty) \times (2, \infty) \quad \text{and} \quad d = 1. 
\end{cases}
\end{align*}
\end{thm}
The rest of this work is devoted to the proof of the two previous theorems. In Section \ref{sectight}, we prove that the sequence $(\mathbb{Q}_n)_{n \geq 1}$ is tight with respect to the Skorohod topology of $\mcb D ( [0,T], \mcb{M}^+)$ and therefore it has at least one limiting point $\mathbb{Q}$. We prove that any limit point $\mathbb{Q}$ is concentrated on trajectories that satisfy the first (resp. second) condition of weak solutions of the corresponding hydrodynamic equations from the results of Sections \ref{seccharac} and \ref{secheurwithout} (resp. Section \ref{secenerest}). The necessary Replacement Lemmas are proved in Section \ref{secheurwithout} and the uniqueness of weak solutions to the hydrodynamic equations is derived in Appendix \ref{secuniq}. Finally, we present some auxiliary results {in Appendices \ref{secdiscconv}, \ref{estdirfor}, \ref{prooftight} and \ref{miscdynk}. We stress} that the uniqueness of weak solutions to the corresponding hydrodynamic equation implies the uniqueness of the limit point $\mathbb{Q}$ and the convergence of the sequence $(\mathbb{Q}_n)_{n \geq 1}$, leading to the desired result.

\section{Heuristic argument to deduce the hydrodynamic equations} \label{secheur}

Observe that {$\mcb S_{\textrm{Dif}} \subsetneq \mcb{S}_{\textrm {Disc}}$ and $ \mcb{S}_{\textrm {Neu}} \subsetneq \mcb{S}_{\textrm {Disc}}$}, therefore all test functions  in Theorems \ref{hydlim0} and \ref{hydlim1} belong to $\mcb{S}_{\textrm {Disc}}$. {Before we proceed, we state a result  which is known in the literature as Dynkin's formula, and we postpone its proof to Appendix \ref{miscdynk}.} 
\begin{prop} \label{dynkform}
\textit{
Let $n \geq 1$, $\gamma \geq 2$, $d \geq 1$ and} $G \in \mcb{S}_{\textrm {Disc}}$\textit{. Then the process $( \mcb M_{t}^{n}(G) )_{0 \leq t \leq T}$ defined by
	\begin{equation} \label{defMnt}
		\mcb M_{t}^{n}(G) := \langle \pi_{t}^{n},G_t \rangle - \langle \pi_{0}^{n},G_0  \rangle  -  \int_0^t   \langle  \pi_{s}^{n},\partial_s G_s  \rangle ds - \int_0^t  \Theta(n) \mcb {L}_n ( \langle  \pi_{s}^{n},G_s  \rangle ) ds
	\end{equation}
is a martingale with respect to the natural filtration $\{ \mcb F_t \}_{t \geq 0}$, where $\mcb F_t:=\sigma ( \{ \eta_s^n \}_{s \in [0,t] } )$ for every $t  \in [0, T]$. Moreover, for every $t \in [0, T]$ it holds
	\begin{equation} \label{defNnt}
		\mathbb{E} \Big[ \big( \mcb M_{t}^{n}(G) \big)^2 \big] =  \mathbb{E} \Big[  \int_0^t \Theta(n) [ \mcb {L}_n (  \langle \pi_{t}^{n},G_t \rangle^2 ) - 2 \langle \pi_{t}^{n},G_t \rangle  \mcb {L}_n (  \langle \pi_{t}^{n},G_t \rangle )  ] ds \Big].
	\end{equation}
}
\end{prop} 
\begin{rem}
	\textit{
When $G$ is of class $C^2$ regarding the time, the last result is a direct consequence of Lemma 5.1 in Appendix 1.5 of \cite{kipnis1998scaling}. This extra regularity was assumed in previous works, see \cite{casodif, superdif} e.g. However, in this work we \textit{do not} impose this extra regularity for our test functions. Therefore, the arguments in Appendix 1.5 of  \cite{kipnis1998scaling} need to be tweaked in order to obtain Proposition \ref{dynkform}. This is done in Appendix \ref{miscdynk}.  
}
\end{rem}
The last term on the right-hand side of \eqref{defMnt} is known in the literature as the \textit{integral term} and it will determine which hydrodynamic equation we will obtain, depending on the values of $\alpha$, $\beta$, $\gamma$ and $d$. Next, we perform some (heuristic) computations in order to study the $L^1(\mathbb{P}_{\mu_{n}})$-convergence of that integral term. According to this, we define some operators which act on $G \in \mcb{S}_{\textrm {Disc}}$.

For $\mcb C \in \{\mcb B, \mcb S\}$, we define the operator $\mcb{K}_{n, \mcb C}$ by 
\begin{equation}\label{op_Kn}
\mcb{K}_{n, \mcb C} G_s ( \tfrac{ \hat{x} }{n} ): = \sum_{  \hat{y} : \, \{ \hat{x},  \hat{y} \} \in \mcb C} [ G_s( \tfrac{ \hat{y}}{n}) -G_s( \tfrac{ \hat{x}}{n}) ] p(\hat{y}-\hat{x}).
\end{equation}
Next we treat two different cases: $\beta \in [0,1)$ and $(\beta,\gamma) \in R_0$.

\textbf{I.:} Case $\beta \in [0,1)$. In this case the space of test functions is $\mcb S_{\textrm{Dif}}$ and we make use of the next two results. Proposition \ref{convprincdif} (resp. Proposition \ref{replemsmallbeta}) is proved in Appendix \ref{secdiscconv} (resp. Section \ref{secheurwithout}).   
\begin{prop} \label{convprincdif}
\textit{
For every} $G \in \mcb S_{\textrm{Dif}}$ \textit{we have
\begin{align*} 
\lim_{n \rightarrow \infty} \frac{1}{n^d} \sum_{ \hat{x} } \; \sup_{s \in [0,T]} | \Theta(n)  \mcb{K}_{n, \mcb B} G_s (\tfrac{\hat{x}}{n} )  - \kappa_{\gamma}  \Delta G_s  (\tfrac{\hat{x}}{n} ) | =0. 
\end{align*}
}
\end{prop}
\begin{prop} \label{replemsmallbeta}
\textit{
Assume \eqref{1entbound}, $\beta \in [0,1)$ and $t \in [0,T]$. For every} $G \in \mcb S_{\textrm{Dif}}$ \textit{we have
\begin{equation*} 
 \lim_{n \rightarrow \infty} \mathbb{E}_{\mu_n} \Big[  \Big| \int_{0}^{t} \sum _{\hat{x} }  \frac{\Theta(n)}{n^d} \mcb {K}_{n,\mcb S} G_s \left( \tfrac{ \hat{x} }{n} \right)  \eta_{s}^{n}(\hat{x}) \, ds \Big| \Big] =0.
\end{equation*}
}
\end{prop}
From the symmetry of $p(\cdot)$, the integral term $\int_0^t  \Theta(n) \mcb {L}_n  \langle  \pi_{s}^{n},G_s \rangle ds$ can be rewritten as
\begin{align}
  \int_{0}^{t}  \sum _{\hat{x} }  \frac{\Theta(n)}{n^d} \mcb {K}_{n,\mcb B} G_s \left( \tfrac{ \hat{x} }{n} \right)  \eta_{s}^{n}(\hat{x}) \,  {ds -} \Big( 1 - \frac{\alpha}{n^{\beta}} \Big) \int_{0}^{t}  \sum _{\hat{x} }  \frac{\Theta(n)}{n^d} \mcb {K}_{n,\mcb S} G_s \left( \tfrac{ \hat{x} }{n} \right)  \eta_{s}^{n}(\hat{x})\, ds \label{gensmallbeta}.
\end{align}
Since our Markov process is the exclusion process, it holds
\begin{equation} \label{exc1part}
\forall n \geq 1, \forall s \in [0,T], \forall \hat{x} \in \mathbb{Z}^d, \quad 0 \leq  \eta_{s}^{n}(\hat{x}) \leq 1.
\end{equation}
Therefore, by applying Proposition \ref{convprincdif} the first term of \eqref{gensmallbeta} converges, as $n \rightarrow \infty$, to
\begin{align*}
 \kappa_{\gamma} \int_0^t \int_{\mathbb{R}^d} \rho_s( \hat{u} ) \Delta G_s( \hat{u} ) \, d \hat{u} \, ds, 
\end{align*}
in $L^1(\mathbb{P}_{\mu_n})$. Moreover, if $\alpha=1$ and $\beta=0$, the second term of \eqref{gensmallbeta} is equal to zero. Otherwise, assuming \eqref{1entbound} this term converges, as $n \rightarrow \infty$, to zero in $L^1(\mathbb{P}_{\mu_n})$, due to Proposition \ref{replemsmallbeta}. Therefore for $\beta \in [0,1)$ we get in the limit (as $n \rightarrow \infty$) the integral equation in \eqref{ehd}. 

\textbf{II.:} Case $(\beta,\gamma) \in R_0$. In this case the space of test functions is $\mcb S_{\gamma,d}$, given by
\begin{equation} \label{defsgammad}
\mcb S_{\gamma,d}:=
\begin{cases}
\mcb{S}_{\textrm {Neu}}, \quad \text{if} \quad \gamma = 2 \quad \text{or} \quad d \geq 2; \\
\mcb{S}_{\textrm {Disc}}, \quad \text{if} \quad \gamma > 2 \quad \text{and} \quad d = 1.
\end{cases}
\end{equation}
Since the test functions can be discontinuous at $\mathcal H$, the operator $\Theta(n)  \mcb{K}_{n, \mcb B} G_s$ is not assured to converge to $\kappa_{\gamma} \Delta G_s$ as $n \rightarrow \infty$  in $L^1( {\mathbb{R}}^d)$, as it was the case in Proposition \ref{convprincdif}. In order to address this issue, some adjustments are required. From \eqref{prob}, once that the direction of a jump (determined by $\hat{e}_j$, for some $j \in \{1, \ldots, d\}$) is chosen, a particle can only move from $(\hat{x}_{\star,j},x_j)$ to some $(\hat{x}_{\star,j},y_j)$, for $y_j \neq x_j$. Therefore  $\mcb{K}_{n, \mcb B}$ can be rewritten as $\sum_{j=1}^d \mcb{K}_{n, j}$, where $\mcb{K}_{n, j}$ is given by
\begin{equation} \label{deftilknj}
\mcb{K}_{n, j} G_s ( \tfrac{\hat{x}}{n} ) = \mcb{K}_{n, j} G_s  ( \tfrac{ \hat{x}_{\star, j}}{n} , \tfrac{x_j}{n}) : = \sum_{y_j \in \mathbb{Z} } [ G_s ( \tfrac{ \hat{x}_{\star, j}}{n} , \tfrac{y_j}{n} ) -G_s ( \tfrac{ \hat{x}_{\star, j}}{n} , \tfrac{x_j}{n}) ] p\big((y_j-x_j) \hat{e}_j \big).
\end{equation}
From Lemma \ref{lemconvprinc}, $\Theta(n) \mcb{K}_{n, j} G_s$ converges to $\kappa_{\gamma} \partial_{\hat{e}_j \hat{e}_j} G_s$ in $L^1( {\mathbb{R}}^d)$, as $n \rightarrow \infty$, when \textit{$G$ is of class $C^2$ along the direction determined by $\hat{e}_j$}; in particular, Proposition \ref{convprincdif} is a direct consequence of Lemma \ref{lemconvprinc}. However, for $G \in \mcb{S}_{\textrm {Disc}}$, $G$ may be discontinuous along the direction determined by $\hat{e}_d$, hence {it is not true that $\Theta(n) \mcb{K}_{n, d} G_s$ converges to $\kappa_{\gamma} \partial_{\hat{e}_d \hat{e}_d} G_s$ in $L^1(\mathbb{R}^d)$, as $n \rightarrow \infty$, since $\partial_{\hat{e}_d \hat{e}_d} G_s$ may not be defined on $\mathcal{H}$.}

{Alternatively, observe that $\mathbb{L}_d G_s$ is well-defined on $\mathbb{R}^d$ for any $G \in \mcb{S}_{\textrm {Disc}}$. Thus, Lemma \ref{lemconvprinc} motivates us to define some operator $\tilde{\mcb{K}}_{n, d}$ such that $\Theta(n) \tilde{\mcb{K}}_{n, d} G_s$ converges to $\kappa_{\gamma} \mathbb{L}_d G_s$ in $L^1(\mathbb{R}^d)$, as $n \rightarrow \infty$, for any $G \in \mcb{S}_{\textrm {Disc}}$. Keeping this in mind, observe that the term independent of $\alpha n^{-\beta}$ in \eqref{gensmallbeta} can be rewritten as
\begin{align*}
 \frac{\Theta(n)}{n^d}  \int_{0}^{t} \sum_{ \hat{x}_{\star} \in \mathbb{Z}^{d-1} } \Big\{ \sum_{j=1}^{d-1}   \sum _{x_d \in \mathbb{Z} } \mcb{K}_{n, j} G_s ( \tfrac{\hat{x}_{\star}}{n}, \tfrac{x_d}{n}  )     +   \Big[ \sum_{ x_d=0}^{\infty} \sum_{y_d=0}^{\infty} A^{n,\hat{x}_{\star}}_{x_d, y_d}(G_s^{+}) +  \sum_{x_d=-\infty}^{-1} \sum_{y_d=-\infty}^{-1} A^{n,\hat{x}_{\star}}_{x_d, y_d}(G_s^{-}) \Big]  \Big\} \eta_{s}^{n}(\hat{x}_{\star} , x_d )
  ds.
\end{align*}
In the last line, for any $H:\mathbb{R}^d \rightarrow \mathbb{R}$, any $\hat{x}_{\star} \in \mathbb{Z}^{d-1}$, any $x_d,y_d \in \mathbb{Z}$ and $n \geq 1$, $A^{n,\hat{x}_{\star}}_{x_d, y_d}(H)$ is given by
\begin{align} \label{defAnG}
 A^{n,\hat{x}_{\star}}_{x_d, y_d}(H):= [ H( \tfrac{ \hat{x}_{\star} }{n}, \tfrac{y_d}{n} ) - H( \tfrac{ \hat{x}_{\star} }{n}, \tfrac{x_d}{n} ) ] p \big( (y_d-x_d) \hat{e}_d \big).
\end{align}
The last display is the shorthand notation for the contribution of the fast bonds $\{ (\hat{x}_{\star}, x_d), \; (\hat{x}_{\star}, y_d)   \}$, that are aligned along the direction determined by $\hat{e}_d$. Since the bonds corresponding to either $x_d \geq 0, y_d <0$, or $x_d < 0, y_d \geq 0$ are slow, we are reduced to deal with sums either on $x_d \geq 0, y_d \geq 0$, or on $x_d < 0, y_d < 0$. For any fixed value of $(  \hat{x}_{\star}, x_d)$, we will subtract a remainder term $- \mcb{R}_{n,d} G_s \left( \tfrac{ \hat{x}_{\star} }{n}, \tfrac{x_d}{n} \right)$, in order to obtain the operator $\tilde{\mcb{K}}_{n, d}$ such that $\Theta(n) \tilde{\mcb{K}}_{n, d} G_s$ converges to $\kappa_{\gamma} \mathbb{L}_d G_s$ in $L^1(\mathbb{R}^d)$, as $n \rightarrow \infty$. More exactly, we define the operators $\mcb{R}_{n,d}$ and $\tilde{\mcb{K}}_{n,d}$ by}
\begin{equation}\label{op_Rn}
\mcb{R}_{n,d} G_s \left( \tfrac{ \hat{x} }{n} \right)= \mcb{R}_{n,d} G_s \left( \tfrac{ \hat{x}_{\star} }{n}, \tfrac{x_d}{n} \right) :=
\begin{dcases}
\frac{1}{n} \partial_{\hat{e}_d} G_s^{+}(\tfrac{\hat{x}_{\star}}{n},0) \sum_{y_d=0}^{\infty}   (y_d-x_d)  p \big( (y_d-x_d) \hat{e}_d \big), &   x_d \geq 0; \\
\frac{1}{n}  \partial_{\hat{e}_d} G_s^{-}(\tfrac{\hat{x}_{\star}}{n},0) \sum_{y_d=-\infty}^{-1}  (y_d-x_d)  p \big( (y_d-x_d) \hat{e}_d \big), & x_d \leq -1.
\end{dcases}
\end{equation}
and
\begin{equation}\label{op_Knbd}
\tilde{\mcb{K}}_{n,d} G_s \left( \tfrac{ \hat{x}_{\star} }{n}, \tfrac{x_d}{n} \right) :=
\begin{dcases}
- \mcb{R}_{n,d} G_s \left( \tfrac{ \hat{x}_{\star} }{n}, \tfrac{x_d}{n} \right)  + \sum_{y_d=0}^{\infty} A^{n,\hat{x}_{\star}}_{x_d, y_d}(G_s^{+}), &  x_d \geq 0; \\
- \mcb{R}_{n,d} G_s \left( \tfrac{ \hat{x}_{\star} }{n}, \tfrac{x_d}{n} \right)  +  \sum_{y_d=-\infty}^{-1} A^{n,\hat{x}_{\star}}_{x_d, y_d}(G_s^{-}), & x_d \leq -1.
\end{dcases}
\end{equation}
From Lemma \ref{lemconvprincdisc}, $\Theta(n) \tilde{\mcb{K}}_{n, d} G_s$ converges {to $\kappa_{\gamma} \mathbb{L}_d G_s$} {in $L^1(\mathbb{R}^d)$}, as $n \rightarrow \infty$, {when either $\gamma =2$ and $G \in \mcb{S}_{\textrm {Neu}}$; or $\gamma > 2$ and $G \in \mcb{S}_{\textrm {Disc}}$.} This motivates us to rewrite the integral term $\int_0^t  \Theta(n) \mcb {L}_n  \langle  \pi_{s}^{n},G_s \rangle \, ds$ as 
	\begin{align}
		&  \int_{0}^{t}  \sum _{ \hat{x} } \frac{\Theta(n)}{n^d}  \mcb {K}_{n}^{*} G_s \left( \tfrac{ \hat{x} }{n} \right)  \eta_{s}^{n}(\hat{x}) \, ds   + \alpha    \int_{0}^{t}  \sum _{ \hat{x} } \frac{\Theta(n)}{n^{d+\beta}} \mcb{K}_{n, \mcb S} G_s( \tfrac{ \hat{x} }{n} ) \eta_{s}^{n}(\hat{x})\, ds \label{extranb1} \\
		+ & \int_{0}^{t}  \sum _{ \hat{x} } \frac{\Theta(n)}{n^d} {\mcb{R}_{n,d}} G_s \left( \tfrac{ \hat{x} }{n} \right)  \eta_{s}^{n}(\hat{x}) \, ds \label{extranb2}, 
	\end{align}
where $\mcb{K}_{n}^{*}:= \tilde{\mcb{K}}_{n,d} +  \sum_{j=1}^{d-1} \mcb{K}_{n, j}$. {We treat} the first term of \eqref{extranb1} with Proposition \ref{convprincdisc} below. This proposition is a direct consequence of Lemmas \ref{lemconvprinc} and \ref{lemconvprincdisc}, proved in Appendix \ref{secdiscconv}.
\begin{prop} \label{convprincdisc}
\textit{	
	Assume that either $\gamma =2$ and} $G \in  \mcb{S}_{\textrm {Neu}}$\textit{; or $\gamma > 2$ and} $G \in \mcb{S}_{\textrm {Disc}}$\textit{. Then
	\begin{align*} 
		\lim_{n \rightarrow \infty} \frac{1}{n^d} \sum_{ \hat{x} } \sup_{s \in [0,T]} | \Theta(n)  \mcb{K}_{n}^{*} G_s (\tfrac{ \hat{x}}{n} )  - \kappa_{\gamma} {\mathbb{L}_{\Delta}} G_s  (\tfrac{\hat{x}}{n} ) | =0. 
	\end{align*}
}
\end{prop}
Combining the last result with \eqref{exc1part}, we get
\begin{align*}
 \lim_{n \rightarrow \infty} \mathbb{E}_{\mu_n} \Big[  \Big| \int_{0}^{t}  \sum _{ \hat{x} } \Big\{ \frac{\Theta(n)}{n^d}  \mcb {K}_{n}^{*} G_s \left( \tfrac{ \hat{x} }{n} \right) - \kappa_{\gamma}  \mathbb{L}_{\Delta}  G_s  (\tfrac{\hat{x}}{n} ) \Big\} \eta_{s}^{n}(\hat{x}) \, ds \Big| \Big] =0. 
\end{align*}
Therefore, from Remark \ref{remL1disc} and \eqref{intmedemp},  the first term of \eqref{extranb1} converges, as $n \rightarrow \infty$, to
\begin{align*}
	\kappa_{\gamma} \int_0^t \int_{\mathbb{R}^d} \rho_s( \hat{u}) \mathbb{L}_{\Delta} G_s(\hat{u}) \, d \hat{u} \, ds, 
\end{align*}
in $L^1(\mathbb{P}_{\mu_n})$. In order to treat the second term of \eqref{extranb1}, it is enough to prove that it holds
\begin{equation} \label{claim1largebeta}
\forall G \in \mcb{S}_{\textrm {Disc}}, \; \forall (\beta,\gamma) \in R_0, \quad \lim_{n \rightarrow \infty}  \sum_{ \hat{x}_{\star} \in { \mathbb{Z}^{d-1} } } \; \sum_{x_d \in \mathbb{Z}} \; \sup_{s \in [0,T]} \Big| \frac{\Theta(n)}{n^{d+\beta}} \mcb{K}_{n, \mcb S} G_s ( \tfrac{ \hat{x}_{\star} }{n}, \tfrac{x_d}{n} ) \Big|=0.
\end{equation}
From \eqref{prob} and the definition of $\mcb S$, for every $G \in \mcb{S}_{\textrm {Disc}}$ a particle can only jump between $\hat{x}$ and $\hat{y}$ for $\{\hat{x}, \hat{y} \} \in \mcb S$ when there exist $x_d \geq 0, y_d <0$ or $x_d < 0, y_d \geq0$ and $\hat{x}_{\star} \in \mathbb{Z}^{d-1}$ such that $\hat{x}=(\hat{x}_{\star}, x_d)$ and $\hat{y} =(\hat{y}_{\star}, y_d) =(\hat{x}_{\star}, y_d)$ (i.e., $\hat{x}_{\star}=\hat{y}_{\star}$). Performing simple algebraic manipulations, we get
\begin{equation} \label{rewritekns}
\mcb{K}_{n, \mcb S} G_s ( \tfrac{ \hat{x}_{\star} }{n}, \tfrac{x_d}{n} ):= 
\begin{dcases}
\sum_{y_d=-\infty}^{-1}  A^{n,\hat{x}_{\star}}_{x_d, y_d}(G_s), \quad & x_d \geq 0; \\
\sum_{y_d=0}^{\infty}  A^{n,\hat{x}_{\star}}_{x_d, y_d}(G_s), \quad & x_d <0.
\end{dcases}
\end{equation}
Since $G$ has compact support uniformly in time, this leads to
\begin{equation} \label{eqivmcbSG}
\sum_{ \hat{x}_{\star} \in \mathbb{Z}^{d-1} } \; \sum_{x_d \in \mathbb{Z}}  \mcb{K}_{n, \mcb S} G_s ( \tfrac{ \hat{x}_{\star} }{n}, \tfrac{x_d}{n} ) = \sum_{| \hat{x}_{\star}| < b_G n } \Big\{  \sum_{x_d=0}^{\infty} \; \sum_{y_d=-\infty}^{-1}  A^{n,\hat{x}_{\star}}_{x_d, y_d}(G_s) +    \sum_{x_d=-\infty}^{-1} \; \sum_{y_d=0}^{\infty}  A^{n,\hat{x}_{\star}}_{x_d, y_d}(G_s) \Big\},
\end{equation}
where $b_G$ is given by
\begin{equation} \label{defbg}
b_G:=\inf\{ b_1>0: | \hat{u} | \geq b_1 \Rightarrow \sup_{s \in [0,T]}|G_s(\hat{u})|=0\}.
\end{equation}
Indeed, for $| \hat{x}_{\star}| \geq b_G n$, we have $G_s( \tfrac{ \hat{x}_{\star} }{n}, \tfrac{y_d}{n} )=G_s( \tfrac{ \hat{x}_{\star} }{n}, \tfrac{x_d}{n} )= A^{n,\hat{x}_{\star}}_{x_d, y_d}(G_s)=0$, for every $x_d, y_d \in \mathbb{Z}$. Thus, from \eqref{eqivmcbSG} and the triangle inequality, the display inside the limit in \eqref{claim1largebeta} is bounded from above by
\begin{align*}
&  \frac{\Theta(n) 2\| G \|_{\infty} c_{\gamma}}{d n^{d+\beta}} \sum_{| \hat{x}_{\star}| < b_G n } \Big\{  \sum_{x_d=0}^{\infty} \; \sum_{y_d=-\infty}^{-1} |y_d-x_d|^{-1-\gamma} +    \sum_{x_d=-\infty}^{-1} \; \sum_{y_d=0}^{\infty} |y_d-x_d|^{-1-\gamma} \Big\} \\
& \leq \frac{\Theta(n) 4\| G \|_{\infty} (2 b_G n)^{d-1} c_{\gamma}}{d n^{d+\beta}}  \sum_{x_d=0}^{\infty} \; \sum_{y_d=-\infty}^{-1} (x_d-y_d)^{-1-\gamma} \lesssim \frac{\Theta(n) }{ n^{1+\beta}}  \sum_{z_d=1}^{\infty} (z_d)^{-\gamma} \lesssim \frac{\Theta(n)   }{ n^{1+\beta}},
\end{align*}
that goes to zero as $n \rightarrow \infty$, due to \eqref{timescale} and the fact that $(\beta,\gamma) \in R_0$. Above we used the fact that the sum over $z_d$ is convergent, since $\gamma \geq 2 >1$. Therefore we have proved \eqref{claim1largebeta}. By combining \eqref{exc1part} and \eqref{claim1largebeta}, the second term of \eqref{extranb1} converges, as $n \rightarrow \infty$, to zero in $L^1(\mathbb{P}_{\mu_n})$. It remains to treat the term in \eqref{extranb2}. When $\gamma = 2$ or $d \geq 2$, $G \in \mcb{S}_{\textrm {Neu}}$, and from \eqref{op_Rn} we get $\partial_{\hat{e}_d} G^{+}_s(\tfrac{\hat{x}_{\star}}{n},0) = \partial_{\hat{e}_d} G^{+}_s(\tfrac{\hat{x}_{\star}}{n},0) = {\mcb{R}_{n,d}} G_s \left( \tfrac{ \hat{x} }{n} \right) = 0$ for every $\hat{x} \in \mathbb{Z}^{d-1}$ and every $s \in [0,T]$. Therefore, we get the integral equation in \eqref{ehdn}. It remains to analyse $\gamma > 2$, $d=1$ and $G \in \mcb{S}_{\textrm {Disc}}$. Next for $\ell \geq 1$, we define the empirical averages on a box of size $\ell$ around $0$ as 
\begin{equation} \label{medbox1}
\overrightarrow{\eta}^{\ell}(0):=\frac{1}{\ell} \sum_{y_d=1}^{\ell} \eta(y_d)   \; \; \text{and} \; \; \overleftarrow{\eta}^{\ell}(0):=\frac{1}{\ell} \sum_{y_d=-\ell}^{-1} \eta(y_d).
\end{equation}
Thus, the term in \eqref{extranb2} can be treated with {the} next result, which is proved in Section \ref{secheurwithout}.   
\begin{prop} \label{replemlargebeta}
\textit{
Assume \eqref{1entbound}, $\gamma >2$, $t \in [0,T]$ and $d=1$. For every} $G \in \mcb{S}_{\textrm {Disc}}$ \textit{we have
	\begin{align*}
		& { \varlimsup_{\varepsilon \rightarrow 0^{+}} \varlimsup_{n \rightarrow \infty} } \mathbb{E}_{\mu_n} \Big[  \Big| \int_{0}^{t} \Big\{ \frac{\Theta(n)}{n^d} \sum_{x_d=0}^{\infty} { \mcb{R}_{n,d} } G_s ( \tfrac{x_d}{n} )  \eta_s^n(x_d) - \kappa_{\gamma}  \nabla G_s^{+}(0)   \overrightarrow{\eta}^{\varepsilon n }_s(0)  \Big\} \, ds \Big| \Big] =0, \\
		& { \varlimsup_{\varepsilon \rightarrow 0^{+}} \varlimsup_{n \rightarrow \infty} } \mathbb{E}_{\mu_n} \Big[  \Big| \int_{0}^{t} \Big\{ \frac{\Theta(n)}{n^d} \sum_{x_d=-\infty}^{-1} {  \mcb{R}_{n,d} } G_s ( \tfrac{x_d}{n} )  \eta_s^n(x_d) + \kappa_{\gamma}  \nabla G_s^{-}(0)   \overleftarrow{\eta}^{\varepsilon n }_s(0)   \Big\} \, ds \Big| \Big] =0.
	\end{align*}
}
\end{prop}
From {the} last result, \eqref{extranb2} converges, as $n \rightarrow \infty$, to
\begin{align*}
\kappa_{\gamma} \int_0^t [ \nabla G_s(0^{+}) \rho_s(0^{+}) - \nabla G_s(0^{-}) \rho_s(0^{-}) ] \, ds 
\end{align*}
in $L^1(\mathbb{P}_{\mu_n})$, and we get the integral equation in \eqref{ehdnlee}.

\section{Tightness} \label{sectight}

This section is devoted to the proof of tightness of the sequence $(\mathbb{Q}_{n})_{ n \geq 1 }$. We will make use of the following definition, which is motivated by the time horizon $[0, T]$.
\begin{definition} \label{deftaut}
We denote by $\mathcal{T}_T$ the set of stopping times $\tau$ such that $0 \leq \tau \leq T$. 
\end{definition}
In many articles dealing with exclusion processes evolving in a bounded spatial domain, to prove tightness it is enough to invoke  Propositions 1.6 and 1.7 in Chapter 4 of \cite{kipnis1998scaling}, see for example \cite{tertuaihp, franco2015phase, byrondif, stefano_patricia, byronsdif, stefano}. We observe that the later proposition  is  explicitly stated for processes evolving in bounded spatial domains, therefore rigorously it cannot be directly  applied for processes evolving in infinite volume, such as it is our context. To clarify the fact that we can actually use the same results as in the setting of bounded domains, we prove  in Appendix \ref{prooftight}  the following technical lemma. 
\begin{lem} \label{lemtight}
\textit{	
Assume that for every $\varepsilon >0$,
\begin{equation} \label{T1sdif}
\displaystyle \lim _{\lambda \rightarrow 0^+} \varlimsup_{n \rightarrow\infty} \sup_{\tau  \in \mathcal{T}_{T}, \;{  0 \leq } t \leq \lambda} {\mathbb{P}}_{\mu _{n}}\Big[\eta_{\cdot}^{n}\in \mcb D ( [0,T], \Omega) :\left\vert \langle\pi^{n}_{{ T \wedge( \tau+ t)}},G\rangle-\langle\pi^{n}_{\tau},G\rangle\right\vert > \ve \Big]  =0, 
\end{equation}
for any function $G\in C_c^2(\mathbb{R}^d)$. Then the sequence $(\mathbb{Q}_{n})_{ n \geq 1 }$ is tight.
}
\end{lem}
The following remark will be useful in the remainder of Section \ref{sectight}.
\begin{rem} \label{remsubset}
	\textit{
In the same way as constant functions are particular cases of polynomials, the elements of $C_{c}^{2}(\mathbb{R}^d)$ can be viewed as elements of} $\mcb S_{\textrm{Dif}}=C_c^{1,2} ( [0, \infty) \times \mathbb{R}^d )$\textit{. We denote this by} $C_{c}^{2}(\mathbb{R}^d) \subsetneq \mcb S_{\textrm{Dif}}$\textit{.}
\end{rem}
Due to \eqref{defMnt}, \eqref{T1sdif} is a direct consequence of  the next result combined with  Markov's and Chebychev's inequalities.
\begin{prop} 
\textit{
For $G \in C_{c}^2(\mathbb{R}^d)$, it holds
\begin{align} 
& \lim_{\lambda \rightarrow 0^+} \varlimsup_{n \rightarrow \infty} \sup_{\tau \in \mathcal{T}_T, \; { 0 \leq } t \leq \lambda} \mathbb{E}_{\mu_n} \left[ \Big| \int_{\tau}^{{ T \wedge( \tau+ t)}} \Theta(n)  \mcb L_{n}\langle \pi_{s}^{n},G\rangle \, ds \Big| \right] = 0, \label{condger1pr1} \\
& \lim_{\lambda \rightarrow 0^+} \varlimsup_{n \rightarrow \infty} \sup_{\tau \in \mathcal{T}_T, \; { 0 \leq t } \leq \lambda} \mathbb{E}_{\mu_n} \left[ \left( \mcb M_{\tau}^{n}(G) -  \mcb M_{{ T \wedge(\tau+t)}}^{n}(G) \right)^2 \right] = 0.  \label{condger1pr2}
\end{align}
}
\end{prop}
We observe that \eqref{condger1pr1} is a direct consequence of the following lemma.
\begin{lem} \label{lemtight1}
\textit{	
If $G \in C_{c}^{2}(\mathbb{R}^d)$, then $\sup_{s \in [0,T]} | \Theta(n)  \mcb L_{n}\langle \pi_{s}^{n},G\rangle| \lesssim 1$.
}
\end{lem}
\begin{proof}
We begin by claiming that
\begin{equation} \label{claimtight1}
\forall G \in C_{c}^{2}(\mathbb{R}^d){  \subsetneq } \mcb S_{\textrm{Dif}}, \quad \frac{\Theta(n)}{n^d} \sum _{\hat{x} } \big|  \mcb {K}_{n,\mcb S}{  G } \left( \tfrac{ \hat{x} }{n} \right) \big| \lesssim 1.  
\end{equation}
Indeed, from \eqref{defAnG} and \eqref{eqivmcbSG}, the sum in {the} last line is bounded from above by
\begin{align}
& \frac{c_{\gamma}\Theta(n)}{d n^{d+1}} \sum_{| \hat{x}_{\star}| < b_G n } \Big\{ {  \sum_{x_d=0}^{n-1} \; \sum_{y_d=-n+1}^{-1}  \| \partial_{\hat{e}_d} G\|_{\infty} |y_d-x_d|^{-\gamma}  +    \sum_{x_d=-n+1}^{-1} \; \sum_{y_d=0}^{n-1}   \| \partial_{\hat{e}_d} G \|_{\infty}  |y_d-x_d|^{-\gamma} } \Big\} \label{boundlemtight1a} \\
+ & \frac{c_{\gamma}\Theta(n)}{d n^{d}} \sum_{| \hat{x}_{\star}| < b_G n } \Big\{ 2 \sum_{x_d=0}^{n-1} \; \sum_{y_d=n}^{\infty} 2 \|  G\|_{\infty} (x_d+y_d)^{-1-\gamma}  +  2  \sum_{x_d=n}^{\infty} \; \sum_{y_d=0}^{\infty}  2 \|  G\|_{\infty} (x_d+y_d)^{-1-\gamma} \Big\}. \label{boundlemtight1b}
\end{align}
In the first line of {the} last display, we applied the Mean Value Theorem {for the terms corresponding to the pairs $(x_d,y_d)$ such that $\max\{ |x_d|, |y_d| \} < n$}. Next, we claim that the expression in \eqref{boundlemtight1a} is of order $1$; indeed, it is bounded from above by a constant, times
\begin{align*}
\frac{\Theta(n)}{ n^{d+1}} \sum_{| \hat{x}_{\star}| < b_G n }  \sum_{x_d=0}^{n-1} \; \sum_{y_d=1}^{n-1} (x_d+y_d)^{-\gamma} \lesssim  \frac{\Theta(n)}{ n^{2}} \sum_{z_d=1}^{2n-1} (z_d)^{1-\gamma}.
\end{align*}
For $\gamma  > 2$, the sum over $z_d$ is bounded for a constant independent of $n$ and {the} last display is of order $1$. Moreover, for $\gamma = 2$, the sum over $z_d$ is of order $\log(n)$ and {the} last display is of order $1$. 

It remains to treat the expression in \eqref{boundlemtight1b}; note that it is bounded from above by a constant, times
\begin{align*}
\frac{\Theta(n)}{ n^{d}} \sum_{| \hat{x}_{\star}| < b_G n } \; \sum_{y_d=0}^{\infty} \; \sum_{x_d=n}^{\infty}  (x_d+y_d)^{-1-\gamma} \lesssim  \frac{\Theta(n)}{ n^{\gamma}} \frac{1}{n^2}   \sum_{y_d=0}^{\infty} \; \sum_{x_d=n}^{\infty} \Big( \frac{x_d+y_d}{n} \Big)^{-1-\gamma} \lesssim \frac{\Theta(n)}{ n^{\gamma}} ,
\end{align*} 
which goes to zero as $n \rightarrow \infty$, due to \eqref{timescale}. Therefore, \eqref{claimtight1} is proved. Plugging \eqref{gensmallbeta} with \eqref{exc1part} and the fact that $|1-\alpha n^{-\beta}| \lesssim 1$, then 
\begin{align*}
& \sup_{s \in [0,T]} | \Theta(n)  \mcb L_{n}\langle \pi_{s}^{n},G\rangle| \\
\lesssim & \frac{1}{n^d}   \sum_{\hat{x}}  | \Theta(n) \mcb K_{n, \mcb B} G( \tfrac{\hat{x}}{n} )  - \kappa_{\gamma} \Delta    G  (\tfrac{\hat{x}}{n} )    |  + \kappa_{\gamma}  \frac{1}{n^d}   \sum_{\hat{x}} | \Delta    G  (\tfrac{\hat{x}}{n} )| +  \frac{\Theta(n)}{n^d} \sum _{\hat{x} } |  \mcb {K}_{n,\mcb S} { G } \left( \tfrac{ \hat{x} }{n} \right) |   \lesssim 1.
\end{align*}
In {the} last display we used Proposition \ref{convprincdif} (resp. \eqref{claimtight1}), to bound the first term (resp. third term), due to the fact that $G \in C_{c}^{2}(\mathbb{R}^d) { \subsetneq } \mcb S_{\textrm{Dif}}$.
\end{proof}
In order to prove \eqref{condger1pr2}, which leads to the tightness of $(\mathbb{Q}_n)_{n \geq 1}$, it is only necessary to consider $G \in C_{c}^{2}(\mathbb{R}^d)$. Nevertheless, Propositions \ref{caraclimsementbnd} and \ref{caraclimsementbnd2} below require us to deal with test functions which depend also on time. This motivates us to state Lemma \ref{lemtight2} below, since the expression in \eqref{condger2pr3} is essential to obtain the results in Section \ref{seccharac}. Afterwards, we obtain \eqref{condger1pr2} from Lemma \ref{lemtight2}, which is proved subsequently.  
\begin{lem} \label{lemtight2}
\textit{	
Define $\mcb S_{\beta}$ by} $\mcb S_{\beta}:=\mcb S_{\textrm{Dif}}$ \textit{if $\beta=0$ and} $\mcb S_{\beta}:=\mcb{S}_{\textrm {Disc}}$ \textit{if $\beta>0$.  Let $G \in \mcb S_{\beta}$. It holds 
\begin{equation} \label{condger2pr1}
 \sup_{ {t \in} [0,T]} \big| \Theta(n) [ \mcb {L}_n (  \langle \pi_{t}^{n},G_t \rangle^2 ) - 2 \langle \pi_{t}^{n},G_t \rangle  \mcb {L}_n (  \langle \pi_{t}^{n},G_t \rangle )  ] \big| \lesssim \frac{1}{\log(n)}.
\end{equation}
In particular, it holds
\begin{align} 
& \forall \delta_1 > 0, \quad \lim_{n \rightarrow \infty} \mathbb{P}_{\mu_n} \Big( \sup_{t \in [0,T]} | \mcb M_{t}^{n}(G) | > \delta_1 \Big) =0. \label{condger2pr3}
\end{align}
}
\end{lem}
Keeping Remark \ref{remsubset} in mind, we observe that $C_{c}^{2}(\mathbb{R}^d) \subsetneq \mcb S_{\textrm{Dif}} \subsetneq \mcb{S}_{\textrm {Disc}}$. Thus, since $( \mcb M_{t}^{n}(G) )_{0 \leq t \leq T}$ is a martingale, from \eqref{defNnt} we get that the expectation in \eqref{condger1pr2} is bounded from above by
\begin{align*}
  \mathbb{E}_{\mu_n} \Big[ \int_{0}^{T} \sup_{s \in [0,T]} \big| \Theta(n) \left(\mcb L_{n}[\langle \pi_{s}^{n},G_s \rangle]^{2}-2\langle \pi_{s}^{n},G_s \rangle \mcb L_{n} \langle \pi_{s}^{n},G_s  \rangle\right) \, ds \big| \Big]  \lesssim  \frac{T}{\log(n)},
\end{align*}
that vanishes as $n \rightarrow \infty$, leading to \eqref{condger1pr2}. In {the} last line we applied \eqref{condger2pr1}.

{Finally, we end this section by proving Lemma \ref{lemtight2}.}

\begin{proof}[Proof of Lemma \ref{lemtight2}]

{By assuming \eqref{condger2pr1}, Doob's inequality leads to \eqref{condger2pr3}.} It remains to prove \eqref{condger2pr1}. In order to do so, we claim that
\begin{equation} \label{claimtight2}
\forall H \in  \mcb S_{\textrm{Dif}}, \quad \sum _{\hat{x}, \hat{y} } \frac{\Theta(n)}{n^{2d} } \sup_{s \in [0,T]} \big[ H_s( \tfrac{\hat{y}}{n}) - H_s( \tfrac{\hat{x}}{n}) \big]^2  p(\hat{y}-\hat{x}) \lesssim \frac{1}{ n^{d-1} \log(n) }.  
\end{equation}
Indeed, from \eqref{prob}, the sum in {the} last display can be rewritten as
\begin{align} \label{boundtight2}
\frac{c_{\gamma}\Theta(n)}{d n^{2d} }  \sum_{j=1}^d  \sum_{ | \hat{x}_{\star,j}| \leq b_{{H}} n } \;   \sum_{x_j  \in \mathbb{Z}} \;  \sum_{y_j \neq x_j \in \mathbb{Z}} \;  \sup_{s \in [0,T]} \big[ {H}_s ( \hat{x}_{\star,j},\tfrac{y_j}{n} ) - {H}_s (\hat{x}_{\star,j}, \tfrac{x_j}{n} ) \big]^2  |y_j - x_j|^{-1-\gamma},
\end{align}
where $b_{{H}}$ is defined in \eqref{defbg} {for} ${G=H}$. We treat the cases $\gamma =2$ and $\gamma > 2$ separately. For $\gamma=2$, from the Mean Value {Theorem, the} double sum over $x_j \neq y_j \in \mathbb{Z}$ in \eqref{boundtight2} is bounded from above by
\begin{align*}
 &   \sum_{|x_j| \leq b_{{H}} n} \;  \sum_{|y_j| > 2 b_{{H}} n} { (2 \| H \|_{\infty})^2 }  |y_j - x_j|^{-1-\gamma} + \sum_{|y_j| \leq 2 b_{{H}} n} \;  \sum_{|x_j| > 3 b_{{H}} n}  {  (2 \| H \|_{\infty})^2 } |y_j - x_j|^{-1-\gamma}  \\
 +& \sum_{|y_j| \leq 2 b_{{H}} n} \;  \sum_{|x_j| \leq 3 b_{{H}} n}  \frac{\| \partial_{\hat{e}_j} {H}  \|_{\infty}^2}{n^2} (y_j-x_j)^2  |y_j - x_j|^{-1-\gamma} \mathbbm{1}_{ \{ y_j \neq x_j \} }  \\
 \lesssim & n^{-1} \frac{1}{n^2} \Bigg[ \sum_{|x_j| \leq b_{{H}} n} \;  \sum_{y_j > 2 b_{{H}} n} \Big( \frac{y_j-x_j}{n} \Big)^{-3} + \sum_{|y_j| \leq 2 b_{{H}} n} \;  \sum_{x_j > 3 b_{{H}} n} \Big( \frac{x_j-y_j}{n} \Big)^{-3} \Bigg] +   \sum_{|y_j| \leq 2 b_{{H}} n} \;  \sum_{|x_j| \leq 3 b_{{H}} n} \frac{1}{n^2}   \\
  \lesssim & n^{-1} + n^{-1} + 1. 
\end{align*}
In order to obtain the first line of the last display, we applied the triangle inequality, which leads to the upper bound $\sup_{s \in [0,T]} \big[ H_s ( \hat{x}_{\star,j},\tfrac{y_j}{n} ) - H_s (\hat{x}_{\star,j}, \tfrac{x_j}{n} ) \big]^2 \leq (2 \| H \|_{\infty})^2$. Moreover, in the third line, we used the fact that $\gamma =2 >1$ to get $(y_j-x_j)^2  |y_j - x_j|^{-1-\gamma} \mathbbm{1}_{ \{ y_j \neq x_j \} } \leq 1$. Therefore, the expression in \eqref{boundtight2} is bounded from above, by a constant times
\begin{align*}
\frac{c_{\gamma}\Theta(n)}{d n^{2d} } \sum_{j=1}^d \; \sum_{ | \hat{x}_{\star,j}| \leq b_{{H}} n } 1 \lesssim \frac{\Theta(n)}{ n^{d+1} } = \frac{n^2}{ n^{d+1} \log(n) } = \frac{1}{ n^{d-1} \log(n) }, 
\end{align*}
leading to \eqref{claimtight2}. On the other hand, for $\gamma>2$, by applying the Mean Value Theorem, the double sum over $x_j \neq y_j \in \mathbb{Z}$ in \eqref{boundtight2} is bounded from above by
\begin{align*}
  & \| \partial_{\hat{e}_j}  \|_{\infty}^2 \Big[ \sum_{|y_j| > 2 b_{{H}} n}  \; \sum_{ |x_j| \leq b_{{H}} n } \Big( \frac{y_j-x_j}{n} \Big)^2  |y_j - x_j|^{-1-\gamma} +  \sum_{|y_j| \leq 2 b_{{H}} n} \;  \sum_{ x_j \neq y_j } \Big( \frac{y_j-x_j}{n}  \Big)^2  |y_j - x_j|^{-1-\gamma}  \Big] \\
\lesssim & \frac{1}{n^2}    \sum_{ |x_j| \leq 2 b_{{H}} n } \;  \sum_{z_j=1 }^{\infty}  (z_j)^{1-\gamma} \lesssim \frac{1}{n}. 
\end{align*}
Therefore, \eqref{boundtight2} is bounded from above, by a constant times
\begin{align*}
\frac{c_{\gamma}\Theta(n)}{d n^{2d} } \sum_{j=1}^d \; \sum_{ | \hat{x}_{\star,j}| \leq b_{{H}} n } \; \frac{1}{n} \lesssim \frac{\Theta(n)}{ n^{d+2} } = \frac{n^2}{ n^{d+2} } = \frac{1}{ n^{d}  } \lesssim \frac{1}{ n^{d-1} \log(n) }, 
\end{align*}
leading again to \eqref{claimtight2}. 

 Next we prove \eqref{condger2pr1} from \eqref{claimtight2}. Performing simple algebraic manipulations, we get for $G\in \mcb{S}_{\textrm {Disc}}$  (which includes the case $G\in \mcb S_{\textrm{Dif}}$), that the left-hand side of \eqref{condger2pr1} is bounded from above by
\begin{equation} \label{tightcond2eq1}
 \frac{\Theta(n)}{2n^{2d}} \Big[  \sum_{\{\hat{x}, \hat{y} \} \in \mcb F } \;  \sup_{s \in [0,T]} \big[ G_s( \tfrac{\hat{y}}{n}) - G_s( \tfrac{\hat{x}}{n}) \big]^2  p(\hat{y}-\hat{x}) +  \sum_{\{\hat{x}, \hat{y} \} \in \mcb S } \frac{\alpha}{n^{\beta}} \sup_{s \in [0,T]} \big[ G_s( \tfrac{\hat{y}}{n}) - G_s( \tfrac{\hat{x}}{n}) \big]^2  p(\hat{y}-\hat{x}) \Big]. 
\end{equation}
In {the} last display we used \eqref{exc1part}. We treat the cases $\beta=0$ and $\beta >0$ separately. For $\beta=0$, we have $\mcb S_{\beta} =  \mcb S_{\textrm{Dif}}$ and the expression in \eqref{tightcond2eq1} is bounded from above by a constant times
\begin{align*}
 \sum _{\hat{x}, \hat{y} } \frac{\Theta(n)}{n^{2d} } \sup_{s \in [0,T]}  \big[ G_s( \tfrac{\hat{y}}{n}) - G_s( \tfrac{\hat{x}}{n}) \big]^2  p(\hat{y}-\hat{x}),
\end{align*}
hence an application of \eqref{claimtight2} {for} ${H=G}$ leads to the desired result. On the other hand, for $\beta >0$ we have $\mcb S_{\beta} =  \mcb{S}_{\textrm {Disc}}$ and the sum over $\{\hat{x}, \hat{y} \} \in \mcb F $ in \eqref{tightcond2eq1} is bounded from above by
\begin{align*}
\sum _{\hat{x}, \hat{y} } \frac{\Theta(n)}{2n^{2d} } \Big\{ \sup_{s \in [0,T]} \big[ G_s^{+} ( \tfrac{\hat{y}}{n}) - G_s^{+}( \tfrac{\hat{x}}{n}) \big]^2  + \sup_{s \in [0,T]} \big[ G_s^{-}( \tfrac{\hat{y}}{n}) - G_s^{-} ( \tfrac{\hat{x}}{n}) \big]^2 \Big\}  p(\hat{y}-\hat{x}) \lesssim \frac{1}{ n^{d-1} \log(n) },
\end{align*} 
that vanishes when $n \rightarrow \infty$. Above we applied \eqref{claimtight2} {for} ${H=G^{+}}$ {and for} ${H=G^{-}}$. Finally, due to \eqref{prob}, the sum over $\{\hat{x}, \hat{y} \} \in \mcb S$ in \eqref{tightcond2eq1} is bounded from above by
\begin{align*}
& \frac{\alpha c_{\gamma} \Theta(n)}{d  n^{2d+\beta} }  \| G \|_{\infty}   \sum_{ | \hat{x}_{\star}| \leq b_G n } \Big( \sum_{x_d=0}^{\infty} \; \sum_{y_d=-\infty}^{-1} (x_d-y_d)^{-1-\gamma}   + \sum_{x_d=-\infty}^{-1} \; \sum_{y_d=0}^{\infty}  (y_d-x_d)^{-1-\gamma} \Big) \\
\lesssim & \frac{\Theta(n)}{ n^{2d+\beta} }   \sum_{ | \hat{x}_{\star}| \leq b_G n } \; \sum_{z_d=1}^{\infty}  (z_d)^{-\gamma} \lesssim \frac{\Theta(n)}{ n^{2d+\beta} }   \sum_{ | \hat{x}_{\star}| \leq b_G n } \lesssim \frac{\Theta(n)}{ n^{d+1+\beta} } = \frac{\Theta(n)}{n^2} \frac{1}{ n^{d-1+\beta} } \lesssim \frac{1}{ n^{d-1+\beta} } \lesssim \frac{1}{ n^{\beta} },
\end{align*}
that vanishes when $n \rightarrow \infty$ due to $\beta >0$, and the proof ends. 
\end{proof}

\section{Characterization of limit points} \label{seccharac}
In this section we characterize the limit point $\mathbb Q$ of the sequence $(\mathbb{Q}_{n})_{ n \geq 1 }$, whose existence is a consequence  of  the results  of {the} last section. We first observe that since we deal with an exclusion process, according to the proof of Theorem 2.1, Chapter 4, of \cite{kipnis1998scaling}, $\mathbb{Q}$ is concentrated on trajectories $\pi_t(d \hat{u})$ which are absolutely continuous with respect to the Lebesgue measure, that is $\pi_t(d \hat{u})= \rho_t(\hat{u}) \, d \hat{u}$. In particular, $\rho_t(\hat{u}) \in [0,1]$ for every $(t,\hat{u}) \in [0,T] \times \mathbb{R}^d$.

Now we want to show that $\mathbb{Q}$ is concentrated on trajectories ${\rho}$ satisfying the corresponding integral equation of our hydrodynamic equations.
\begin{prop} \label{caraclimsementbnd}
\textit{
Every limit point $\mathbb{Q}$ of $(\mathbb{Q}_n)_{n \geq 1}$ satisfies}
\begin{align*}
\mathbb{Q} \Big(\pi_{ \cdot} \in \mcb D([0,T], \mcb{M}^+): \forall t \in [0,T], \forall G \in  \mcb S_{\alpha, \beta, \gamma}, \quad F_{\textrm{Dif}}(t,\rho,G, g, \kappa_{\gamma}) = 0  \Big) = 1.
\end{align*}
\textit{The space of test functions $\mcb S_{\alpha, \beta, \gamma}$ depends on whether \eqref{1entbound} holds or not. If \eqref{1entbound} \textit{does not hold}, $\mcb S_{\alpha, \beta, \gamma}$ is given by}
\begin{align*}
\mcb S_{\alpha, \beta, \gamma}:= 
\begin{cases}
\mcb{S}_{\textrm {Neu}}, & \quad \text{if} \quad  (\beta,\gamma) \in R_0 \quad \text{and} \quad \alpha > 0; \\
\mcb S_{\textrm{Dif}}, & \quad \text{if} \quad (\alpha,\beta) = (1,0)  \quad \text{and} \quad \gamma \geq 2.
\end{cases}
\end{align*}
\textit{On the other hand, if \eqref{1entbound} \textit{holds}, $\mcb S_{\alpha, \beta, \gamma}$ is given by}
\begin{align*}
\mcb S_{\alpha, \beta, \gamma}:= 
\begin{cases}
\mcb{S}_{\textrm {Neu}}, & \quad \text{if} \quad  (\beta,\gamma) \in R_0 \quad \text{and} \quad \alpha { > } 0; \\
\mcb S_{\textrm{Dif}}, & \quad \text{if} \quad \beta \in [0,1)  \quad \text{and} \quad (\alpha,\gamma) \in (0, \infty) \times [2, \infty).
\end{cases}
\end{align*}
\end{prop}
\begin{proof}
The proof ends as long as we show, for any $\delta_1 >0$ fixed and $G \in \mcb S_{\alpha, \beta, \gamma }$, that
\begin{align*} 
\mathbb{Q} \Big(\pi_{ \cdot} \in \mcb D([0,T], \mcb{M}^+): \sup_{t \in [0,T]} |F_{\textrm{Dif}}(t,\rho,G, g, \kappa_{\gamma})| > \delta_1  \Big) = 0.
\end{align*}
To simplify notation hereinafter we erase $\pi_{\cdot}$ from the sets where we look at. We can bound {the} last probability from above by the sum of the following terms:
\begin{align}
& \mathbb Q\Big(  \sup_{t \in [0,T]} \Big| \int_{\mathbb{R}^d}[ \rho_t(\hat{u}) G_t(\hat{u}) - \rho_0(\hat{u}) G_0(\hat{u})] \, d \hat{u}  
-  \int_0^t \int_{\mathbb{R}^d} \rho_s(\hat{u}) [ \kappa_{\gamma} {  \mathbb{L}_{\Delta} }+ \partial_s ] G_s(\hat{u})\, d \hat{u} \, ds  \Big| > \dfrac{\delta_1}{2} \Big), \label{f1term1dif} \\
+& \mathbb{Q} \Big(  \Big| \int_{\mathbb{R}^d} \left[ \rho_0(\hat{u}) - g(\hat{u}) \right] G_0(\hat{u}) \, d \hat{u} \Big| > \dfrac{\delta_1}{2}  \Big). \nonumber
\end{align}
{The last} probability is equal to zero, since {$\mathbb{Q}$ is a limit point of $(\mathbb{Q}_n)_{n \geq 1 }$, which is induced by $(\mathbb{P}_{\mu_n})_{n  \geq 1}$, $(\mu_n)_{n  \geq 1}$; and $(\mu_n)_{n  \geq 1}$ is a sequence associated to the profile $g$.}   Now from Portmanteau's Theorem, the probability in \eqref{f1term1dif} is bounded from above by
\begin{equation} \label{portman}
 \varliminf_{n \rightarrow \infty} \mathbb{Q}_{n} \Big(  \sup_{t \in [0,T]} \Big| \int_{\mathbb{R}^d}[ \rho_t(\hat{u}) G_t(\hat{u}) - \rho_0(\hat{u}) G_0(\hat{u})] \, d \hat{u}   -   \int_0^t \int_{\mathbb{R}^d} \rho_s(\hat{u}) [ \kappa_{\gamma}  {  \mathbb{L}_{\Delta} } + \partial_s ] G_s(\hat{u}) \, d\hat{u} \, ds  \Big| > \dfrac{\delta_1}{2} \Big).
\end{equation}
Next we treat two cases: $(\beta,\gamma) \in R_0$ and $\beta \in [0,1)$.

\textbf{I.:} Case $(\beta,\gamma) \in R_0$. Combining the definition of $\mcb M_{t}^{n}(G)$ in \eqref{defMnt} with \eqref{extranb1} and \eqref{extranb2}, the display \eqref{portman} is bounded from above by the sum of the next terms
\begin{align}
& \varliminf_{n \rightarrow \infty} \mathbb{P}_{\mu_n}  \Big(   \sup_{t \in [0,T]} |\mcb M_{t}^{n}(G)| > \dfrac{\delta_1}{6} \Big) \label{f1term1bdif}\\
+&  \varliminf_{n \rightarrow \infty}  \mathbb{P}_{\mu_n} \Big(   \sup_{t \in [0,T]}  \Big|  \int_{0}^{t} \frac{1}{n^d} \sum_{ \hat{x} } [ \Theta(n)   \mcb{K}_{n}^{*} G_s( \tfrac{\hat{x}}{n}) -  \kappa_{\gamma} {  \mathbb{L}_{\Delta} } G_s( \tfrac{\hat{x}}{n})] \eta_s^n(\hat{x})  \, ds  \Big| > \dfrac{\delta_1}{6} \Big) \label{f1term1cdif} \\
+& \varliminf_{n \rightarrow \infty}  \mathbb{P}_{\mu_n} \Big(   \sup_{t \in [0,T]}  \Big|     \int_{0}^{t}  \sum _{ \hat{x} } \frac{\Theta(n)}{n^{d}}  \Big[  \frac{\alpha}{n^{\beta}}   \mcb{K}_{n, \mcb S} G_s( \tfrac{ \hat{x} }{n} ) + \mcb{R}_{n,d} G_s\left( \tfrac{ \hat{x} }{n} \right) \Big]  \eta_{s}^{n}(\hat{x}) \, ds    \Big| > \dfrac{\delta_1}{6} \Big). \label{f1term1ddif}
\end{align}
Since $(\beta,\gamma) \in R_0$, then $G \in \mcb{S}_{\textrm {Neu}}$ and we can apply Proposition \ref{convprincdisc}. Hence, from \eqref{condger2pr3} (resp. Proposition \ref{convprincdisc}) the limit in \eqref{f1term1bdif} (resp. \eqref{f1term1cdif})  is equal to zero. It remains to treat  \eqref{f1term1ddif}. From the fact that $G \in \mcb{S}_{\textrm {Neu}}$ and \eqref{op_Rn} we get $\mcb{R}_{n,d} G_s \left( \tfrac{ \hat{x} }{n} \right) = 0$ for every $\hat{x} \in \mathbb{Z}^d$ and every $s \in [0,T]$. Combining this with \eqref{claim1largebeta}, the limit in \eqref{f1term1ddif} is also equal to zero.

\textbf{II.:} Case $\beta \in [0,1)$. Combining the definition of $\mcb M_{t}^{n}(G)$ in \eqref{defMnt} with \eqref{gensmallbeta}, the display in \eqref{portman} is bounded from above by the sum of the next terms
\begin{align}
& \varliminf_{n \rightarrow \infty} \mathbb{P}_{\mu_n}  \Big(   \sup_{t \in [0,T]} |\mcb M_{t}^{n}(G)| > \dfrac{\delta_1}{6} \Big) \label{f1term1edif} \\
+& \varliminf_{n \rightarrow \infty}  \mathbb{P}_{\mu_n} \Big(   \sup_{t \in [0,T]}  \Big|  \int_{0}^{t} \frac{1}{n^d} \sum_{ \hat{x} } [ \Theta(n)   \mcb K_{n, \mcb B} G_s( \tfrac{\hat{x}}{n}) -  \kappa_{\gamma} {  \mathbb{L}_{\Delta} } G_s( \tfrac{\hat{x}}{n})] \eta_s^n(\hat{x}) \, ds  \Big| > \dfrac{\delta_1}{6} \Big) \label{f1term1fdif} \\
+&  \varliminf_{n \rightarrow \infty}  \mathbb{P}_{\mu_n} \Big(   \sup_{t \in [0,T]}  \Big| (1 - \alpha n^{-\beta}) \int_{0}^{t} \frac{\Theta(n)}{n^d} \sum_{ \hat{x} }     \mcb K_{n, \mcb S} G_s( \tfrac{\hat{x}}{n})  \eta_{s}^{n}(\hat{x}) \, ds  \Big| > \dfrac{\delta_1}{6} \Big). \label{f1term1gdif}
\end{align}
Since $(\beta,\gamma) \notin R_0$, then $G \in \mcb S_{\textrm{Dif}}$, ${\mathbb{L}_{\Delta} G = \Delta G}$ and {we} can apply Proposition \ref{convprincdif}. Hence, from \eqref{condger2pr3} (resp. Proposition \ref{convprincdif}) the limit in \eqref{f1term1edif} (resp. \eqref{f1term1fdif})  is equal to zero. It remains to treat  \eqref{f1term1gdif}. If $(\alpha,\beta)=(1,0)$, the expression inside the absolute value is equal to zero and we are done. Otherwise \eqref{1entbound} holds, and applying Proposition \ref{replemsmallbeta} we conclude that \eqref{f1term1gdif} is equal to zero, ending the proof.
\end{proof}
In the particular case $\beta>1$, $\gamma > 2$ and $d=1$, the integral equation in \eqref{ehdnlee} can be obtained when \eqref{1entbound} is assumed. Next result can be proved exactly in the same way as Proposition 4.2 of \cite{casodif}, therefore we omit its proof; it is obtained by combining \eqref{extranb2} and \eqref{claim1largebeta} with Propositions \ref{convprincdisc} and \ref{replemlargebeta}. 
\begin{prop} \label{caraclimsementbnd2}
\textit{
Assume $\eqref{1entbound}$, $\beta>1$, $\gamma > 2$ and $d=1$. Then every limit point $\mathbb{Q}$ of $(\mathbb{Q}_n)_{n \geq 1}$ satisfies}
\begin{align*}
\mathbb{Q} \Big(\pi_{ \cdot} \in \mcb D([0,T], \mcb{M}^+):  \forall t \in [0,T], \forall G \in \mcb{S}_{\textrm {Disc}}, \quad F_{\textrm{Neu}}(t,\rho,G, g, \kappa_{\gamma}) = 0 \Big) = 1.
\end{align*}
\end{prop}

\section{Useful $L^1(\mathbb{P}_{\mu_n})$ estimates} \label{secheurwithout}

The goal of this section is to prove Propositions \ref{replemsmallbeta} and \ref{replemlargebeta}. We begin with the former one.
\begin{proof}[Proof of Proposition \ref{replemsmallbeta}]
We begin by rewriting $\int_{0}^{t} \sum _{\hat{x} }  \frac{\Theta(n)}{n^d} \mcb {K}_{n,\mcb S} G_s \left( \tfrac{ \hat{x} }{n} \right)  \eta_{s}^{n}(\hat{x}) \, ds$ as $\sum_{j=1}^3 \mcb A_{n,j}(G)$, where 
\begin{align}
& \mcb A_{n,1}(G) := \int_{0}^{t} \frac{\Theta(n)}{n^d}  \sum_{| \hat{x}_{\star}| < b_G n } \Big\{ \sum_{x_d=0}^{n-1} \; \sum_{y_d=-\infty}^{-n}  A^{n,\hat{x}_{\star}}_{x_d, y_d}(G_s)
+    \sum_{x_d=-n+1}^{-1} \; \sum_{y_d=n}^{\infty} A^{n,\hat{x}_{\star}}_{x_d, y_d}(G_s) \Big\} \eta_s^n(\hat{x}_{\star}, x_d) \, ds,  \nonumber \\
& \mcb A_{n,2}(G) := \int_{0}^{t} \frac{\Theta(n)}{n^d}  \sum_{| \hat{x}_{\star}| < b_G n } \Big\{ \sum_{x_d=n}^{\infty} \; \sum_{y_d=-\infty}^{-1}  A^{n,\hat{x}_{\star}}_{x_d, y_d}(G_s)
+    \sum_{x_d=-\infty}^{-n} \; \sum_{y_d=0}^{\infty} A^{n,\hat{x}_{\star}}_{x_d, y_d}(G_s) \Big\} \eta_s^n(\hat{x}_{\star},  x_d) \, ds, \nonumber  \\
& \mcb A_{n,3}(G) := \int_{0}^{t} \frac{\Theta(n)}{n^d}  \sum_{| \hat{x}_{\star}| < b_G n } \Big\{ \sum_{x_d=0}^{n-1} \; \sum_{y_d=-n+1}^{-1} A^{n,\hat{x}_{\star}}_{x_d, y_d}(G_s)  
+    \sum_{x_d=-n+1}^{-1} \; \sum_{y_d=0}^{n-1} A^{n,\hat{x}_{\star}}_{x_d, y_d}(G_s) \Big\} \eta_s^n(\hat{x}_{\star}, x_d)\, ds. \label{extranear0}
\end{align}
Above we used \eqref{eqivmcbSG} and the definition of $A^{n,\hat{x}_{\star}}_{x_d, y_d}(G_s)$ given in \eqref{defAnG}. Next we claim that $ \mathbb{E}_{\mu_n} [  | \mcb A_{n,1}(G) + \mcb A_{n,2}(G) | ]$ goes to zero as $n \rightarrow \infty$. Indeed, combining \eqref{exc1part} and \eqref{defAnG}, it holds
\begin{align*}
|\mcb A_{n,1}(G)| + |\mcb A_{n,2}(G)| \leq&  \int_{0}^{T} \frac{8\Theta(n) \| G \|_{\infty} c_{\gamma} }{d n^d} \sum_{| \hat{x}_{\star}| < b_G n } \; \sum_{x_d=n}^{\infty} \; \sum_{y_d=0}^{\infty} (x_d+y_d)^{-1-\gamma} \, ds \\
 \lesssim & \frac{\Theta(n) }{n^{\gamma}} \frac{1}{n^2} \sum_{x_d=n}^{\infty} \;  \sum_{y_d=0}^{\infty}  \Big( \frac{x_d+y_d}{n} \Big)^{-1-\gamma} \lesssim \frac{\Theta(n) }{n^{\gamma}},
\end{align*}
which vanishes as $n \rightarrow \infty$, due to \eqref{timescale}. Above we used the fact that $\gamma \geq 2 >1$. It remains to treat $\mcb A_{n,3}(G)$. Exchanging the variables $x_d$ and $y_d$, and from the symmetry of $p(\cdot)$, \eqref{extranear0} can be rewritten as
\begin{align*}
-\int_0^t \frac{\Theta(n)c_{\gamma}}{dn^{d}} \sum_{ | \hat{x}_{\star} | < b_G n } \; \sum_{x_d=0}^{ n -1} \; \sum_{y_d=- n +1}^{-1} \big[ G_s( \tfrac{ \hat{x}_{\star} }{n}, \tfrac{y_d}{n} ) - G_s( \tfrac{ \hat{x}_{\star} }{n}, \tfrac{x_d}{n} ) \big] (x_d - y_d)^{-\gamma-1}  [\eta_s^n(\hat{x}_{\star}, y_d) - \eta_s^n( \hat{x}_{\star}, x_d ) ] \, ds.
\end{align*}
Since $G \in \mcb S_{\textrm{Dif}}$, we can perform a second order Taylor expansion on $G_s( \tfrac{ \hat{x}_{\star} }{n},\cdot)$ around $x_d/n$ and afterwards a first order Taylor expansion on $\partial_{\hat{e}_d} G_s(\tfrac{ \hat{x}_{\star} }{n},\cdot)$ around $0$, and both of the expansions being along the direction determined by $\hat{e}_d$. In this way, {the} last display can be rewritten as $c_{\gamma} \sum_{j=4}^6 \mcb A_{n,j}(G)$, where
\begin{align*}
& \mcb A_{n,4}(G) := - \int_0^t \frac{\Theta(n) }{dn^{d}} \sum_{ | \hat{x}_{\star} | < b_G n } \; \sum_{x_d=0}^{ n -1} \; \sum_{y_d=- n +1}^{-1}  \frac{(x_d-y_d)^{1-\gamma}}{2 n^2} \partial_{\hat{e}_d \hat{e}_d} G_s( \tfrac{ \hat{x}_{\star} }{n}, \xi^n_{x_d,y_d} )   \big[\eta_s^n(\hat{x}_{\star}, y_d) - \eta_s^n( \hat{x}_{\star}, x_d ) \big] \, ds \nonumber \\
& \mcb A_{n,5}(G) := \int_0^t \frac{\Theta(n)}{dn^{d}} \sum_{ | \hat{x}_{\star} | < b_G n } \; \sum_{x_d=0}^{ n -1} \; \sum_{y_d=- n +1}^{-1}  \frac{x_d(x_d-y_d)^{-\gamma}}{ n^2}  \partial_{\hat{e}_d \hat{e}_d} G_s( \tfrac{ \hat{x}_{\star} }{n}, \tilde{\xi}^n_{x_d} )   \big[\eta_s^n(\hat{x}_{\star}, y_d) - \eta_s^n( \hat{x}_{\star}, x_d ) \big] \, ds \nonumber \\
& \mcb A_{n,6}(G) := \int_0^t \frac{\Theta(n) }{dn^{d+1}} \sum_{ | \hat{x}_{\star} | < b_G n } \partial_{\hat{e}_d} G_s( \tfrac{ \hat{x}_{\star} }{n}, 0 )   \sum_{x_d=0}^{ n -1} \; \sum_{y_d=- n +1}^{-1}  (x_d - y_d)^{-\gamma}  \big[\eta_s^n(\hat{x}_{\star}, y_d) - \eta_s^n( \hat{x}_{\star}, x_d ) \big] \, ds, \nonumber 
\end{align*}
for some $\xi^n_{x_d,y_d} \in (\tfrac{y_d}{n}, \tfrac{x_d}{n})$ and some $\tilde{\xi}^n_{x_d} \in (0, \tfrac{x_d}{n})$. We claim that $\mathbb{E}_{\mu_n} [  | \mcb A_{n,4}(G) + \mcb A_{n,5}(G) | ]$ goes to zero as $n \rightarrow \infty$. Indeed, due to \eqref{exc1part}, $|\mcb A_{n,4}(G)|+|\mcb A_{n,5}(G)|$ is bounded from above by
\begin{align*}
\int_0^T \frac{\Theta(n) }{dn^{d+2}} \sum_{ | \hat{x}_{\star} | < b_G n } \; \sum_{x_d=0}^{ n -1} \; \sum_{y_d=- n +1}^{-1} \| \partial_{\hat{e}_d \hat{e}_d} G \|_{\infty} (x_d - y_d)^{1-\gamma} \, ds \lesssim  \frac{\Theta(n)}{n^{3}} \sum_{z_d=0}^{2n -1}  (z_d)^{2-\gamma}.  
\end{align*}
If $\gamma > 3$ the sum over $z_d$ is of order $1$ and {the} last display is of order $\Theta(n) n^{-3} = n^{-1}$. If $\gamma = 3$ the sum over $z_d$ is of order $\log(n)$ and {the} last display is of order $\log(n) / n$. If $\gamma \in (2,3)$ the sum over $z_d$ is of order $n^{3-\gamma}$ and {the} last display is of order $n^{2-\gamma}$. Finally, if $\gamma \in (2,3)$ the sum over $z_d$ is of order $n$ and {the} last display is of order $1 / \log(n)$. Hence, {the} last display goes to zero as $n \rightarrow \infty$ and $\lim_{n \rightarrow \infty} \mathbb{E}_{\mu_n} [  | \mcb A_{n,4}(G) + \mcb A_{n,5}(G) | ]=0$. 

It remains to treat $\mcb A_{n,6}(G)$. It is bounded from above by
\begin{align*}
\int_0^t \frac{\Theta(n)}{dn^{d+1}} \sum_{ | \hat{x}_{\star} | < b_G n } \| \partial_{\hat{e}_d} G \|_{\infty}   \sum_{x_d=0}^{ n -1} \; \sum_{y_d=- n +1}^{-1} p (x_d - y_d)^{-\gamma}  \big[\eta_s^n(\hat{x}_{\star}, y_d) - \eta_s^n( \hat{x}_{\star}, x_d ) \big] \, ds.
\end{align*}
Therefore, the claim that $\lim_{n \rightarrow \infty} \mathbb{E}_{\mu_n} [  | \mcb A_{n,6}(G)| ]=0$ is a direct consequence of Lemmas \ref{replemma2smallbet} and \ref{replemma1smallbet}, which are stated and proved below. Lemma \ref{replemma2smallbet} deals with exchanges of particles through \textit{fast} bonds, therefore there are no restrictions on the value of $\beta$. On the other hand, Lemma \ref{replemma1smallbet} makes use of exchanges of particles through \textit{slow} bonds, hence a restriction on the value of $\beta$ is necessary.

Heuristically, from Lemma \ref{replemma2smallbet} the occupation variable $\eta_s^n(\hat{x}_{\star}, x_d)$ (resp., $\eta_s^n(\hat{x}_{\star}, y_d)$) can be replaced by $\eta_s^n(\hat{x}_{\star}, 0)$ when $0 \leq x_d < n$ (resp., can be replaced by $\eta_s^n(\hat{x}_{\star}, -1)$ when $-n < y_d \leq -1$). Similarly, from Lemma \ref{replemma1smallbet} the occupation variable $\eta_s^n(\hat{x}_{\star}, 0)$ can be replaced by $\eta_s^n(\hat{x}_{\star}, -1)$ when $ \beta \in [0,1)$.
\end{proof}
\begin{lem}   \label{replemma2smallbet}
\textit{
Assume \eqref{1entbound}. For $t \in [0,T]$ and $C>0$, it holds 
\begin{align}
& \lim_{n \rightarrow \infty} \mathbb{E}_{\mu_n} \Big[  \Big| \int_0^t \frac{\Theta(n)}{n^{d+1}} \sum_{ | \hat{x}_{\star} | < C n }  \; \sum_{x_d=0}^{ n -1} \; \sum_{y_d=- n +1}^{-1}  (x_d-y_d)^{-\gamma}  \big[\eta_s^n(\hat{x}_{\star}, x_d) - \eta_s^n( \hat{x}_{\star}, 0 ) \big] \, ds \Big|  \Big] = 0,  \label{rlright1} \\
& \lim_{n \rightarrow \infty} \mathbb{E}_{\mu_n} \Big[  \Big| \int_0^t \frac{\Theta(n)}{n^{d+1}} \sum_{ | \hat{x}_{\star} | < C n }  \; \sum_{x_d=0}^{ n -1} \; \sum_{y_d=- n +1}^{-1}  (x_d-y_d)^{-\gamma}  \big[\eta_s^n(\hat{x}_{\star}, -1) - \eta_s^n( \hat{x}_{\star}, y_d ) \big] \, ds \Big|  \Big] = 0. \label{rlleft1}
\end{align}
}
\end{lem}
\begin{lem}   \label{replemma1smallbet}
\textit{
Assume \eqref{1entbound}. For $t \in [0,T]$ and $C>0$, it holds
\begin{equation}  \label{rlleft10}
\lim_{n \rightarrow \infty} \mathbb{E}_{\mu_n} \Big[  \Big| \int_0^t \frac{\Theta(n)}{n^{d+1}} \sum_{ | \hat{x}_{\star} | < C n }  \; \sum_{x_d=0}^{ n -1} \; \sum_{y_d=- n +1}^{-1}  (x_d-y_d)^{-\gamma}  \big[\eta_s^n(\hat{x}_{\star}, 0) - \eta_s^n( \hat{x}_{\star}, -1 ) \big] \, ds \Big|  \Big] = 0.
\end{equation}
}
\end{lem}
Recall the definition of a element of $\textrm{Ref}$ (resp. of $\nu_h^n$) in Definition \ref{defrefprof} (resp. Definition \ref{defberprod}). In order to show Lemmas \ref{replemma2smallbet} and \ref{replemma1smallbet}, we will use the \textit{Dirichlet form}, defined by $\langle \sqrt{f},- \mcb L_n \sqrt{f} \rangle_{\nu_h^n}$, where $f: \Omega \rightarrow \mathbb{R}$ is a density with respect to $\nu_h^n$. For functions $g_1, g_2: \Omega \rightarrow \mathbb{R}$, $\langle g_1, g_2 \rangle_{\nu_h^n}$ denotes the scalar product in $L^{2} (\Omega, \nu_h^n )$. For every $ \hat{x}, \hat{y} \in \mathbb{Z}^d$ we define
\begin{align*} 
I_{\hat{x}, \hat{y}}  (\sqrt{f}, \nu_h^n ) := \int_{\Omega}    \big[ \sqrt{f \left( \eta^{\hat{x}, \hat{y}} \right) } - \sqrt{f \left( \eta \right) } \big]^2 \, d \nu_h^n.
\end{align*}
{The next} result is analogous to Lemma 5.5 of \cite{byrondif}, therefore we omit its proof. 
\begin{prop}
\textit{
For every} $h \in \textrm{Ref}$ \textit{and every density $f$ with respect to $\nu_h^n$, there exists a constant $E_h>0$ depending only on $h$ such that for every $\hat{x}, \hat{y} \in \mathbb{Z}^d$ and every $A_{\hat{x}, \hat{y}}>0$,
\begin{align} \label{young}
\Big| \int [\eta(\hat{x}) - \eta(\hat{y})] f(\eta) d \nu_h^n \Big| \leq E_h \Big[  \frac{ I_{\hat{x}, \hat{y}}  (\sqrt{f}, \nu_h^n )}{A_{\hat{x}, \hat{y}}} +  A_{\hat{x},\hat{y}} + \frac{|\hat{x}-\hat{y}|}{n} \Big].
\end{align}
}
\end{prop}
Next we define $D_n (\sqrt{f}, \nu_h^n ) := D^{\mcb F} (\sqrt{f}, \nu_h^n ) + \alpha n^{-\beta}  D^{\mcb S} (\sqrt{f}, \nu_h^n )$,
where for $\mcb K \in \{\mcb F, \mcb S\}$, 
\begin{align*}
D^{\mcb K}  (\sqrt{f}, \nu_h^n) := \frac{1}{2} \sum_{\{ \hat{x}, y \} \in \mcb K} p(\hat{y}-\hat{x}) I_{\hat{x},\hat{y}}  (\sqrt{f}, \nu_h^n ). 
\end{align*}
Next we state a result which is useful to estimate the Dirichlet form; we postpone its proof to Appendix \ref{estdirfor}.
\begin{prop} \label{bound}
\textit{
For every} $h \in \textrm{Ref}$ \textit{and every density $f$ with respect to $\nu_h^n$, there exists a constant $M_h>0$ depending only on $h$ such that for every $n \geq 1$,
\begin{align} \label{boundlip}
 \frac{\Theta(n)}{n^d} \langle \mcb{L}_{n} \sqrt{f} , \sqrt{f} \rangle_{\nu_h^n} \leq - \frac{\Theta(n)}{4n^d}  D_n (\sqrt{f}, \nu_h^n ) + M_h.
\end{align}
}
\end{prop}
Due to \eqref{1entbound}, we fix here and in what follows a constant $C_h > 0$ such that
\begin{equation} \label{defCh}
\forall n \geq 1, \quad H( \mu_n | \nu_h^n) \leq  C_h n^d.
\end{equation}
{The next} result will be useful in what follows.
\begin{lem} \label{entjen}
\textit{
Assume \eqref{1entbound}. For every $n \geq 1$, let $(\tilde{Z}_{n,j})_{j \geq 1}$ be a sequence of arbitrary random applications $\mcb D ( [0,T], \Omega ) \rightarrow \mathbb{R}$. For every $j \geq 1$, define $Z_{n,j}$ by $Z_{n,j}: = \max_{1 \leq i \leq j} \{ \tilde{Z}_{n,i} \}$. Then 
\begin{align*}
\forall n, j \geq 1, \, \, \forall B >0, \quad \quad \mathbb{E}_{\mu_n} [ Z_{n,j} ] \leq \frac{C_h}{B} +  \frac{j}{Bn^d} + \frac{1}{Bn^d} \log \Big( \max_{1 \leq i \leq j} \mathbb{E}_{\nu_h^n} \big[ \exp  (Bn^d \tilde{Z}_{n,i})  \big]  \Big).
\end{align*}
}
\end{lem}
\begin{proof}
For every $n \geq 1$, from the {fact} that $\mathbb{P}_{\mu_n}$ is induced by $\mu_n$ and Fubini's Theorem, we get
\begin{align*}
 \mathbb{E}_{\mu_n} [ Z_{n,j} ] = \int_{ \mcb D ( [0,T], \Omega )}  Z_{n,j} \, d \mathbb{P}_{\mu_n} = \int_{\Omega} \Big( \int Z_{n,j} d \mathbb{P}_{\eta_{\cdot}} \Big) \, d \mu_n = \frac{1}{B n^d}  \int_{\Omega} \Big( \int B n^d   Z_{n,j} d \mathbb{P}_{\eta_{\cdot}} \Big) \, d \mu_n.
\end{align*}
Applying the entropy's inequality, {the} last display is bounded from above by
\begin{equation} \label{fubstat2}
	\begin{split}
		& \frac{H(\mu_n | \nu_h^n)}{B n^d}  + \frac{1}{B n^d} \log \Big( \int_{\Omega} \exp \Big\{B n^d \int    Z_{n,j}  d \mathbb{P}_{\eta_{\cdot}} \Big\} \, d \nu_h^n \Big). 
	\end{split}
\end{equation}
Since for every $n \geq 1$ $\lambda \mapsto e^{B n^d \lambda}$ is a convex function, from Jensen's inequality, we conclude that
\begin{align*}
	\exp \Big\{B n^d \int    Z_{n,j} \, d \mathbb{P}_{\eta_{\cdot}} \Big\} \leq \int e^{B n^d Z_{n,j}} \, d \mathbb{P}_{\eta_{\cdot}}.
\end{align*}
Combining this with Fubini's Theorem, \eqref{fubstat2} is bounded from above by
\begin{align*}
	& \frac{H(\mu_n | \nu_h^n)}{B n^d}  + \frac{1}{B n^d} \log \Big\{ \int_{\Omega} \Big( \int  e^{B n^d Z_{n,j}   } \, d \mathbb{P}_{\eta_{\cdot}} \Big) \, d \nu_h^n \Big\} = \frac{C_h}{B}  + \frac{1}{B n^d} \log \big(  \mathbb{E}_{\nu_h^n} \big[ e^{B n^d Z_{n,j}}  \big]  \big).
\end{align*}
In the equality above we used \eqref{1entbound}. Recall the definition of $Z_{n,j}$. Then the second term in {the} last display is bounded from above by
\begin{align*}
	& \frac{1}{B n^d} \log \Big( \mathbb{E}_{\nu_h^n} \big[ \max_{1 \leq i \leq j} \exp (B n^d \tilde{Z}_{n,i} ) \big] \Big) \leq  \frac{1}{B n^d}\log \Big( \mathbb{E}_{\nu_h^n} \Big[ \sum_{1 \leq i \leq j} \exp ( B n^d \tilde{Z}_{n,i} ) \Big] \Big) \\
	\leq& \frac{1}{B n^d} \log \Big( j  \max_{1 \leq i \leq j} \mathbb{E}_{\nu_h^n} [  \exp (  n^d \tilde{Z}_{n,i} ) ] \Big) = \frac{j}{B n^d} + \frac{1}{B n^d} \log \Big(   \max_{1 \leq i \leq j} \mathbb{E}_{\nu_h^n} [  \exp ( B n^d \tilde{Z}_{n,i} ) ] \Big),
\end{align*}
and the proof ends.
\end{proof}
\begin{proof}[Proof of Lemma \ref{replemma2smallbet}]
We present here only the proof of \eqref{rlright1}, but we observe that the proof of \eqref{rlleft1} is analogous. Recall the definition of $\Theta(n)$ in \eqref{timescale}. For  every $\varepsilon >0$, define $\ell = \ell(\varepsilon,n)$ by
\begin{equation} \label{defellmic}
\ell { (\varepsilon,n) : }= \frac{\varepsilon \Theta(n)}{n}=
\begin{cases}
\varepsilon n, \quad & \gamma > 2; \\
\varepsilon n / \log(n), \quad & \gamma = 2.
\end{cases}.
\end{equation}
Next we claim that
\begin{equation}  \label{claim1smabet}
\varlimsup_{\varepsilon \rightarrow 0^+} \varlimsup_{n \rightarrow \infty} \mathbb{E}_{\mu_n} \Big[  \Big| \int_0^t \Theta(n) \sum_{ | \hat{x}_{\star} | < C n } \;    \sum_{x_d=\ell}^{ n -1} \;  \sum_{y_d=- n +1}^{-1} \frac{(x_d-y_d)^{-\gamma}}{n^{d+1}}    \big[\eta_s^n(\hat{x}_{\star}, x_d) - \eta_s^n( \hat{x}_{\star}, 0 ) \big] \, ds \Big|  \Big] = 0.
\end{equation}
Indeed, due to \eqref{exc1part}, {the inside of the expectation in \eqref{claim1smabet}} is bounded from above by 
\begin{align*}
&   \Big| \int_0^t \frac{\Theta(n)}{n^{d+1}} \sum_{ | \hat{x}_{\star} | < C n }   \; \sum_{x_d=\ell}^{ n -1} \;  \sum_{y_d=- n +1}^{-1}  (x_d-y_d)^{-\gamma}  \big[\eta_s^n(\hat{x}_{\star}, x_d) - \eta_s^n( \hat{x}_{\star}, 0 ) \big] \, ds \Big|   \\
\leq &  T  (2 C n)^{d-1}  \frac{\Theta(n)}{n^{d+1}} \sum_{x_d= \ell }^{ n -1} \;  \sum_{y=- n +1}^{-1}  (x_d-y_d)^{-\gamma}    \lesssim \frac{\Theta(n)}{n^2} \sum_{x_d= \ell }^{ { 2 }n -1} x_{{d}}^{1-\gamma},
\end{align*}
which vanishes as $n \rightarrow \infty$, for every $\varepsilon >0$ fixed. Indeed, if  $\gamma > 2$, the last expression is equal to
\begin{align*}
\sum_{x_d= \varepsilon n }^{{2} n -1} x_{{ d } }^{1-\gamma} \leq (\varepsilon n)^{1-\gamma} + \int_{\varepsilon n}^{{2} n} u^{1-\gamma} \, du \lesssim (\varepsilon n)^{2-\gamma}, 
\end{align*}
that goes to zero as $n \rightarrow \infty$, for any $\varepsilon >0$; while for $\gamma=2$, it is equal to
\begin{align*}
\frac{1}{\log(n)} \sum_{x_d= \varepsilon n / \log(n)   }^{{ 2 }n -1} x_{{ d}}^{-1} \leq \frac{1}{\log(n)} \Big[  \Big(\frac{\varepsilon n}{\log(n)} \Big)^{-1} + \int_{\varepsilon n / \log(n)}^{{2}n} \, u^{-1} du \Big] \lesssim \frac{1}{\varepsilon n} +\frac{1}{\log(n)} \log \Big( \frac{ { 2 }\log(n)}{\varepsilon} \Big), 
\end{align*}
and it also goes to zero as $n \rightarrow \infty$, for every $\varepsilon >0$. 

 It remains to prove that
	\begin{equation}  \label{exprpl1} 
		\varlimsup_{\varepsilon \rightarrow 0^+} \varlimsup_{n \rightarrow \infty} \mathbb{E}_{\mu_n} \Big[  \Big| \int_0^t \frac{\Theta(n)}{n^{d+1}} \sum_{ | \hat{x}_{\star} | < C n } \;    \sum_{x_d={0}}^{ \ell -1} \;  \sum_{y_d=- n +1}^{-1}  (x_d-y_d)^{-\gamma}  \big[\eta_s^n(\hat{x}_{\star}, x_d) - \eta_s^n( \hat{x}_{\star}, 0 ) \big] \, ds \Big|  \Big] = 0.
	\end{equation}
Defining the random application $Z_{n}^{\varepsilon}: \mcb D ( [0,T], \Omega ) \rightarrow \mathbb{R}$ by
	\begin{align*}
		Z_{n}^{\varepsilon}:= \int_0^t \frac{\Theta(n)}{n^{d+1}} \sum_{ | \hat{x}_{\star} | < C n } \;    \sum_{x_d={0}}^{ \ell -1} \;  \sum_{y_d=- n +1}^{-1}  (x_d-y_d)^{-\gamma}  \big[\eta_s^n(\hat{x}_{\star}, x_d) - \eta_s^n( \hat{x}_{\star}, 0 ) \big] \, ds,
	\end{align*}
the expectation in \eqref{exprpl1} can be rewritten as
\begin{align} \label{defZneps}
\mathbb{E}_{\mu_n} [ | Z_{n}^{\varepsilon} | ] = \mathbb{E}_{\mu_n} [ \max \{ Z_{n}^{\varepsilon},- Z_{n}^{\varepsilon} \} ]. 
\end{align}
From Lemma \ref{entjen}, the display in \eqref{exprpl1} is bounded from above by
\begin{align*}
\varlimsup_{\varepsilon \rightarrow 0^+} \varlimsup_{n \rightarrow \infty}  \Big\{ \frac{C_h}{B} +  \frac{2}{B n^d} + \frac{1}{B n^d} \log \Big( \max \big\{  \mathbb{E}_{\nu_h^n} \big[ \exp  (B n^d Z_{n}^{\varepsilon}), \; \mathbb{E}_{\nu_h^n} \big[ \exp  (-{B} n^d Z_{n}^{\varepsilon})   \big] \Big) \Big\},  
\end{align*}
for every $B >0$. Applying Feynman-Kac's formula (see Lemma A.1 of \cite{baldasso}) and \eqref{defZneps}, the expression in {the} last double limit is bounded from above by
	\begin{equation} \label{suprpl1}
		\begin{split}
			\frac{C_h+2}{B} + T \sup_{f} \Big\{&   \frac{\Theta(n)}{n^{d+1}} \sum_{ | \hat{x}_{\star} | < C n } \; \sum_{x_d={0}}^{\ell -1} \; \sum_{y_d=-  n +1}^{-1}   (x_d-y_d)^{-\gamma}  | \langle \eta(\hat{x}_{\star}, x_d) - \eta(\hat{x}_{\star}, 0)  , f \rangle_{\nu_{h}^n} | \\
			+ & \frac{\Theta(n)}{B n^d} \langle \mcb L_n \sqrt{f}, \sqrt{f} \rangle_{\nu_{h}^n}   \Big\}, 
		\end{split}
	\end{equation}
for every $B > 0$, where the supremum above is carried over all the densities $f$ with respect to $\nu_{h}^n$. From Proposition \ref{bound}, the second term inside the supremum in {the} last display is bounded from above by
\begin{align} \label{rpl1a}
 \frac{M_h}{B} - \frac{\Theta(n)}{8Bn^d} \sum_{ | \hat{x}_{\star} | \leq C n }  \;  \sum_{w=0}^{\ell-1}  I_{( (\hat{x}_{\star}, w+1)), ((\hat{x}_{\star}, w)) }  (\sqrt{f}, \nu_{h}^n ) p( \hat{e}_d ) ,  
\end{align}
where $M_h$ is a positive constant depending only on $h \in \textrm{Ref}$. Observe that
\begin{align*}
  | \langle \eta(\hat{x}_{\star}, x_d) - \eta(\hat{x}_{\star}, 0)  , f \rangle_{\nu_{h}^n} | \leq &  \sum_{w=0}^{x_{{d}}-1} \Big| \int_{\Omega} [ \eta(\hat{x}_{\star},w+1) - \eta(\hat{x}_{\star}, w) ] f (\eta) \, d \nu_{h}^n \Big| \\
   \leq &   \sum_{w=0}^{\ell-1} \Big| \int_{\Omega} [ \eta(\hat{x}_{\star}, w+1) - \eta(\hat{x}_{\star}, w) ] f (\eta) \, d \nu_{h}^n \Big|.
\end{align*}
Above we used {the fact} that $0 \leq x_{{d}} \leq \ell-1$. Then the first term inside the supremum in \eqref{suprpl1} is bounded from above by
	\begin{align} \label{corclaim} 
		\frac{\Theta(n)}{n^{d+1}} \sum_{ | \hat{x}_{\star} | < C n }\;\sum_{x_d={0}}^{\ell -1} \; \sum_{y_d=-  n +1}^{-1}   (x_d-y_d)^{-\gamma} \sum_{w=0}^{\ell-1} \Big| \int_{\Omega} [ \eta(\hat{x}_{\star}, w+1) - \eta(\hat{x}_{\star}, w) ] f (\eta) \, d \nu_{h}^n \Big|. 
	\end{align}
Before proceeding, we claim that for every $K, C>0$ fixed, we have
\begin{align} \label{convkappagamma}
 \lim_{n \rightarrow \infty} \frac{\Theta(n)}{ n^2 } \sum_{r=1}^{ K n } c_{\gamma} r^{1-\gamma}  = d \kappa_{\gamma} = \lim_{n \rightarrow \infty} \frac{\Theta(n)}{ n^2 } \sum_{r=1}^{C n - 1 } c_{\gamma} r^{1-\gamma}.
\end{align}
In particular, we get
\begin{align} \label{thetan}
\frac{\Theta(n)}{n^2}  \sum_{x_d=0}^{K n  -1} \; \sum_{y_d=- K n +1}^{-1}   (x_d-y_d)^{-\gamma}  \leq  \frac{\Theta(n)}{n^2} \sum_{r=1}^{ 2K n} r^{1-\gamma}  \lesssim 1.
\end{align}
Indeed, assume that either $a_n = Kn$ for every $n \geq 1$ or $a_n = Kn-1$ for every $n \geq 1$. If $\gamma > 2$, it holds
 \begin{align*}
\lim_{n \rightarrow \infty} \frac{\Theta(n)}{ n^2 } \sum_{r =1}^{a_n} c_{\gamma} r^{1-\gamma} = \lim_{n \rightarrow \infty} \sum_{r =1}^{a_n} r^2 p(r)  = \sum_{r =1}^{\infty} r^2 p(r) =  \frac{\sigma^2}{2} = d \kappa_{\gamma}.
\end{align*}
In order to analyse the case $\gamma=2$, observe that
\begin{align*}
 \log(a_n+1) = \sum_{r=1}^{a_n} \int_{r}^{r+1} u^{-1} du \leq \sum_{r=1}^{a_n} r^{-1}  \leq 1 + \sum_{r=2}^{a_n} \int_{r-1}^{r} u^{-1} \, du = 1+ \log(a_n).
\end{align*}
Then
\begin{align*}
\forall n \geq 1, \quad c_2 \frac{\log(a_n + 1)}{\log(n)} \leq \frac{\Theta(n)}{n^2} \sum_{r=1}^{ a_n } c_{\gamma} r^{1-\gamma} =  \frac{c_{2}}{\log(n) } \sum_{r=1}^{ a_n }  r^{-1} \leq c_2 \frac{1+ \log(a_n)}{\log(n) }.
\end{align*}
When $n \rightarrow \infty$, the left-hand and right-hand sides of the display above both go to ${c_2=d\kappa_{\gamma}}$, proving \eqref{convkappagamma}, and also \eqref{thetan}. Hence, there exists $M_1 >0$ such that
	\begin{align} \label{defM1}
		\frac{\Theta(n)}{n^{2}} \sum_{x_d={0}}^{\ell -1} \; \sum_{y_d=-  n +1}^{-1}   (x_d-y_d)^{-\gamma} \leq \frac{\Theta(n)}{n^{2}} \sum_{x_d={0}}^{n -1} \; \sum_{y_d=-  n +1}^{-1}   (x_d-y_d)^{-\gamma}  \leq M_1,
	\end{align}
for every $\varepsilon > 0$ and every $n \geq 1$ {such that} ${\ell \leq n}$. {Combining \eqref{defM1} with \eqref{young}, the display in \eqref{corclaim} is} bounded from above by a constant times
\begin{align} \label{rpl1c}
\frac{M_1 E_h }{n^{d-1}} \sum_{ | \hat{x}_{\star} | \leq C n }   \; \sum_{w=0}^{\ell-1}  I_{\big( (\hat{x}_{\star}, w+1)), ((\hat{x}_{\star}, w) \big) }  (\sqrt{f}, \nu_{h}^n ) A + \frac{M_1 E_h }{n^{d-1}} \sum_{ | \hat{x}_{\star} | \leq C n }  \;  \sum_{w=0}^{\ell-1} \Big( \frac{1}{A} + \frac{1}{n} \Big).
\end{align}
Above, $E_h$ is another positive constant depending only on $h \in \textrm{Ref}$. Choosing $A=\Theta(n) p(\hat{e}_d )[8 M_1 E_h  B n]^{-1}$, the sum of the expressions in \eqref{rpl1a} and \eqref{rpl1c} is bounded from above by a constant times
\begin{align*}
 \frac{B \ell n}{\Theta(n)} + \frac{\ell}{n},
\end{align*}
therefore the display in \eqref{suprpl1} is bounded from above by a constant (depending on $h \in \textrm{Ref}$) times
\begin{align*}
\frac{1}{B} + B \varepsilon  + \varepsilon.
\end{align*}
In {the} last line we used the definition of $\ell$ in \eqref{defellmic}. By choosing $B=\varepsilon^{- 1 /2}$, due to \eqref{timescale} {the} last display vanishes by taking first the limit in $n \rightarrow \infty$  and then ${\varepsilon} \rightarrow 0^+$.      
\end{proof}
\begin{proof}[Proof of Lemma \ref{replemma1smallbet}]
Applying Lemma \ref{entjen} and Feynman-Kac's formula in the same way as in the proof of Lemma \ref{replemma2smallbet}, the expectation in \eqref{rlleft10} is bounded from above by
\begin{equation} \label{suprpl10}
\begin{split}
\frac{C_h+2}{B} + T \sup_{f} \Big\{ &  \frac{\Theta(n)}{n^{d+1}} \sum_{ | \hat{x}_{\star} | < C n } \;   \sum_{x_d=0}^{ n -1} \; \sum_{y_d=- n +1}^{-1}  (x_d-y_d)^{-\gamma} \big| \langle\eta(\hat{x}_{\star}, 0) - \eta( \hat{x}_{\star}, -1 )   , f \rangle_{\nu_{h}^n} \big|  \\
 +& \frac{\Theta(n)}{B n^d} \langle \mcb L_n \sqrt{f}, \sqrt{f} \rangle_{\nu_{h}^n}   \Big\}, 
\end{split}
\end{equation}
for every $B>0$, where the supremum above is carried over all the densities $f$ with respect to $\nu_{h}^n$. From Proposition \ref{bound}, the second term inside the supremum in {the} last display is bounded from above by
\begin{align}
 & \frac{M_h}{B} - \frac{\alpha \Theta(n)}{8Bn^{d+\beta}} \sum_{ | \hat{x}_{\star} | \leq C n }   I_{( (\hat{x}_{\star}, 0)), ((\hat{x}_{\star}, -1)) }  (\sqrt{f}, \nu_{h}^n ) p( \hat{e}_d ), \label{rpl0a}
\end{align}
where $M_h$ is a positive constant depending only on $h \in \textrm{Ref}$. Recall the definition of $M_1$ in \eqref{defM1}. From \eqref{defM1} and \eqref{young}, the first term inside the supremum in \eqref{suprpl10} is bounded from above by
\begin{align} \label{rpl10c}
\frac{M_1 E_h }{n^{d-1}} \sum_{ | \hat{x}_{\star} | \leq C n }      I_{\big( (\hat{x}_{\star}, -1)), ((\hat{x}_{\star}, 0) \big) }  (\sqrt{f}, \nu_{h}^n ) A + \frac{M_1 E_h }{n^{d-1}} \sum_{ | \hat{x}_{\star} | \leq C n }    \Big( \frac{1}{A} + \frac{1}{n} \Big),
\end{align}
where $E_h$ is another positive constant depending only on $h \in \textrm{Ref}$. Choosing $A= \alpha \Theta(n) p(\hat{e}_d) [M_1 E_h  B n^{1+\beta}]^{-1}$, the sum of the expressions in \eqref{rpl1a} and \eqref{rpl1c} is bounded from above by a constant times $ \frac{B n^{1+\beta}}{\Theta(n)} + \frac{1}{n}$, therefore the display in \eqref{suprpl10} is bounded from above by a constant (depending on $h \in \textrm{Ref}$) times
\begin{align*}
\frac{1}{B} +  \frac{B n^{1+\beta}}{\Theta(n)} + \frac{1}{n}.
\end{align*}
Choose $B = \log(n)$. Since $\beta \in [0,1)$, from \eqref{timescale}, {the} last display vanishes by taking first the limit in $n \rightarrow \infty$ and then $\epsilon \rightarrow 0^+$ and the proof ends.
\end{proof}
We end this section by observing that Proposition \ref{replemlargebeta} is a direct consequence of Lemma \ref{obe} below, when we choose $F_1 = \nabla G^{+}(0)$, $F_2 = \nabla G^{-}(0)$ and $\lambda= \lambda_0$, where $\lambda_0: \mathbb{Z} \rightarrow \mathbb{R}$ is given by
\begin{equation*} 
\lambda_0(x) :=  \sum_{y=-\infty}^{-1} (y-x)    p(y-x),  \, \, x \leq -1; \quad \lambda_0(x):= \sum_{y=0}^{\infty}  (y-x)    p(y-x), \, \,  x \geq 0.
\end{equation*}
In order to state the result of Lemma \ref{obe}, we assume that $d=1$. Recall the definitions of $\overleftarrow{\eta}^{\varepsilon n }_s(0)$ and $\overrightarrow{\eta}^{\varepsilon n }_s(0)$ in \eqref{medbox1}. Similar results to Lemma \ref{obe} are stated in Lemma 6.2 of \cite{casodif} and Lemma 5.7 of \cite{byrondif}. Since Lemma \ref{obe} can obtained in an analogous way as those two previous results, we leave the proof of Lemma \ref{obe} to the reader.
\begin{lem}  \label{obe}
\textit{
Assume \eqref{1entbound}. Fix $\gamma >2$ and $ t \in [0,T]$. For every $F_1, F_2 \in L^{\infty}([0,T])$ and $\lambda \in L^1(\mathbb{Z})$, it holds
\begin{align*}
& \varlimsup_{\varepsilon \rightarrow 0^+} \varlimsup_{n \rightarrow \infty} \mathbb{E}_{\mu_n} \Big[ \Big| \int_0^t F_1(s)  \sum_{x =0}^{\infty}  \lambda(x) [\eta_{s}^{n}(x) - \overrightarrow{\eta}^{\varepsilon n }_s(0) ] \, ds  \Big|    \Big] = 0, \\
&\varlimsup_{\varepsilon \rightarrow 0^+} \varlimsup_{n \rightarrow \infty} \mathbb{E}_{\mu_n} \Big[  \Big|  \int_0^t F_2(s) \sum_{x =-\infty}^{-1}  \lambda(x) [\eta_{s}^{n}(x) - \overleftarrow{\eta}^{\varepsilon n }_s(0) ] \, ds \Big|   \Big] = 0.
\end{align*}
}
\end{lem}

\section{Energy estimates}  \label{secenerest}

In this section, we {assume} ${d}=1$ and $\gamma >2$. Our goal is to prove that $\rho$ satisfies the second condition of weak solutions of \eqref{ehdnlee}, and therefore $\rho_s(0^{+})$ and $\rho_s(0^{-})$ are well-defined for a.e. $s \in [0,T]$. It is sufficient to prove that
\begin{align} \label{eqsob}
\mathbb{Q} \big (  \rho_s \in  \mcb{H}^1(J), \; \text{a.e.} \; s \in [0,T]   \big) = 1
\end{align}
for every bounded open interval $J$ such that $J \subset \mathbb{R}_{-}^{*}$ or $J \subset \mathbb{R}_{+}^{*}$. Without loss of generality, we assume that $J=(0,a_J)$, where $a_J>0$ is fixed. In order to prove \eqref{eqsob}, we apply the argument analogous to the one used in Section 6.1 of \cite{byrondif}. More exactly, we define $N:=C_c^{0,\infty} ( [0, T] \times J)$, $Y:=L^2([0,T] \times J)$ and make use of the (random) auxiliary linear functional $l_{\rho,J}$ on $N \subset Y$ given by
\begin{align} \label{deflrhoJ}
l_{ \rho, J }(G):=&   \int_0^T \int_{J} \nabla G_t(u) \rho_t(u)\, du \, dt.
\end{align}
The following lemma will be useful in order to prove \eqref{eqsob}.
\begin{lem} \label{lemesten}
\textit{
If there exist $K_0, K_1$ such that
\begin{equation} \label{supkfin}
\forall j \geq 1, \forall G_1, \ldots G_j \in N,  \quad \mathbb{E}_{\mathbb{Q}} \big[ \max_{1 \leq i \leq j} \{ l_{ \rho, J }(G_i) - K_0 \| G_i \|_Y^2 \} \big] \leq K_1,
\end{equation}
then \eqref{eqsob} holds. Above, $K_0$ and $K_1$ are positive constants independent of $j$ and $G_1, \ldots G_j$. 
}
\end{lem}
\begin{proof}
First, define $X:=L^2 \big(0,T; \mcb H^1(J) \big)$ as the metric space of the functions $G:[0,T] \times J$ such that
\begin{align*}
\| G \|_X^2: = \int_0^T \| G_s \|^2_{\mcb H^1(J)} ds < \infty.
\end{align*}
Next, observe that from the fact that $|\rho| \leq 1$ and H\"older's inequality, we get from \eqref{deflrhoJ} that
\begin{equation} \label{boundlrho}
\forall G \in N, \quad |l_{ \rho, J }(G)|  \leq \Big[  \int_0^T \int_{J} \, du \, dt \Big]^{1/2} \Big[  \int_0^T \int_{J} \nabla G_t(u) \, du \, dt \Big]^{1/2} \leq \sqrt{T a_J} \| G \|_X.
\end{equation}
Since $(H^1(J), \| \cdot \| )$ is a separable metric space, we get from Proposition 23.2 (f) in \cite{zeidler} that $(X, \| \cdot \|_X)$ is a separable metric space. In particular, $(N, \| \cdot \|_X)$ is a separable metric space and we can fix a sequence $(G_j)_{j \geq 1} \subset N$ which is dense in $(N, \| \cdot \|_X)$. In order to apply the Monotone Convergence Theorem (MCT), we define $K_2:=\sqrt{T a_J} \| G_1 \|_X + K_0 \| G_1 \|_Y^2 < \infty$ and for every $j \geq 1$, we define
\begin{align*}
 Z_j:= K_2 + \max_{1 \leq i \leq j} \{ l_{ \rho, J }(G_i) - K_0 \| G_i \|_Y^2 \} .
\end{align*}
In particular, from \eqref{boundlrho} and \eqref{supkfin} we get $Z_{j+1} \geq Z_j \geq Z_1 \geq 0$ and $\mathbb{E}_{\mathbb{Q}}[Z_j] \leq K_1 + K_2$, for every $j \geq 1$. Therefore, an application of the MCT leads to
\begin{align} \label{supkcount}
\mathbb{E}_{\mathbb{Q}} \big[ \sup_{j \geq 1} \{ l_{ \rho, J }(G_j) - K_0 \| G_j \|_Y^2 \} \big] =& - K_2 + \mathbb{E}_{\mathbb{Q}} \big[ \sup_{j \geq 1} \{ Z_j \} \big] = -  K_2 + \lim_{j \rightarrow \infty} \mathbb{E}_{\mathbb{Q}}[Z_j] \leq  K_1.
\end{align}
Next we claim that
\begin{align} \label{boundsupF}
	\mathbb{E}_{\mathbb{Q}} \big[ \sup_{G \in N} \{ \ell_{\rho, J} (G) - 3K_0 \| G \|^{2}_{Y} \} \big] \leq K_3, 
\end{align}
where $K_3:=K_1 + \sqrt{T a_J} +  3 K_0$. Indeed, let $\tilde{G} \in N$. Since $(G_i)_{j \geq 1}$ is dense in $(N, \| \cdot \|_{X} )$, there exists $j_0 \geq 1$ such that $\|  \tilde{G} - G_{j_0}  \|_{X} < 1$. From \eqref{boundlrho} and the fact that $\|G\|_Y \leq \|G\|_X$ for every $G \in N$, we get
\begin{align*}
& l_{ \rho, J }( \tilde{G} - G_{j_0} ) + K_0 \big( |  G_{j_0}   \|^{2}_{Y} - \|  \tilde{G}   \|^{2}_{Y} \big) \\
\leq & |l_{ \rho, J }( \tilde{G} - G_{j_0} )| + K_0  ( |  G_{j_0}   \|_{Y} - \|  \tilde{G}   \|_{Y} ) ( \|  G_{j_0}   \|_{Y} + \|  \tilde{G}   \|_{Y} ) \\
\leq & \big[ \sqrt{T a_J} + K_0  ( |  G_{j_0}   \|_{Y} + \|  \tilde{G}   \|_{Y} )  \big]  \|  \tilde{G} - G_{j_0}  \|_{X} <  \sqrt{T a_J} + K_0  ( \|  G_{j_0}   \|_{Y} + \|  \tilde{G}   \|_{Y} ),
\end{align*}
which leads to
\begin{equation} \label{boundsupF1}
\begin{split}
	l_{ \rho, J }(\tilde{G}) - K_0  \| \tilde{G} \|_Y^2 =& l_{ \rho, J }(G_{j_0}) - K_0 \| G_{j_0} \|_Y^2 + l_{ \rho, J }( \tilde{G} - G_{j_0} ) +  K_0 \big( \|  G_{j_0}   \|^{2}_{Y} - \|  \tilde{G}   \|^{2}_{Y} \big) \\
\leq & \sup_{j \geq 1} \{ l_{ \rho, J }(G_j) - K_0 \| G_j \|_Y^2\} \sqrt{T a_J} + K_0  ( \|  G_{j_0}   \|_{Y} + \|  \tilde{G}   \|_{Y} )  .
\end{split}
\end{equation}
Since $\|  \tilde{G}   \|_{Y} \leq 1 + \|  \tilde{G}   \|_{Y}^2$, we have
\begin{equation} \label{boundsupF2}
 \|  G_{j_0}   \|_{Y} + \|  \tilde{G}   \|_{Y}\leq  \|   G_{j_0} - \tilde{G}   \|_{Y} + 2  \|  \tilde{G}   \|_{Y} \leq 3 + 2  \|  \tilde{G}   \|_{Y}^{2}.	
\end{equation}
Then from \eqref{boundsupF1} and \eqref{boundsupF2} we get
\begin{align*}
l_{ \rho, J }(\tilde{G}) - 3 K_0  \| \tilde{G} \|_Y^2 \leq \sup_{j \geq 1} \{ l_{ \rho, J }(G_j) - K_0 \| G_j \|_Y^2\} + \sqrt{T a_J} +  3 K_0,
\end{align*}
for every $\tilde{G} \in N$. This leads to
\begin{align*}
\sup_{\tilde{G} \in N} \{	l_{ \rho, J }(\tilde{G}) - 3K_0  \| \tilde{G} \|_Y^2 \} \leq \sup_{j \geq 1} \{ l_{ \rho, J }(G_j) - K_0 \| G_j \|_Y^2\} + \sqrt{T a_J} +  3 K_0.
\end{align*}
Taking the expectation in both sides of {the} last inequality, we get
\begin{align*}
&\mathbb{E}_{\mathbb{Q}} \big[ \sup_{\tilde{G} \in N} \{	l_{ \rho, J }(\tilde{G}) - 3K_0  \| \tilde{G} \|_Y^2 \}  \big] \\
\leq & \mathbb{E}_{\mathbb{Q}} \big[\sup_{j \geq 1} \{ l_{ \rho, J }(G_j) - K_0 \| G_j \|_Y^2\} \big] + \sqrt{T a_J} +  3 K_0 \leq  K_1 +  \sqrt{T a_J} +  3 K_0 = K_3,
\end{align*}
leading to \eqref{boundsupF}. In {the} last line we used \eqref{supkcount}. Finally, in order to obtain \eqref{eqsob}, we need to prove that $\mathbb{Q}(E_1)=1$, where $E_1$ is the event in which there exists $g_s^{\rho} \in L^2(J)$ such that
\begin{align} \label{defsob}
\forall \phi \in C_c^{\infty}(J), \quad \int_{J} \partial_1 \phi(u)  \rho_s(u) \,  du =-  \int_{J} \phi(u) g_s^{\rho}(u) \, du,
\end{align}
for a.e. $s$ in $[0,T]$. From \eqref{boundsupF}, we get
\begin{align*}
\mathbb{E}_{\mathbb{Q}} \big[ \sup_{G \in N, \| G \|_Y \leq 1} \{ l_{ \rho, J }(G) \} \big] \leq 3K_0 + K_3 < \infty.
\end{align*}
In particular, $\mathbb{Q}(E_2)=1$, where $E_2$ is the event where $l_{\rho,J}$ is continuous. By density we can extend it from $N$ to the Hilbert space $Y$. Then we can apply Riesz's Representation Theorem to conclude that there exists $\xi^{\rho} \in Y$ such that
\begin{equation} \label{riesz}
l_{ \rho,J }(G)= \int_0^T \int_{J} \nabla G_t(u)  \rho_t(u)\,  du \, dt =  \int_0^T \int_{J} G_t(u) \xi_t^{\rho}(u) \, du \,  dt, 
\end{equation}
for every $G \in Y$. Now given $\phi \in C_c^{\infty}(J)$ and $t \in [0,T]$, define $G_{t,\phi}: [0,T] \times \mathbb{R} \rightarrow \mathbb{R}$ in $L^2 \big( [0,T] \times J \big)$ by 
$G_{t,\phi}(s,u) =\phi(u)\mathbbm{1}_{ (s,u) \in [0,t] \times J}$. Therefore from \eqref{riesz}, for every $t \in [0,T]$ it holds
\begin{align*}
 \int_0^t \Big\{ \int_{J} \nabla \phi(u)  \rho_s(u) \,  du +  \int_{J} \phi(u) \big(-  \xi^{\rho}_t(u) \big) \, du \Big\} \, ds=0.
\end{align*}
As a consequence, for a.e. $s$ in $[0,T]$ we have 
\begin{align*}
\int_{J} \nabla \phi(u)  \rho_s(u) \,  du =-  \int_{J} \phi(u) \big(-  \xi^{\rho}_s(u) \big) \, du.
\end{align*}
Note that  $\xi^{\rho} \in Y$, hence $- \xi^{\rho}(s, \cdot) \in L^2(J)$ for a.e. $s$ in $[0,T]$. Choosing $g^{\rho}_s:=-\xi^{\rho}_s$ for every $s \in [0,T]$, we get \eqref{defsob}, therefore $E_2 \subset E_1$ and $\mathbb{Q}(E_1) \geq \mathbb{Q}(E_2)=1$ as desired.
\end{proof}
In order to be able to apply Lemma \ref{lemesten}, we need the following result.
\begin{prop} \label{estenergsembarlenfor}
\textit{
Assume \eqref{1entbound}. Then, there exist positive constants $K_0$ and $K_1$ satisfying \eqref{supkfin}.
}
\end{prop}
\begin{proof} 
In order to prove this result, we apply the strategy used to prove Proposition 6.1 in \cite{byrondif}. Fix $j \geq 1, G_1, \ldots, G_j \in N$. For every $c>0$, define the random application $\Phi_{j,c}: \mcb D ( [0,T], \mcb M^+) \rightarrow \mathbb{R}$ by
\begin{align*}
\Phi_{j,c}(\pi):=& \max_{1 \leq 1 \leq j} \Big\{ - c \| G_i \|_Y^2 +  \int_0^T \int_{J} \nabla G_i(t,u) \pi_t(du) \, dt  \Big\} \\
=& \max_{1 \leq 1 \leq j} \Big\{ - c \| G_i \|_Y^2 +  \int_0^T \int_{J} \nabla G_i(t,u) \rho(t,u) \, du \, dt  \Big\}  = \max_{1 \leq 1 \leq j} \{ - c \| G_i \|_Y^2 + \ell_{\rho, J }(G_i)  \}.
\end{align*}  
Let $(\mathbb{Q}_{n_k})_{k \geq 1}$ be a subsequence converging weakly to the limit point $\mathbb{Q}$. Since $\Phi_{j,c}$ is a continuous and bounded linear functional in the Skorohod topology of $\mcb D ( [0,T], \mcb M^{+} )$ for every $c>0$, it holds
\begin{align*}
\mathbb{E}_{\mathbb{Q}} \big[ \max_{1 \leq 1 \leq j} \{ - c \| G_i \|_Y^2 + \ell_{\rho, J }(G_i)  \} \big] =& \mathbb{E}_{\mathbb{Q}} [ \Phi_{j,c}(\pi) ] = \lim_{k \rightarrow \infty} \mathbb{E}_{\mathbb{Q}_{n_k}} [ \Phi_{j,c}(\pi) ] \leq \varlimsup_{n \rightarrow \infty} \mathbb{E}_{\mathbb{Q}_{n}} [ \Phi_{j,c}(\pi) ].
\end{align*}
Since $(\mathbb{Q}_n)_{n \geq 1 }$ is induced by $(\mathbb{P}_{\mu_n})_{n  \geq 1}$, which is induced by $(\mu_n)_{n  \geq 1}$, the last expectation above can be rewritten as $\mathbb{E}_{\mu_n}[Z_{j,n}]$, where $Z_{j,n}:= \max_{1 \leq i \leq j} \{ \tilde{Z}_{i,n} \}$ and for $i \in \{1, \ldots, j\}$, $\tilde{Z}_{i,n}: \mcb D ( [0,T], \mcb M^+) \rightarrow \mathbb{R}$ is given by
\begin{align*}
 \tilde{Z}_{i,n}:= - c \| G_i \|_Y^2 +  \int_0^T \frac{1}{n} \sum_{x=2}^{a_J n-1}  \nabla G_t ( \tfrac{x}{n} ) \eta_t^n(x) \, dt. 
\end{align*}
Choosing $B=1$ in Lemma \ref{entjen}, we have that $ \varlimsup_{n \rightarrow \infty}\mathbb{E}_{\mu_n}[Z_{j,n}]$ is bounded from above by
\begin{align*}
& \varlimsup_{n \rightarrow \infty} \Big( C_h + \frac{j}{n} + \max_{1 \leq i \leq j} \Big\{ \frac{1}{n} \log   \Big(  \mathbb{E}_{\nu_h^n} \Big[   \exp \Big\{ \int_0^T \sum_{x=2}^{a_J n-1}  \nabla G_i ( t,\tfrac{x}{n} ) \eta_t^n(x) \, dt - c n \|G_i \|_{Y}^2 \Big\}  \Big]  \Big) \Big\} \Big) \\
=& C_h + \varlimsup_{n \rightarrow \infty} \Big(  \max_{1 \leq i \leq j} \Big\{ \frac{1}{n} \log   \Big(  \mathbb{E}_{\nu_h^n} \Big[   \exp \Big\{ \int_0^T \sum_{x=2}^{a_J n-1}  \nabla G_i ( t,\tfrac{x}{n} ) \eta_t^n(x) \, dt - c n \|G_i \|_{Y}^2 \Big\}  \Big]  \Big) \Big\} \Big).
\end{align*} 
Therefore we are done if we can prove that there exist $c>0$ and $K_1>0$ such that
\begin{align} \label{jensenest}
\forall G \in N, \quad  \varlimsup_{n \rightarrow \infty} \Big\{   \frac{1}{n} \log   \Big(  \mathbb{E}_{\nu_h^n} \Big[   \exp \Big\{ \int_0^T \sum_{x=2}^{a_J n-1}  \nabla G ( t,\tfrac{x}{n} ) \eta_t^n(x) \, dt - c n \|G \|_{Y}^2 \Big\}  \Big]  \Big) \Big\} \leq K_1.
\end{align}
From the Feynman-Kac's formula, the left-hand side in {the} last display is bounded from above by
\begin{align} \label{supestenerg}
   \int_0^T  \sup_f \Big\{   \sum_{x=2}^{a_J n-1} \frac{1}{n} \nabla G_t \left( \tfrac{x}{n} \right)  \langle \eta(x), f\rangle_{\nu_h^n} + \frac{\Theta(n)}{n}  \langle\mcb {L}_{n} \sqrt{f} , \sqrt{f} \rangle_{\nu_h^n}  - c  \|G \|_{Y}^2 \Big\} \, dt,
\end{align}
where the supremum is carried over all the densities $f$ with respect to $\nu_{h}^n$. Since $G \in N$, by a first order Taylor expansion on $G$, we get (neglecting terms of lower order with respect to $n$)
\begin{align*}
 \frac{1}{n} \nabla G_t \left( \tfrac{x}{n} \right)  =  G_t \left( \tfrac{x}{n} \right) - G_t \left( \tfrac{x-1}{n} \right), \quad x \in \{2, \ldots,  a_J n -1\}.
\end{align*}
Since $G$ has compact support, we can assume without loss of generality that $\nabla G_t ( \tfrac{1}{n} ) =0$ for every $t \in [0,T]$. Then the sum inside the supremum in \eqref{supestenerg} is bounded from above by
\begin{align} \label{sumest}
  \Big|  \sum_{x =1}^{a_J n -1} G (t, \tfrac{x}{n} ) \int[ \eta(x) - \eta(x+1) ]  f(\eta) \, d \nu_h^n \Big| \leq  \sum_{x =1}^{a_J n -1} |G (t, \tfrac{x}{n} )| \Big| \int[ \eta(x) - \eta(x+1) ]  f(\eta) \, d \nu_h^n \Big|,
\end{align}
plus terms of lower order with respect to $n$ that vanish as $n \rightarrow \infty$. Recalling the definition of $E_h$ in \eqref{young} and choosing
\begin{align*}
 A_{x,x+1}:= \frac{8 E_h|G \left(t, \tfrac{x}{n} \right)|}{n p(1)},  \quad x \in \{ 1, \ldots, a_J n - 1 \}, 
\end{align*}
we get from \eqref{young} that the right-hand side of \eqref{sumest} is bounded from above by
\begin{align} \label{sumest2}
 \sum_{x =1}^{a_J n -1} \frac{np(1) I_{x,x+1}   (\sqrt{f}, \nu_h^n )}{8} + \frac{8 E_h^2}{p(1)} \frac{1}{n} \sum_{x=1}^{a_J n -1} [G \left(t, \tfrac{x}{n} \right)]^2 + \frac{E_h}{n} \sum_{x=1}^{a_J n -1} |G \left(t, \tfrac{x}{n} \right)|.
\end{align}
Above we used the fact that $p(1)>0$. Next, observe that all the bonds $\{x,x+1\}$ in \eqref{sumest} are \textit{fast} bonds. Due to \eqref{boundlip} and $\gamma >2$, it holds
\begin{align*}
&\frac{\Theta(n)}{n}  \langle\mcb {L}_{n} \sqrt{f} , \sqrt{f} \rangle_{\nu_h^n} \leq -\frac{n}{4}  D^{\mcb F}  (\sqrt{f}, \nu_h^n ) + M_h \leq  -   \sum_{x =1}^{a_J n -1} \frac{np(1) I_{x,x+1}   (\sqrt{f}, \nu_h^n )}{8}  + M_h,
\end{align*}
where $M_h$ is given in \eqref{boundlip}. Combining {the} last display with \eqref{sumest2}, we conclude that the expression inside the supremum in \eqref{supestenerg} is bounded from above by
\begin{align*} 
\frac{8 E_h^2}{p(1)} \frac{1}{n} \sum_{x=1}^{a_J n -1} [G \left(t, \tfrac{x}{n} \right)]^2 + M_h - c  \|G \|_{Y}^2 + E_h \sum_{x=1}^{a_J n -1} \frac{|G \left(t, \tfrac{x}{n} \right)|}{\sqrt{n}} \frac{1}{\sqrt{n}}.
\end{align*}
Since $\frac{|G \left(t, \tfrac{x}{n} \right)|}{\sqrt{n}} \frac{1}{\sqrt{n}} \leq \frac{1}{n} + \frac{[G \left(t, \tfrac{x}{n} \right)]^2}{n}$, {the} last sum is bounded from above by $a_J + n^{-1} \sum_{j=1}^{a_J n -1} [G \left(t, \tfrac{x}{n} \right)]^2$. Then taking the $\varlimsup$ when $n \rightarrow \infty$, the display in \eqref{supestenerg} is bounded from above by
\begin{align*}
 T (M_h + a_J E_h  ) + \big( E_h \big( 8 E_h [p(1)]^{-1} +1 \big) - c \big) \|G \|_{Y}^2.
\end{align*}
Therefore, choosing $c=2 E_h \big( 8 E_h [p(1)]^{-1} +1 \big) >0$ and taking the $\limsup$ when $n \rightarrow \infty$ we get that the left-hand side of \eqref{jensenest} is bounded from above by $K_1:=T (M_h + a_J E_h  )$, ending the proof.
\end{proof}

	\appendix

\section{Uniqueness of weak solutions} \label{secuniq}

In this section we prove Lemma \ref{lemuniq}. We recall that weak solutions of \eqref{ehdnlee} are also weak solutions of \eqref{ehdn}, therefore Lemma \ref{lemuniq} is a direct consequence of Propositions \ref{propuniqdif} and \ref{propuniqneu}.
\begin{prop} \label{propuniqdif}
\textit{
There exists at most one weak solution of \eqref{ehd}.}
\end{prop}
\begin{proof}
{Fix $\kappa >0$.} By linearity, it is enough to prove the uniqueness of \eqref{ehd} when $g \equiv 0$. Assume that $\rho:[0,T] \times \mathbb{R}^d \rightarrow [0,1]$ is such that for every $t \in [0,T]$ and for every $G \in \mcb S_{\textrm{Dif}}$, it holds
\begin{align*}
 F_{\textrm{Dif}}(t,\rho,G,0,\kappa)=  \int_{\mathbb{R}^d} \rho(t, \hat{u}) G(t, \hat{u} ) \, d \hat{u} -  \int_0^t \int_{\mathbb{R}^d} \rho(s, \hat{u}) [ \kappa \Delta + \partial_s ] G(s, \hat{u}) \, d \hat{u} \, ds=0.
\end{align*}
{In the last line, we used the fact that $\mathbb{L}_{\Delta} G= \Delta G$ for any $G \in \mcb S_{\textrm{Dif}}$.} Next, we will make use of {a regularization argument which is analogous to the one in} Section 8.1 of \cite{jara2009hydrodynamic}. More exactly, {for every mensurable and bounded $H:(0, \infty) \times \mathbb{R}^d \rightarrow \mathbb{R}$,} define the following extension {$\rho^H: \mathbb{R} \times \mathbb{R}^d \rightarrow \mathbb R$} of $\rho$:
\begin{align*}
\rho^H(s, \hat{u})=
\begin{cases}
0, \quad &(s, \hat{u}) \in (-\infty,0) \times \mathbb{R}^d; \\
\rho(s,u), \quad  &(s, \hat{u}) \in [0,T] \times \mathbb{R}^d; \\
H(s-T, \hat{u}), \quad  &(s, \hat{u}) \in (T, \infty) \times \mathbb{R}^d.
\end{cases}
\end{align*}
In particular, $\rho^H$ is bounded. Moreover, for every $t \in [0, T]$ and every $G \in \mcb S_{\textrm{Dif}}$, it holds
\begin{align} \label{eqintH}
  \int_{\mathbb{R}^d} \rho^H(t, \hat{u} ) G(t, \hat{u} ) \, d \hat{u} -  \int_0^t \int_{\mathbb{R}^d} \rho^H(s,\hat{u}) [ \kappa \Delta + \partial_s ] G(s, \hat{u}) \, d \hat{u} \, ds=  F_{\textrm{Dif}}(t,\rho,G,0,\kappa)= 0.
\end{align} 
Next, for every $\theta >0$, define $\rho^H_{\theta}: \mathbb{R} \times \mathbb{R}^d \rightarrow \mathbb{R}$ by $\rho^H_{\theta}(s,\hat{u}) := \rho^H(s-\theta,\hat{u}), (s,\hat{u}) \in \mathbb{R} \times \mathbb{R}^d$. Hence, $\rho^H_{\theta}$ is a temporal translation of $\rho^H$, $\rho^H_{\theta}$ is bounded and $\rho^H_{\theta} \equiv 0$ on $[0, t] \times \mathbb{R}^d$ for $0 \leq t < \theta$. Then
\begin{equation} \label{eqintthet0}
\forall t \in [0, \theta), \;  \forall G \in \mcb S_{\textrm{Dif}}, \quad \int_{\mathbb{R}^d} \rho^H_{\theta}(t, \hat{u}) G(t, \hat{u} ) \, d \hat{u} -  \int_0^t \int_{\mathbb{R}^d} \rho^H_{\theta}(s, \hat{u}) [ \kappa \Delta + \partial_s ] G(s, \hat{u}) \, d \hat{u} \, ds =0.
\end{equation}
On the other hand, if $\theta \leq t \leq T$, for any $G \in \mcb S_{\textrm{Dif}}$, it holds
\begin{align}
& \int_{\mathbb{R}^d} \rho^H_{\theta}(t, \hat{u}) G(t, \hat{u} ) \, d \hat{u} -  \int_0^t \int_{\mathbb{R}^d} \rho^H_{\theta}(s, \hat{u}) [ \kappa \Delta + \partial_s ] G(s, \hat{u}) \, d \hat{u} \, ds \nonumber \\
=& \int_{\mathbb{R}^d} \rho^H(t-\theta, \hat{u}) G(t, \hat{u} ) \, d \hat{u} -  \int_{\theta}^t \int_{\mathbb{R}^d} \rho^H(s-\theta, \hat{u}) [ \kappa \Delta + \partial_s ] G(s, \hat{u}) \, d \hat{u} \, ds \nonumber \\
=& \int_{\mathbb{R}^d} \rho^H(t-\theta, \hat{u}) G(t, \hat{u} ) \, d \hat{u} -  \int_{0}^{t-\theta} \int_{\mathbb{R}^d} \rho^H(r, \hat{u}) [ \kappa \Delta + \partial_r ] G(r+\theta, \hat{u}) \, d \hat{u} \, dr. \label{eqintthet1}
\end{align}
Keeping \eqref{eqintH} and \eqref{eqintthet1} in mind, we observe that for every $G \in \mcb S_{\textrm{Dif}}$, there exists $G_{\theta} \in \mcb S_{\textrm{Dif}}$ such that $G_{\theta}(r, \hat{u})=G(r+\theta, \hat{u})$, for any $(r, \hat{u}) \in [0, \infty) \times \mathbb{R}^d$, since $\theta >0$. Then, the display in \eqref{eqintthet1} can be rewritten as
\begin{align*}
  \int_{\mathbb{R}^d} \rho^H(t - \theta, \hat{u} ) G_{\theta}(t - \theta, \hat{u} ) \, d \hat{u} -  \int_0^{t-\theta} \int_{\mathbb{R}^d} \rho^H(r,\hat{u}) [ \kappa \Delta + \partial_r ] G_{\theta}(r, \hat{u}) \, d \hat{u} \, ds=  F_{\textrm{Dif}}(t-\theta,\rho,G_{\theta},0,\kappa)= 0.
\end{align*}
The equalities in the last line hold due to \eqref{eqintH}. Therefore, combining the last display with \eqref{eqintthet1}, we conclude that for any $\theta \in (0, T)$, we have 
\begin{equation} \label{eqintthet2}
\forall t \in [0, T], \; \forall G \in \mcb S_{\textrm{Dif}}, \quad \int_{\mathbb{R}^d} \rho^H_{\theta}(t, \hat{u}) G(t, \hat{u} ) \, d \hat{u} -  \int_0^t \int_{\mathbb{R}^d} \rho^H_{\theta}(s, \hat{u}) [ \kappa \Delta + \partial_s ] G(s, \hat{u}) \, d \hat{u} \, ds =0.
\end{equation}
Next, let $\phi \in C_c^1(\mathbb{R})$ be such that its compact support is $K \subset (0, T)$ and define $\rho^H_{\phi}: \mathbb{R} \times \mathbb{R}^d \rightarrow \mathbb{R}$ by 
\begin{align*}
\rho^H_{\phi}(s,\hat{u}) := \int_{\mathbb{R}} \rho^H (s - \theta, \hat{u})  \phi(\theta) \, d \theta = \int_{0}^T \rho^H (s - \theta, \hat{u})  \phi(\theta) \, d \theta, \quad  (s,\hat{u}) \in \mathbb{R} \times \mathbb{R}^d.
\end{align*}
In particular, $\rho^H_{\phi}$ is bounded, $\rho^H_{\phi}(0, \cdot) \equiv 0$ on $\mathbb{R}^d$ and for any $t \in [0, T]$ and $G \in \mcb S_{\textrm{Dif}}$, it holds
\begin{align*}
& \int_{\mathbb{R}^d} \rho^H_{\phi}(t, \hat{u}) G(t, \hat{u} ) \, d \hat{u} -  \int_0^t \int_{\mathbb{R}^d} \rho^H_{\phi}(s, \hat{u}) [ \kappa \Delta + \partial_s ] G(s, \hat{u}) \, d \hat{u} \, ds \\
= & \int_{\mathbb{R}^d} \Big\{\int_{0}^T \rho^H (t - \theta, \hat{u})  \phi(\theta) \, d \theta \Big\} G(t, \hat{u} ) \, d \hat{u} -  \int_0^t \int_{\mathbb{R}^d} \Big\{\int_{0}^T \rho^H (s - \theta, \hat{u})  \phi(\theta) \, d \theta \Big\} [ \kappa \Delta + \partial_s ] G(s, \hat{u}) \, d \hat{u} \, ds \\
=&\int_{0}^T \phi(\theta) \Big\{ \int_{\mathbb{R}^d} \rho^H_{\theta}(t, \hat{u}) G(t, \hat{u} ) \, d \hat{u} -  \int_0^t \int_{\mathbb{R}^d} \rho^H_{\theta}(s, \hat{u}) [ \kappa \Delta + \partial_s ] G(s, \hat{u}) \, d \hat{u} \, ds \Big\} d \theta  =0.
\end{align*}
In the last line, we applied Fubini's Theorem and \eqref{eqintthet2}. In particular, for any $v \in \mathbb{R}^d$ fixed and any $t \in [0, T]$, we have 
\begin{equation} \label{spattrans}
\begin{split}
\forall G \in \mcb S_{\textrm{Dif}},  \quad & \int_{\mathbb{R}^d} \rho^H_{\phi}(t, \hat{u}-\hat{v}) G(t, \hat{u} ) \, d \hat{u} -  \int_0^t \int_{\mathbb{R}^d} \rho^H_{\phi}(s, \hat{u}-\hat{v}) [ \kappa \Delta + \partial_s ] G(s, \hat{u}) \, d \hat{u} \, ds \\
=& \int_{\mathbb{R}^d} \rho^H_{\phi}(t, \hat{w}) G_{\hat{v}}(t, \hat{w})  \, d \hat{w} -  \int_0^t \int_{\mathbb{R}^d} \rho^H_{\phi}(s, \hat{w}) [ \kappa \Delta + \partial_s ] G_{\hat{v}}(s, \hat{w}) \, d \hat{w} \, ds =0.
\end{split}
\end{equation}
In the last display, $G_{\hat{v}} \in \mcb S_{\textrm{Dif}}$ is given by $G_{\hat{v}}(s, \hat{u}):=G(s, \hat{u}+\hat{v}),  (s,\hat{u}) \in [0, \infty) \times \mathbb{R}^d$, for any $\hat{v} \in \mathbb{R}^d$ and any $G \in \mcb S_{\textrm{Dif}}$. Next, for every $\psi \in C_c^2(\mathbb{R}^d)$, define
\begin{align*}
 \rho^H_{\phi,\psi}(s,\hat{u}) := \int_{\mathbb{R}^d} \rho^H_{\phi} (s , \hat{u} - \hat{v})  \psi(\hat{v}) \, d \hat{v} \quad  (s,\hat{u}) \in \mathbb{R} \times \mathbb{R}^d.
\end{align*}
In particular, $\rho^H_{\phi, \psi}$ is bounded, $\rho^H_{\phi, \psi}(0, \cdot) \equiv 0$ on $\mathbb{R}^d$ and for any $t \in [0, T]$ and $G \in \mcb S_{\textrm{Dif}}$, it holds
\begin{align*}
& \int_{\mathbb{R}^d} \rho^H_{\phi, \psi}(t, \hat{u}) G(t, \hat{u} ) \, d \hat{u} -  \int_0^t \int_{\mathbb{R}^d} \rho^H_{\phi, \psi}(s, \hat{u}) [ \kappa \Delta + \partial_s ] G(s, \hat{u}) \, d \hat{u} \, ds \\
= & \int_{\mathbb{R}^d} \Big\{ \int_{\mathbb{R}^d} \rho^H_{\phi} (t , \hat{u} - \hat{v})  \psi(\hat{v}) \, d \hat{v} \Big\} G(t, \hat{u} ) \, d \hat{u} -  \int_0^t \int_{\mathbb{R}^d} \Big\{\int_{\mathbb{R}^d} \rho^H_{\phi} (s , \hat{u} - \hat{v})  \psi(\hat{v}) \, d \hat{v} \Big\} [ \kappa \Delta + \partial_s ] G(s, \hat{u}) \, d \hat{u} \, ds \\
=& \int_{\mathbb{R}^d} \psi(\hat{v}) \Big\{ \int_{\mathbb{R}^d} \rho^H_{\phi}(t, \hat{w}) G_{\hat{v}}(t, \hat{w})  \, d \hat{w} -  \int_0^t \int_{\mathbb{R}^d} \rho^H_{\phi}(s, \hat{w}) [ \kappa \Delta + \partial_s ] G_{\hat{v}}(s, \hat{w}) \, d \hat{w} \, ds \Big\}  d \hat{v} = 0.
\end{align*}
In the last line, we applied Fubini's Theorem and \eqref{spattrans}. Now we observe that $\rho^H_{\phi, \psi} \in C^{1,2}([0, \infty) \times \mathbb{R}^d)$. Therefore, after performing some integrations by parts, we get $[ \kappa \Delta  - \partial_s ] \rho^H_{\phi,\psi} \equiv 0$ a.e. in $(0, T) \times \mathbb{R}^d$.

The proof follows by contradiction. In what follows, given $f: [0,T] \times \mathbb{R}^d \rightarrow \mathbb{R}$, we write $f \not\equiv 0$ a.e. in $[0,T] \times \mathbb{R}^d$ to denote that it is false that $f \equiv 0$ a.e. in $[0,T] \times \mathbb{R}^d$ (but we do not exclude the possibility that $f \equiv 0$ in a very large subset $[0,T] \times \mathbb{R}^d$). 

Assume $\rho \not\equiv 0$ a.e. in $[0,T] \times \mathbb{R}^d$. {In this case, we can choose $H, \phi, \psi$ appropriately such that $\rho^H_{\phi,\psi} \not\equiv 0$ a.e. in $[0,T] \times \mathbb{R}^d$. }This leads to ${\tilde{\rho}} \not\equiv 0$ a.e. in $[0,T] \times \mathbb{R}^d$, where ${ \tilde{\rho} }:[0,T] \times \mathbb{R}^d \rightarrow \mathbb{R}$ is given by 
\begin{align*}
{\tilde{\rho}}(s,\hat{u}) := {\rho^H_{\phi, \psi}} (s, \kappa^{-\frac{1}{2}} \hat{u} ),  \quad (s,\hat{u}) \in [0,T] \times \mathbb{R}^d.
\end{align*}
{Since $\rho, H$ are bounded, we get that $\tilde{\rho}$ is also bounded. Moreover, $[  \Delta  - \partial_s ] \tilde{\rho} \equiv 0$ a.e. in $(0, T) \times \mathbb{R}^d$, $\tilde{\rho} \in C^{1,2} [0,T ] \times \mathbb{R}^d )$ and $\tilde{\rho} \equiv 0$ on $\{ 0 \} \times \mathbb{R}^d$. } Therefore, from Theorem 2.6  in \cite{evans1998partial}, $ {\tilde{\rho}} \equiv 0 $ in $(0, T) \times \mathbb{R}^d$. Since this leads to a contradiction, we get $\rho \equiv 0$ a.e. in $[0,T] \times \mathbb{R}^d$ and the proof ends.
\end{proof}
\begin{prop} \label{propuniqneu}
\textit{
There exists at most one weak solution of \eqref{ehdn}.
}
\end{prop}
\begin{proof}
We begin by observing that for every $A \subset \mathbb{R}^d$ with Lebesgue measure on $\mathbb{R}^d$ equal to zero, it holds
\begin{equation} \label{L1negl}
\forall H \in L^1( \mathbb{R}^d ), \quad \int_A H(\hat{u}) \, d \hat{u} \  =0.
\end{equation}

Assume that $\rho:[0,T] \times \mathbb{R}^d \rightarrow [0,1]$ is such that for every $t \in [0,T]$ and for every $G \in \mcb{S}_{\textrm {Neu}}$, it holds
\begin{align} \label{uniqneu}
 \int_{\mathbb{R}^d} \rho_t( { \hat{u} }) G_t( \hat{u} ) \, d \hat{u} -  \int_0^t \int_{\mathbb{R}^d} \rho_s(\hat{u}) [ \kappa { \mathbb{L}_{\Delta} } + \partial_s ] G_s( \hat{u}) \, d \hat{u} \, ds=0.
\end{align}
This means that $\rho$ is a weak solution of \eqref{ehdn} for $g \equiv 0$. By linearity, the proof ends if we can prove that $\rho \equiv 0$ a.e. in $[0,T] \times \mathbb{R}^d$. Next, define $\rho^{+}:[0,T] \times \mathbb{R}^{ { d}} \rightarrow [0,1]$  by
\begin{align} \label{defrho+}
\rho^{+}_s( \hat{u}_{\star}, u_d ):=
\begin{cases}
\rho_s( \hat{u}_{\star}, -u_d ), \quad & (s, \hat{u}_{\star}, u_d ) \in [0,T] \times \mathbb{R}^{d-1} \times (-\infty,0), \\
\rho_s( \hat{u}_{\star}, u_d ), \quad &  (s, \hat{u}_{\star}, u_d ) \in [0,T] \times \mathbb{R}^{d-1} \times [0, \infty).
\end{cases}
\end{align} 
{For any $G \in \mcb S_{\textrm{Dif}}$, define} $G^{*-}, G^{*+} \in \mcb S_{\textrm{Dif}}$ by {$G^{*-} \equiv 0$ on $[0,\infty) \times \mathbb{R}^{d}$ and} 
\begin{align} \label{defGstarplus}
  G^{*+}_s( \hat{u}_{\star}, u_d ):=G_s( \hat{u}_{\star}, u_d )+G_s( \hat{u}_{\star}, - u_d ), \quad (s, \hat{u}_{\star}, u_d ) \in [0,\infty) \times \mathbb{R}^{d-1} \times \mathbb{R}.
\end{align}
{Observe that $\partial_{\hat{e}_d} G^{*+}_s( \hat{u}_{\star}, u_d ) = \partial_{\hat{e}_d} G_s( \hat{u}_{\star}, u_d ) -  \partial_{\hat{e}_d} G_s( \hat{u}_{\star}, -u_d )$, for any $(s, \hat{u}_{\star}, u_d ) \in [0,\infty) \times \mathbb{R}^{d-1} \times \mathbb{R}$.} In particular, $\partial_{\hat{e}_d} G^{*-}_s( \hat{u}_{\star}, 0 ) = \partial_{\hat{e}_d} G^{*+}_s( \hat{u}_{\star}, 0 ) = 0$, for every $s\in [0,T]$ and every $\hat{u}_{\star} \in \mathbb{R}^{d-1}$. Next, {define $G^{*}:=\mathbbm{1}_{\{ [0, \infty)  \times \mathbb{R}_{-}^{d*} \} } G^{*-} + \mathbbm{1}_{ \{[0, \infty) \times \mathbb{R}_{+}^{d*} \}} G^{*+}$.} Then $G^{*} \in \mcb{S}_{\textrm {Neu}}$. Finally, we observe that ${F_{\textrm{Dif}}(t,\rho^{+},G,0,\kappa)}$ can be rewritten as
\begin{align*}
& \int_{\mathbb{R}^{d-1} \times \mathbb{R}^{*} } \rho^{+}_t( \hat{u}_{\star}, u_d ) G_t( \hat{u}_{\star}, u_d ) \, d \hat{u} -  \int_0^t \int_{\mathbb{R}^{d-1} \times \mathbb{R}^{*}} \rho^{+}_s( \hat{u}_{\star}, u_d ) [ \kappa \mathbb{L}_{\Delta}  + \partial_s ] G_s( \hat{u}_{\star}, u_d ) \, d \hat{u} \, ds.
\end{align*}
In the last line, we combined Remark \ref{remL1disc} and the fact that $\rho$ is bounded with \eqref{L1negl}. From \eqref{defrho+}, the last display is equivalent to
\begin{align*}
& \int_{\mathbb{R}^{d-1} \times (-\infty,0)} \rho_t( \hat{u}_{\star}, -u_d ) G_t( \hat{u}_{\star}, u_d ) \, d \hat{u} -  \int_0^t \int_{\mathbb{R}^{d-1} \times (-\infty,0)} \rho_s( \hat{u}_{\star}, -u_d ) [ \kappa \mathbb{L}_{\Delta}  + \partial_s ] G_s( \hat{u}_{\star}, u_d ) \, d \hat{u} \, ds \\
+ & \int_{\mathbb{R}^{d-1} \times (0, \infty)} \rho_t( \hat{u}_{\star}, u_d ) G_t( \hat{u}_{\star}, u_d ) \, d \hat{u} -  \int_0^t \int_{\mathbb{R}^{d-1} \times (0, \infty)} \rho_s( \hat{u}_{\star}, u_d ) [ \kappa \mathbb{L}_{\Delta}  + \partial_s ] G_s( \hat{u}_{\star}, u_d ) \, d \hat{u} \, ds.
\end{align*}
Applying the change of variables $( \hat{u}_{\star}, -u_d ) \mapsto ( \hat{v}_{\star}, v_d )$ in the first line of the last display, we can rewrite $F_{\textrm{Dif}}(t,\rho^{+},G,0,\kappa)$ as
\begin{align*}
& \int_{\mathbb{R}^{d-1} \times (0, \infty)} \rho_t( \hat{v}_{\star}, v_d ) G_t( \hat{v}_{\star}, - v_d ) \, d \hat{v} -  \int_0^t \int_{\mathbb{R}^{d-1} \times  (0, \infty)} \rho_s( \hat{v}_{\star}, v_d ) [ \kappa \mathbb{L}_{\Delta}  + \partial_s ] G_s( \hat{v}_{\star}, - v_d ) \, d \hat{v} \, ds \\
+ & \int_{\mathbb{R}^{d-1} \times (0, \infty)} \rho_t( \hat{u}_{\star}, u_d ) G_t( \hat{u}_{\star}, u_d ) \, d \hat{u} -  \int_0^t \int_{\mathbb{R}^{d-1} \times (0, \infty)} \rho_s( \hat{u}_{\star}, u_d ) [ \kappa \mathbb{L}_{\Delta}  + \partial_s ] G_s( \hat{u}_{\star}, u_d ) \, d \hat{u} \, ds
\end{align*}
From \eqref{defGstarplus}, the last display can be rewritten as
\begin{align*}
& \int_{\mathbb{R}^{d-1} \times (0, \infty)} \rho_t( \hat{u}_{\star}, u_d ) G^{*+}_t( \hat{u}_{\star}, u_d ) \, d \hat{u} -  \int_0^t \int_{\mathbb{R}^{d-1} \times (0, \infty)} \rho_s( \hat{u}_{\star}, u_d ) [ \kappa \mathbb{L}_{\Delta}  + \partial_s ] G^{*+}_s( \hat{u}_{\star}, u_d ) \, d \hat{u} \, ds \\
=& \int_{\mathbb{R}^{d-1} \times [0, \infty)} \rho_t( \hat{u}_{\star}, u_d ) G^{*+}_t( \hat{u}_{\star}, u_d ) \, d \hat{u} -  \int_0^t \int_{\mathbb{R}^{d-1} \times [0, \infty)} \rho_s( \hat{u}_{\star}, u_d ) [ \kappa \mathbb{L}_{\Delta}  + \partial_s ] G^{*+}_s( \hat{u}_{\star}, u_d ) \, d \hat{u} \, ds \\
=& \int_{\mathbb{R}^{d} } \rho_t( \hat{u}_{\star}, u_d ) G^{*}_t( \hat{u}_{\star}, u_d ) \, d \hat{u} -  \int_0^t \int_{\mathbb{R}^{d}} \rho_s( \hat{u}_{\star}, u_d ) [ \kappa \mathbb{L}_{\Delta}  + \partial_s ] G^{*}_s( \hat{u}_{\star}, u_d ) \, d \hat{u} \, ds=0,
\end{align*}
for every $t \in [0,T]$. {The first and second equalities in the last display hold due to \eqref{L1negl} and to the definition of $G^{*}$.} In {the} last equality above we used \eqref{uniqneu}, since $G^{*} \in \mcb{S}_{\textrm {Neu}}$. Since $G \in \mcb S_{\textrm{Dif}}$ is arbitrary, we conclude that $\rho^{+}$ is the unique weak solution of \eqref{ehd} with initial condition equal to zero. In particular, due to the arguments in the proof of Proposition \ref{propuniqdif}, $\rho^{+} \equiv 0$ a.e. in $[0,T] \times  \mathbb{R}^d$, which leads to $\rho \equiv 0$ a.e. in $[0,T] \times \mathbb{R}^{d-1} \times [0, \infty)$. 

Reasoning in an analogous way, we get $\rho \equiv 0$ a.e. in $[0,T] \times \mathbb{R}^{d-1} \times  (- \infty,0]$, ending the proof.
\end{proof}

	\section{Discrete convergences} \label{secdiscconv}

In this section we state and prove Lemmas \ref{lemconvprinc} and \ref{lemconvprincdisc}, which are useful to obtain Propositions \ref{convprincdif} and \ref{convprincdisc}. Recall from \eqref{deftilknj} the definition of $\mcb{K}_{n, j}$ and the fact that $\mcb{K}_{n, \mcb B} = \sum_{j=1}^d \mcb{K}_{n, j}$. Due to $\Delta = \sum_{j=1}^d \partial_{\hat{e}_d \hat{e}_d}$, Proposition \ref{convprincdif} is a direct consequence of next result. 

\begin{lem} \label{lemconvprinc}
\textit{	
	Assume that either} $G \in \mcb S_{\textrm{Dif}}$ \textit{and $j \in \{1, \ldots, d\}$; or} $G \in \mcb{S}_{\textrm {Disc}}$ \textit{and $j \in \{1, \ldots, d-1\}$. Then 
\begin{align*} 
\lim_{n \rightarrow \infty} \frac{1}{n^d} \sum_{ \hat{x} } \sup_{s \in [0,T]} \big| \Theta(n)  \mcb{K}_{n, j} G_s (\tfrac{\hat{x}}{n} )  - \kappa_{\gamma}  \partial_{\hat{e}_j \hat{e}_j} G_s  (\tfrac{\hat{x}}{n} ) \big| =0. 
\end{align*}
}
\end{lem}

\begin{proof}
{	
First, we observe that from the assumptions of Lemma \ref{lemconvprinc}, $G$ is of class $C^{2}$ along the direction determined by $\hat{e}_j$ and $\partial_{\hat{e}_j \hat{e}_j} G_s( \hat{u} )$ is well-defined for every $\hat{u} \in \mathbb{R}^d$.} The expression inside the limit in {the} last display can be rewritten as
\begin{align*}
\frac{1}{n^d} \sum_{ |\hat{x}_{\star,j}| < b_G n } \; \sum_{x_j \in \mathbb{Z}} \; \sup_{s \in [0,T]} \big| \Theta(n) \mcb{K}_{n, j} G_s ( \tfrac{ \hat{x}_{\star, j}}{n} , \tfrac{x_j}{n}) - \kappa_{\gamma}  \partial_{\hat{e}_j \hat{e}_j} G_s  ( \tfrac{ \hat{x}_{\star, j}}{n} , \tfrac{x_j}{n})\big|,
\end{align*}
due to the fact that $G_s ( \tfrac{ \hat{x}_{\star, j}}{n} , \tfrac{y_j}{n} ) = G_s ( \tfrac{ \hat{x}_{\star, j}}{n} , \tfrac{x_j}{n} ) = \mcb{K}_{n, j} G_s ( \tfrac{ \hat{x}_{\star, j}}{n} , \tfrac{x_j}{n}) = \partial_{\hat{e}_j \hat{e}_j} G_s  ( \tfrac{ \hat{x}_{\star, j}}{n} , \tfrac{x_j}{n}) =0$ when $|\hat{x}_{\star,j}| \geq b_G n$ and $s \in [0,T]$, where $b_G>0$ comes from \eqref{defbg}. {The last} display is bounded from above by
\begin{align}
& \frac{1}{n^d} \sum_{ |\hat{x}_{\star,j}| < b_G n } \; \sum_{|x_j| > 2 b_G n} \; \sup_{s \in [0,T]} \big| \Theta(n) \mcb{K}_{n, j} G_s ( \tfrac{ \hat{x}_{\star, j}}{n} , \tfrac{x_j}{n}) - \kappa_{\gamma}  \partial_{\hat{e}_j \hat{e}_j} G_s  ( \tfrac{ \hat{x}_{\star, j}}{n} , \tfrac{x_j}{n})\big| \label{convprincfar0} \\
+& \frac{1}{n^d} \sum_{ |\hat{x}_{\star,j}| < b_G n } \; \sum_{|x_j| \leq 2 b_G n} \; \sup_{s \in [0,T]} \big| \Theta(n) \mcb{K}_{n, j} G_s ( \tfrac{ \hat{x}_{\star, j}}{n} , \tfrac{x_j}{n}) - \kappa_{\gamma}  \partial_{\hat{e}_j \hat{e}_j} G_s  ( \tfrac{ \hat{x}_{\star, j}}{n} , \tfrac{x_j}{n}) \big|. \label{convprincnear0}
\end{align}
From \eqref{deftilknj} and the definition of $b_G$, the sum in \eqref{convprincfar0} as can be rewritten as
	\begin{align*}
		& \frac{1}{n^d} \sum_{ |\hat{x}_{\star,j}| < b_G n } \; \sum_{|x_j| > 2 b_G n} \; \sup_{s \in [0,T]}  \Theta(n) \sum_{|y_j| \leq b_G n }  \Big| G_s ( \tfrac{ \hat{x}_{\star, j}}{n} , \tfrac{y_j}{n} )  p\big((y_j-x_j) \hat{e}_j \big) \Big| \\
		& \lesssim \| G \|_{\infty} { \frac{\Theta(n)}{n^{\gamma}} }  \sum_{|x_j| > 2b_Gn} \; \sum_{|y_j| \leq b_Gn} \frac{1}{n^2} \Big( \frac{|y_j - x_j|}{n} \Big)^{-\gamma-1} \lesssim   \frac{\Theta(n)}{n^{\gamma}}, 
	\end{align*}
which vanishes as $n \rightarrow \infty$, due to \eqref{timescale}. Above we used the fact that the double sum in the second line is the Riemann sum of a finite double integral, since $\gamma \geq 2$. Therefore we are done if we can prove that the expression in \eqref{convprincnear0} also goes to zero when $n \rightarrow \infty$. In order to do so, we observe from \eqref{deftilknj} that for every $\varepsilon >0$, $\mcb{K}_{n, j} G_s  ( \tfrac{ \hat{x}_{\star, j}}{n} , \tfrac{x_j}{n})$ can be rewritten as
\begin{align*}
\sum_{|r| \geq \varepsilon n } [ G_s ( \tfrac{ \hat{x}_{\star, j}}{n} , \tfrac{x_j+r}{n} ) -G_s ( \tfrac{ \hat{x}_{\star, j}}{n} , \tfrac{x_j}{n}) ] p (r \hat{e}_j ) + \sum_{|r| < \varepsilon n } [ G_s ( \tfrac{ \hat{x}_{\star, j}}{n} , \tfrac{x_j+r}{n} ) -G_s ( \tfrac{ \hat{x}_{\star, j}}{n} , \tfrac{x_j}{n}) ] p( r \hat{e}_j ).
\end{align*}
Now we proceed in the same way as we did to deal with \eqref{convprincfar0}. Note that
\begin{align*}
&  \frac{1}{n^d} \sum_{ |\hat{x}_{\star,j}| < b_G n } \; \sum_{|x_j| \leq 2 b_G n} \; \sup_{s \in [0,T]} \Big| \Theta(n) \sum_{|r| \geq \varepsilon n } [ G_s ( \tfrac{ \hat{x}_{\star, j}}{n} , \tfrac{x_j+r}{n} ) -G_s ( \tfrac{ \hat{x}_{\star, j}}{n} , \tfrac{x_j}{n}) ] p( r \hat{e}_j )  \Big| \\
 \leq & \frac{2 c_{\gamma} \|  G \|_{\infty} \Theta(n)}{n^d d}  \sum_{ |\hat{x}_{\star,j}| < b_G n }  \;\sum_{|x_j| \leq 2 b_G n} \; \sum_{|r| \geq  \varepsilon n} |r|^{-\gamma-1}   \lesssim   \frac{\Theta(n)}{n^{\gamma}}  \frac{1}{n}  \sum_{r = \varepsilon n}^{\infty} \Big( \frac{r}{n} \Big)^{-\gamma-1}   \lesssim \frac{\Theta(n)}{ n^{\gamma}} \varepsilon^{-\gamma}, 
\end{align*}
which vanishes as $n \rightarrow \infty$, for every $\varepsilon>0$, again due to \eqref{timescale}. Therefore, the desired result follows if we can prove that {the display in \eqref{convprinc3} below}
	\begin{align} \label{convprinc3}
		\frac{1}{n^d} \sum_{ |\hat{x}_{\star,j}| < {b_G} n } \; \sum_{|x_j| \leq 2 b_G n} \; \sup_{s \in [0,T]} \Big| \Theta(n) \sum_{|r| < \varepsilon n } [ G_s ( \tfrac{ \hat{x}_{\star, j}}{n} , \tfrac{x_j+r}{n} ) -G_s ( \tfrac{ \hat{x}_{\star, j}}{n} , \tfrac{x_j}{n} ) ] p( r \hat{e}_j )  - \kappa_{\gamma} \partial_{\hat{e}_j \hat{e}_j} G_s  ( \tfrac{ \hat{x}}{n} )  \Big| 
	\end{align}
{goes to zero when $n \rightarrow \infty$ and afterwards, $\varepsilon \rightarrow 0^+$.}
Since $p(\cdot)$ is symmetric, $ \sum_{|r| < \varepsilon n } r p( r \hat{e}_j ) =0$ for every $\varepsilon >0$. By performing a second order Taylor expansion on $G_s(\tfrac{ \hat{x}_{\star, j}}{n} , \cdot)$ around $\tfrac{x_j}{n}$, the term inside the absolute value in \eqref{convprinc3} can be rewritten as
\begin{align*}
 &    \kappa_{\gamma}   \partial_{\hat{e}_j \hat{e}_j} G_s (  \tfrac{ \hat{x}_{\star, j}}{n}, \tfrac{x_j}{n} ) \Big( - 1 + \frac{\Theta(n)}{2 n^2 d \kappa_{\gamma}} \sum_{|r| < \varepsilon n} c_{\gamma} r^{1-\gamma} \mathbbm{1}_{\{r \neq 0\}} \Big) \\
 +& \frac{\Theta(n)}{2 n^2 d} \sum_{|r| < \varepsilon n} c_{\gamma} r^{1-\gamma} \mathbbm{1}_{\{r \neq 0\}}   \big[\partial_{\hat{e}_j \hat{e}_j} G_s ( \chi^n_{s,x_j,r} ) -  \partial_{\hat{e}_j \hat{e}_j} G_s (\tfrac{ \hat{x}_{\star, j}}{n}, \tfrac{x_j}{n} ) \big]  ,
\end{align*}
for some $\chi^n_{s, x_j, r}$ between $\tfrac{x_j}{n}$ and $\tfrac{x_j+r}{n}$. Then \eqref{convprinc3} is bounded from above by
	\begin{align}
		&   \frac{1}{n^d} \sum_{| \hat{x}_{\star,j} | < {b_G} n} \;  \sum_{|x_j| \leq 2 b_G n} \; \sup_{s \in [0,T]} \Big|     \frac{\kappa_{\gamma}}{d} \partial_{\hat{e}_j \hat{e}_j} G_s (\tfrac{ \hat{x}_{\star, j}}{n}, \tfrac{x_j}{n} )  \Big( - 1 + \frac{\Theta(n)}{2 n^2 d \kappa_{\gamma}} \sum_{|r| < \varepsilon n} c_{\gamma} r^{1-\gamma} \mathbbm{1}_{\{r \neq 0\}} \Big)  \Big| \label{convprinc4dif} \\
		+ &   \frac{c_{\gamma}}{n^d} \sum_{| \hat{x}_{\star,j} | < {b_G} n} \; \sum_{|x_j| \leq 2 b_G n} \; \sup_{s \in [0,T]} \Big|  \frac{\Theta(n)}{2 n^2 d} \sum_{|r| < \varepsilon n}  r^{1-\gamma} \mathbbm{1}_{\{r \neq 0\}}   [  \partial_{\hat{e}_j \hat{e}_j} G_s (\tfrac{ \hat{x}_{\star, j}}{n}, \chi^n_{s,x_j,r} ) -  \partial_{\hat{e}_j \hat{e}_j} G_s (\tfrac{ \hat{x}_{\star, j}}{n}, \tfrac{x_j}{n} ) ]  \Big|. \label{convprinc5dif}
	\end{align}
The term in \eqref{convprinc4dif} is bounded from above by a constant times
\begin{align*}
&   \| \partial_{\hat{e}_j \hat{e}_j} G \|_{\infty}  \Big| - 1 + \frac{\Theta(n)}{2 n^2 d \kappa_{\gamma}} \sum_{|r| < \varepsilon n} c_{\gamma} r^{1-\gamma} \mathbbm{1}_{\{r \neq 0\}} \Big|,
\end{align*} {which goes to zero as $n \rightarrow \infty$ for every $\varepsilon >0$ fixed}, due to \eqref{convkappagamma}. Next we bound \eqref{convprinc5dif} from above by
	\begin{align*}
		&  \frac{1}{n^d} \sum_{| \hat{x}_{\star,j} | < {b_G} n} \; \sum_{|x_j| \leq 2 b_G n} \frac{\Theta(n)}{2 n^2 {d} } \sum_{|r| < \varepsilon n} c_{\gamma} r^{1-\gamma} \mathbbm{1}_{\{r \neq 0\}} \sup_{s \in [0,T], \hat{u}, \hat{v}  \in \mathbb{R}^d: |\hat{v}- \hat{u}| \leq \varepsilon} |\partial_{\hat{e}_j \hat{e}_j} G_s ( \hat{v} ) -  \partial_{\hat{e}_j \hat{e}_j} G_s ( \hat{u} ) | \\
		\lesssim &  \sup_{s \in [0,T], \hat{u}, \hat{v}  \in \mathbb{R}^d: |\hat{v}- \hat{u}| \leq \varepsilon} |\partial_{\hat{e}_j \hat{e}_j} G_s ( \hat{v} ) -  \partial_{\hat{e}_j \hat{e}_j} G_s ( \hat{u} ) |,
	\end{align*}
which vanishes as $\varepsilon \rightarrow 0^+$, since each of the functions $\partial_{\hat{e}_j \hat{e}_j} G$ is uniformly continuous. The bound in second line of {the} last display holds due to \eqref{convkappagamma}. This proves \eqref{convprinc3} , ending the proof.
\end{proof}
Recall the {definitions} of $\tilde{\mcb{K}}_{n,d}$ in \eqref{op_Knbd} and of {$\mathbb{L}_{d}$ in \eqref{defLd}}. Proposition \ref{convprincdisc} follows by combining  {Lemma \ref{lemconvprinc} with the next result.}  
\begin{lem} \label{lemconvprincdisc}	
\textit{	
	Assume that either $\gamma =2$ and} $G \in \mcb{S}_{\textrm {Neu}}$\textit{; or $\gamma > 2$ and} $G \in \mcb{S}_{\textrm {Disc}}$\textit{. Then
	\begin{equation*} 
		\lim_{n \rightarrow \infty} \frac{1}{n^d} \sum_{| \hat{x}_{\star} | {<} b_G {n} } \; \sum_{x_d \in \mathbb{Z}} \; \sup_{s \in [0,T]} \big| \Theta(n) \tilde{\mcb{K}}_{n,d} G_s (\tfrac{\hat{x}_{\star}}{n}, \tfrac{x_d}{n} ) - \kappa_{\gamma} { \mathbb{L}_{d}} G_s(\tfrac{\hat{x}_{\star}}{n}, \tfrac{x_d}{n} ) \big| =0.
	\end{equation*}
}
\end{lem}
\begin{proof}
{
	Combining \eqref{op_Knbd} and \eqref{defLd},} the expression inside the limit in {the} last display {is bounded from above by}
	\begin{align*}
		&  \frac{1}{n^d} \sum_{ |\hat{x}_{\star}| < b_G n } \; \sum_{i=1}^4 \mathcal{B}^G_{n,i}(\hat{x}_{\star}), 
	\end{align*}
where for every $\hat{x}_{\star} \in \mathbb{Z}^{d-1} $ and $n \geq 1$,
\begin{align*}
 \mathcal{B}^G_{n,1}(\hat{x}_{\star}):= \sup_{s \in [0,T]}\Big| \Theta(n) & \sum_{r=0 }^{\infty} \Big[  [G^{+}_s( \tfrac{\hat{x}_{\star}}{n} , \tfrac{r}{n}) - G^{+}_s( \tfrac{\hat{x}_{\star}}{n} , \tfrac{0}{n})  ]    - \frac{{r-0}}{n}   \partial_{\hat{e}_d} G^{+}_s( \tfrac{\hat{x}_{\star}}{n} ,0)   \Big] p(r \hat{e}_d) - \kappa_{\gamma} \partial_{\hat{e}_d \hat{e}_d} G^{+}_s( \tfrac{\hat{x}_{\star}}{n} ,0 )     \Big|, \\
 \mathcal{B}^G_{n,2}(\hat{x}_{\star}):= \sum_{x_d=1}^{\infty} { \sup_{s \in [0,T]} } \Big|& \Theta(n)  \Big[  [G^{+}_s( \tfrac{\hat{x}_{\star}}{n} ,\tfrac{0}{n}) - G^{+}_s( \tfrac{\hat{x}_{\star}}{n} , \tfrac{x_d}{n})  ] - \frac{ { 0-x_d } }{n} \partial_{\hat{e}_d} G^{+}_s( \tfrac{\hat{x}_{\star}}{n} , 0)  \Big]   p(-x_d \hat{e}_d )    \Big|, \\
 \mathcal{B}^G_{n,3}(\hat{x}_{\star}):= \sum_{x_d=1}^{\infty} { \sup_{s \in [0,T]} } \Big| & \Theta(n) \sum_{y_d=1 }^{\infty} \Big[ \big[  G^{+}_s( \tfrac{\hat{x}_{\star}}{n} ,\tfrac{y_d}{n}) - G^{+}_s( \tfrac{\hat{x}_{\star}}{n} , \tfrac{x_d}{n})    \big]  - \frac{ { y_d-x_d }}{n} \partial_{\hat{e}_d} G^{+}_s( \tfrac{\hat{x}_{\star}}{n} ,0)  \Big]  p\big((y_d-x_d)\hat{e}_d \big) \\
 - & \kappa_{\gamma} \partial_{\hat{e}_d \hat{e}_d} G^{+}_s( \tfrac{\hat{x}_{\star}}{n} ,\tfrac{x_d}{n} )     \Big|, \\
\mathcal{B}^G_{n,4}(\hat{x}_{\star}):=  \sum_{x_d=-\infty}^{-1} { \sup_{s \in [0,T]} } \Big| &  \Theta(n) \sum_{y_d=- \infty }^{-1} \big[  G^{-}_s( \tfrac{\hat{x}_{\star}}{n} ,\tfrac{y_d}{n}) -G^{-}_s( \tfrac{\hat{x}_{\star}}{n} , \tfrac{x_d}{n})    \big]   - \frac{ { y_d-x_d } }{n} \partial_{\hat{e}_d} G^{-}_s( \tfrac{\hat{x}_{\star}}{n} ,0)  \Big] p\big((y_d-x_d)\hat{e}_d \big) \\
 - & \kappa_{\gamma} \partial_{\hat{e}_d \hat{e}_d} G^{-}_s( \tfrac{\hat{x}_{\star}}{n} ,\tfrac{x_d}{n} )     \Big|.
\end{align*}
In {the} last display we applied \eqref{op_Knbd}. Therefore we are done if we can prove that for every $i \in \{1, 2, 3, 4\}$,
\begin{equation} \label{claimlemA2}
\lim_{n \rightarrow \infty}  \sum_{ |\hat{x}_{\star}| < b_G n } \frac{1}{n^{d}} \mathcal{B}^G_{n,i}(\hat{x}_{\star})=0.
\end{equation}
We begin by obtaining \eqref{claimlemA2} for $i=1$. In this case, the sum in \eqref{claimlemA2} is bounded from above by
\begin{align}
&   \frac{\kappa_{\gamma}}{n^{d}} \sum_{ |\hat{x}_{\star}| < b_G n }  \|  \partial_{\hat{e}_d \hat{e}_d} G^{+} \|_{\infty} \nonumber   \\
+&     \frac{1}{n^{d}} \sum_{ |\hat{x}_{\star}| < b_G n } \;  { \Theta(n) \sum_{r=0 }^{\infty} \sup_{s \in [0,T]} \Big|  [G^{+}_s( \tfrac{\hat{x}_{\star}}{n} ,\tfrac{r}{n}) - G^{+}_s( \tfrac{\hat{x}_{\star}}{n} ,0 )  ]     - \frac{r}{n} \partial_{\hat{e}_d} G^{+}_s( \tfrac{\hat{x}_{\star}}{n} , 0)   \Big| p(r \hat{e}_d )      .} \label{eqconvneu0a1}
\end{align}
The expression inside the first line of {the} last display is {of} order $n^{-1}$ and vanishes as $n \rightarrow \infty$. Next we treat \eqref{eqconvneu0a1}. {If $\gamma=2$ and $G \in \mcb{S}_{\textrm {Neu}}$, then} $\partial_{\hat{e}_d} G^{+}_s( \tfrac{\hat{x}_{\star}}{n} ,0)=0$ for every $\hat{x}_{\star} \in \mathbb{Z}^{d-1}$ and every $s \in [0,T]$. Hence, by performing a first order Taylor expansion on $G_s^{+}(\tfrac{\hat{x}_{\star}}{n} , \cdot)$ around $0$, \eqref{eqconvneu0a1} is bounded from above by
\begin{align*}
  \frac{\Theta(n)}{n^{d+1}} \sum_{ |\hat{x}_{\star}| < b_G n } \; \sum_{r=1 }^{\infty} r p(r \hat{e}_d ) \| \partial_{\hat{e}_d} G^{+} \|_{\infty} \lesssim    \frac{1}{\log(n)} \sum_{r=1 }^{\infty} r^{-2} \lesssim    \frac{1}{\log(n)},
\end{align*}
which goes to zero as $n \rightarrow \infty$. Next, we treat the case $\gamma > 2$, where $\Theta(n)=n^2$. By performing a second order Taylor expansion on $G_s^{+}( \tfrac{\hat{x}_{\star}}{n} , \cdot)$ around $0$, \eqref{eqconvneu0a1} is bounded from above by
\begin{align*}
 \frac{\Theta(n)}{n^{d}} \sum_{ |\hat{x}_{\star}| < b_G n } \frac{\Theta(n)}{2n^2} \sum_{r=1 }^{\infty} r^2 p(r \hat{e}_d ) \| \partial_{\hat{e}_d \hat{e}_d} G^{+} \|_{\infty} \lesssim    \frac{1}{n} \sum_{r=1 }^{\infty} r^{1-\gamma}  \lesssim    \frac{1}{n},
\end{align*}
which vanishes as $n \rightarrow \infty$. {Then, we conclude that \eqref{claimlemA2} holds when $i=1$. Furthermore, we observe that for $i=2$, the sum in \eqref{claimlemA2} is bounded from above by the display in \eqref{eqconvneu0a1}, which goes to zero as $n \rightarrow \infty$, as we proved above. It remains to treat the cases $i=3$ and $i=4$.}
We observe that the proof of \eqref{claimlemA2} for $i=4$ is analogous to the proof for $i=3$, then we omit it for $i=4$. Now for $i=3$, the expression inside the limit in \eqref{claimlemA2} is bounded from above by
\begin{align}
&  \frac{1}{n^{d}} \sum_{ |\hat{x}_{\star}| < b_G n } \; \sum_{x_d=1}^{\varepsilon n - 1} \; \sup_{s \in [0,T]} \Big| \kappa_{\gamma} \partial_{\hat{e}_d \hat{e}_d} G^{+}_s( \tfrac{\hat{x}_{\star}}{n} ,\tfrac{x_d}{n} )     \Big| \label{eqconvneuposa} \\
+& \frac{\Theta(n)}{n^{d}} \sum_{ |\hat{x}_{\star}| < b_G n } \; \sum_{x_d=1}^{\varepsilon n - 1} \; \sup_{s \in [0,T]} \Big|  \sum_{r>-x_d } \Big[  [G^{+}_s( \tfrac{\hat{x}_{\star}}{n} ,\tfrac{x_d+r}{n}) - G^{+}_s( \tfrac{\hat{x}_{\star}}{n} ,\tfrac{x_d}{n})  ] - \frac{r}{n} \partial_{\hat{e}_d} G^{+}_s( \tfrac{\hat{x}_{\star}}{n} , 0)  \Big]   p\big(r\hat{e}_d \big)      \Big| \label{eqconvneuposb} \\
+&  \frac{1}{n^{d}} \sum_{ |\hat{x}_{\star}| < b_G n } \; \sum_{x_d = \varepsilon n}^{\infty} \frac{\Theta(n)}{n} \sup_{s \in [0,T]} \Big| \partial_{\hat{e}_d} G^{+}_s( \tfrac{\hat{x}_{\star}}{n} ,0) \sum_{r=1 - x_d }^{\infty}    r    p\big(r\hat{e}_d \big)    \Big| \label{eqconvneuposc} \\
+&  \frac{1}{n^{d}} \sum_{ |\hat{x}_{\star}| < b_G n } \; \sum_{x_d = \varepsilon n }^{\infty} \; \sup_{s \in [0,T]} \Big| \Theta(n) \sum_{r > - x_d }   [G^{+}_s( \tfrac{\hat{x}_{\star}}{n} ,\tfrac{x_d+r}{n}) - G^{+}_s( \tfrac{\hat{x}_{\star}}{n} , \tfrac{x_d}{n})  ]   p\big( r\hat{e}_d \big) - \kappa_{\gamma} \partial_{\hat{e}_d \hat{e}_d} G^{+}_s( \tfrac{\hat{x}_{\star}}{n} ,\tfrac{x_d}{n} )     \Big|, \label{eqconvneuposd}
\end{align}
for any $\varepsilon >0$. Therefore, we are done if we can prove that \eqref{eqconvneuposa}, \eqref{eqconvneuposb}, \eqref{eqconvneuposc} and \eqref{eqconvneuposd} go to zero when first $n \rightarrow \infty$ and afterwards, $\varepsilon \rightarrow 0^+$.   
We begin by treating \eqref{eqconvneuposa}. We bound it from above by
\begin{align*}
 \frac{1}{n^{d-1}} \sum_{ |\hat{x}_{\star}| < b_G n } \; \sum_{x_d=1}^{\varepsilon n-1}  \frac{\kappa_{\gamma}}{n} \| \partial_{\hat{e}_d \hat{e}_d} G^{+}  \|_{\infty}    \leq (2 b_G)^{d-1} \kappa_{\gamma} \| \partial_{\hat{e}_d \hat{e}_d} G^{+}  \|_{\infty}  \varepsilon \lesssim \varepsilon,
\end{align*}
which vanishes when $\varepsilon \rightarrow 0^+$.
Next we bound \eqref{eqconvneuposb} from above by 
 \begin{align}
&   \frac{\Theta(n)}{n^{d}} \sum_{ |\hat{x}_{\star}| < b_G n } \; \sum_{x_d=1}^{\varepsilon n - 1 } \; \sup_{s \in [0,T]} \Big|  \sum_{r=1-x_d}^{x_d-1} \Big[  [G^{+}_s( \tfrac{\hat{x}_{\star}}{n} , \tfrac{x_d+r}{n}) - G^{+}_s( \tfrac{\hat{x}_{\star}}{n} , \tfrac{x_d}{n}) ] - \frac{r}{n} \partial_{\hat{e}_d} G^{+}_s( \tfrac{\hat{x}_{\star}}{n} ,0) \Big]   p(r \hat{e}_d )  \Big| \label{lemneu1posa} \\
+&  \frac{\Theta(n)}{n^{d}} \sum_{ |\hat{x}_{\star}| < b_G n } \; \sum_{x_d=1}^{\varepsilon n - 1 } \; \sup_{s \in [0,T]} \Big| \sum_{r=x_d}^{\varepsilon n -1} \Big[  [G^{+}_s( \tfrac{\hat{x}_{\star}}{n} ,\tfrac{x_d+r}{n}) - G^{+}_s( \tfrac{\hat{x}_{\star}}{n} ,\tfrac{x_d}{n}) ]    - \frac{r}{n} \partial_{\hat{e}_d} G^{+}_s( \tfrac{\hat{x}_{\star}}{n} ,0) \Big] p(r \hat{e}_d )  \Big| \label{lemneu1posb} \\
+ &  \frac{\Theta(n)}{n^{d}} \sum_{ |\hat{x}_{\star}| < b_G n } \; \sum_{x_d=1}^{\varepsilon n - 1 } \; \sup_{s \in [0,T]}  \ \Big| \sum_{r= \varepsilon n}^{\infty} \Big[ [G^{+}_s( \tfrac{\hat{x}_{\star}}{n} , \tfrac{x_d+r}{n}) - G^{+}_s( \tfrac{\hat{x}_{\star}}{n} ,\tfrac{x_d}{n}) ]    - \frac{r}{n} \partial_{\hat{e}_d} G^{+}_s( \tfrac{\hat{x}_{\star}}{n} ,0) \Big] p(r \hat{e}_d )  \Big|. \label{lemneu1posc}
\end{align}
Due to the symmetry of $p(\cdot)$, by performing a second order Taylor expansion on $G_s^{+}(\tfrac{\hat{x}_{\star}}{n}, \cdot)$ around $\tfrac{x_d}{n}$, the display in \eqref{lemneu1posa} is bounded from above by
\begin{align*}
&  \frac{1}{n^{d}} \sum_{ |\hat{x}_{\star}| < b_G n } \frac{\Theta(n) \| \partial_{\hat{e}_d \hat{e}_d} G^{+} \|_{\infty}}{2 n^2 }  \sum_{x_d=1}^{\varepsilon n - 1} \; \sum_{r=1-x_d}^{x_d-1} r^2    p(r \hat{e}_d ) \lesssim   \frac{\Theta(n)}{ n^{3}} \sum_{x_d=1}^{\varepsilon n } \; \sum_{r=1}^{\varepsilon n} r^{1-\gamma} \lesssim  \varepsilon \frac{\Theta(n)}{ n^{2}} \sum_{r=1}^{\varepsilon n} r^{1-\gamma} \lesssim  \varepsilon,
\end{align*}
which goes to zero as $\varepsilon \rightarrow 0^+$. Above we applied \eqref{thetan}. From a first order Taylor expansion {of first order} on $G^{+}_s(\tfrac{\hat{x}_{\star}}{n},\cdot)$ around $\tfrac{x_d}{n}$, the double limit in \eqref{lemneu1posb} is bounded from above by
\begin{align*}
&  \frac{1}{n^{d}} \sum_{ |\hat{x}_{\star}| < b_G n }  \frac{\Theta(n)}{  n}  \sum_{x_d=1}^{\varepsilon n - 1} \; \sum_{r=x_d}^{\varepsilon n} r    p(r \hat{e}_d ) \sup_{s \in [0,T], \, \hat{u}, \hat{v} \in \mathbb{R}^d: |\hat{u} - \hat{v} | \leq  2 \varepsilon  } | \partial_{\hat{e}_d} G^{+}_s( \hat{u}  ) - \partial_{\hat{e}_d} G^{+}_s( \hat{v})   | \\
  \lesssim &  \frac{1}{n^{d-1}} \sum_{ |\hat{x}_{\star}| < b_G n }  \frac{\Theta(n)}{  n^2} \sum_{r=1}^{\varepsilon n }   r^{1-\gamma} \sup_{s \in [0,T], \, \hat{u}, \hat{v} \in \mathbb{R}^d: |\hat{u} - \hat{v} | \leq  2 \varepsilon  } | \partial_{\hat{e}_d} G^{+}_s( \hat{u}  ) - \partial_{\hat{e}_d} G^{+}_s( \hat{v})   |\\
  \lesssim &   \sup_{s \in [0,T], \, \hat{u}, \hat{v} \in \mathbb{R}^d: |\hat{u} - \hat{v} | \leq  2 \varepsilon  } | \partial_{\hat{e}_d} G^{+}_s( \hat{u}  ) - \partial_{\hat{e}_d} G^{+}_s( \hat{v})   |,
\end{align*}
which vanishes when $\varepsilon \rightarrow 0^+$, due to the uniform continuity of $\partial_{\hat{e}_d} G^{+}$. Above we used \eqref{thetan}. Next we treat \eqref{lemneu1posc}, which is bounded from above by a constant times
\begin{align*}
&  \frac{\Theta(n)}{n^{d}} \sum_{ |\hat{x}_{\star}| < b_G n } \sum_{x_d=1}^{\varepsilon n - 1} \Big[ 2 \| G \|_{\infty}    \sum_{r= \varepsilon n}^{\infty}   r^{-\gamma-1} +  \frac{1}{n} \sum_{r= \varepsilon n}^{\infty} r^{-\gamma} \| \partial_{\hat{e}_d} G^{+} \|_{\infty}   \Big]   \\
\lesssim &  \varepsilon   \frac{\Theta(n)}{n^{\gamma}}  \Big[  \frac{1}{n}  \sum_{z_d= \varepsilon n}^{\infty}   \Big( \frac{z_d}{n} \Big)^{-\gamma-1} +     \frac{1}{n} \sum_{z_d= \varepsilon n}^{\infty}   \Big( \frac{z_d}{n} \Big)^{-\gamma} \Big] \lesssim    \frac{\Theta(n)}{n^{\gamma}}  \varepsilon^{1-\gamma} \big( 1 +   \varepsilon \big),
\end{align*}
which goes to zero when $n \rightarrow \infty$ for every $\varepsilon > 0$, due to the fact that $\lim_{n \rightarrow \infty}  \Theta(n)/n^{\gamma}=0$ (this comes from \eqref{timescale}). Thus, we conclude that \eqref{eqconvneuposb} vanishes when first $n \rightarrow \infty$ and then $\varepsilon \rightarrow 0^+$.
Next we analyse \eqref{eqconvneuposc}, and we note that it is equal to zero for {$\gamma=2$ and $G \in \mcb{S}_{\textrm {Neu}}$, since in this case} we have $\partial_{\hat{e}_d} G^{+}_s( \tfrac{\hat{x}_{\star}}{n} ,0)=0$ for every $\hat{x}_{\star} \in \mathbb{Z}^{d-1}$ and every $s \in [0,T]$. If $\gamma > 2$, from the symmetry of $p(\cdot)$, we rewrite \eqref{eqconvneuposc} as
\begin{align*}
 \frac{1}{n^{d-1}} \sum_{ |\hat{x}_{\star}| < b_G n } \; \sum_{x_d=\varepsilon n}^{\infty}  \Big|  \sum_{r=x_d }^{\infty}   \partial_{\hat{e}_d} G^{+}_s( \tfrac{\hat{x}_{\star}}{n} ,0) r    p(r \hat{e}_d )    \Big| \lesssim   \sum_{x_d=\varepsilon n}^{\infty} \; \sum_{r=x_d }^{\infty} r^{-\gamma} \lesssim \sum_{x_d=\varepsilon n}^{\infty} (x_d)^{1-\gamma} \lesssim (\varepsilon n)^{2 - \gamma},
\end{align*}
which again goes to zero when $n \rightarrow \infty$ for every $\varepsilon > 0$, since we are in the case $\gamma >2$. It remains to treat \eqref{eqconvneuposd}, which is bounded from above by the sum
\begin{align}
&  \frac{1}{n^{d}} \sum_{ |\hat{x}_{\star}| < b_G n } \; \sum_{x_d = 2 b_G n+1}^{\infty} \; \sup_{s \in [0,T]} \Big| \Theta(n) \sum_{r >  - x_d }^{\infty}   G^{+}_s( \tfrac{\hat{x}_{\star}}{n} ,\tfrac{x_d+r}{n})    p(r\hat{e}_d )   \Big| \label{eqconvneuposd1} \\
+& \frac{1}{n^{d}} \sum_{ |\hat{x}_{\star}| < b_G n } \; \sum_{x_d=\varepsilon n }^{2 b_G n} \; \sup_{s \in [0,T]} \Big| \Theta(n) \sum_{r > - x_d, |r| \geq \varepsilon n}^{\infty}   [G^{+}_s( \tfrac{\hat{x}_{\star}}{n} ,\tfrac{x_d+r}{n}) - G^{+}_s( \tfrac{\hat{x}_{\star}}{n} , \tfrac{x_d}{n})  ]    p (r\hat{e}_d )     \Big| \label{eqconvneuposd2} \\
+& \frac{1}{n^{d}} \sum_{ |\hat{x}_{\star}| < b_G n } \; \sum_{x_d=\varepsilon n }^{2 b_G n}  \; \sup_{s \in [0,T]} \Big| \Theta(n) \sum_{|r| < \varepsilon n }   [G^{+}_s( \tfrac{\hat{x}_{\star}}{n} ,\tfrac{x_d+r}{n}) - G^{+}_s( \tfrac{\hat{x}_{\star}}{n} , \tfrac{x_d}{n})  ]    p(r \hat{e}_d ) - \kappa_{\gamma} \partial_{\hat{e}_d \hat{e}_d} G^{+}_s( \tfrac{\hat{x}_{\star}}{n} , \tfrac{x_d}{n} )     \Big|.  \label{eqconvneuposd3}
\end{align}
In \eqref{eqconvneuposd1} we used the fact that $\partial_{\hat{e}_d \hat{e}_d} G^{+}_s( \tfrac{\hat{x}_{\star}}{n} ,\tfrac{x_d}{n} )=G^{+}_s( \tfrac{\hat{x}_{\star}}{n} , \tfrac{x_d}{n})=0$ when $x_d > 2 b_G$, due to the definition of $b_G$ in \eqref{defbg}. Thus, \eqref{eqconvneuposd1} is bounded from above by 
\begin{align*}
&   \frac{\Theta(n)}{n^{d}} \sum_{ |\hat{x}_{\star}| < b_G n } \; \sum_{x_d > 2 b_G n} \;   \sum_{y_d=1 }^{b_G n} \|G\|_{\infty} \frac{c_{\gamma} (x_d-y_d)^{-1-\gamma}}{d} \lesssim  \frac{\Theta(n)}{n^{\gamma}} \frac{1}{n^2} \sum_{x_d=2 b_G n }^{\infty} \; \sum_{y_d=1 }^{b_G n} \Big( \frac{x_d-y_d}{n} \Big)^{-1-\gamma} \lesssim \frac{\Theta(n)}{n^{\gamma}},
\end{align*}
which vanishes as $n \rightarrow \infty$, again due to \eqref{timescale}. Next we bound the display \eqref{eqconvneuposd2} from above by
\begin{align*}
 \frac{\Theta(n)}{n^{d}} \sum_{ |\hat{x}_{\star}| < b_G n } \; \sum_{x=1}^{2 b_G n} 4 c_{\gamma} \| G \|_{\infty}  \sum_{r = \varepsilon n }^{\infty} r^{-\gamma-1}  \lesssim   \frac{\Theta(n)}{n^{\gamma}}  \frac{1}{n} \sum_{r = \varepsilon n }^{\infty}   \Big( \frac{r}{n} \Big)^{-\gamma} \lesssim  \frac{\Theta(n)}{n^{\gamma}} \varepsilon^{-\gamma},  
\end{align*}
which goes to zero as $n \rightarrow \infty$, for every $\varepsilon > 0$ {fixed}. Applying an argument analogous to the one used to treat \eqref{convprinc3}, we have that \eqref{eqconvneuposd3} also vanishes as $n \rightarrow \infty$ and then $\varepsilon \rightarrow 0^+$, ending the proof.
\end{proof}

\section{Estimates on Dirichlet Forms} \label{estdirfor}

In this section we present some lemmas that  are useful to prove Proposition \ref{bound}. Recall the definition of a element of $\textrm{Ref}$ in Definition \ref{defrefprof}, and also the definitions of $a_h$, $b_h$, $K_h$ and $L_h$ in Definition \ref{defrefprof}. 
\begin{lem} \label{lemaux0}
\textit{
Let} $h \in \textrm{Ref}$\textit{. Then there exist $M_{1,h}>0$ depending only on $h$ such that
\begin{align} \label{boundprof}
\forall n \geq 1, \quad \frac{\Theta(n)}{n^d}  \sum_{ \hat{x} , \hat{y}} p(\hat{y}-\hat{x}) [ h( \tfrac{\hat{y}}{n} ) - h( \tfrac{\hat{x}}{n} )]^2 \leq M_{1,h}. 
\end{align}
}
\end{lem}
\begin{proof}
From \eqref{prob}, the left-hand side of {the} last display can be rewritten as
\begin{align*}
\frac{\Theta(n) c_{\gamma} }{dn^d} \sum_{j=1}^d \; \sum_{ |\hat{x}_{\star,j}| < K_h n } \; \sum_{x_j, y_j \in \mathbb{Z}} [ h( \tfrac{\hat{x}_{\star,j}}{n}, \tfrac{y_j}{n} ) -h( \tfrac{\hat{x}_{\star,j}}{n}, \tfrac{x_j}{n} )]^2 |y_j - x_j |^{-1-\gamma} \mathbbm{1}_{ \{ x_j \neq y_j  \} }, 
\end{align*}
since $ h( \tfrac{\hat{x}_{\star,j}}{n}, \tfrac{y_j}{n} ) -h( \tfrac{\hat{x}_{\star,j}}{n}, \tfrac{x_j}{n} ) = 0$ for $|\hat{x}_{\star,j}|  \geq  K_h n$.  {The last} display is bounded from above by
\begin{align}
& \frac{\Theta(n) c_{\gamma} }{dn^d} \sum_{j=1}^d \; \sum_{ |\hat{x}_{\star,j}| < K_h n } \Big[ \sum_{|y_j| \leq  K_h n}  \;\sum_{|x_j| \geq 2 K_h n}  |y_j-x_j|^{-\gamma-1}    +     \sum_{|x_j| \leq 2 K_h n}  \;\sum_{|y_j| \geq 3 K_h n}  |y_j-x_j| ^{-\gamma-1}  \Big]  \label{profbound1b1} \\
+ & \frac{\Theta(n) c_{\gamma} }{dn^{d+2} } \sum_{j=1}^d  \; \sum_{ |\hat{x}_{\star,j}| < K_h n } \;  \sum_{x_j = -2 K_h n+1}^{2 K_h n-1} \; \sum_{ y_j = - 3 K_h n }^{ x_j-1} L_h (x_j-y_j)^{-\gamma+1}   \label{profbound1b2} \\
+ & \frac{\Theta(n) c_{\gamma} }{dn^{d+2} } \sum_{j=1}^d \; \sum_{ |\hat{x}_{\star,j}| < K_h n }  \;  \sum_{x_j = -2 K_h n+1}^{2 K_h n-1} \; \sum_{ y-j = x_j+1 }^{ 3 K_h n}  L_h (y_j-x_j)^{-\gamma+1}, \label{profbound1b3}
\end{align}
since $ h( \tfrac{\hat{x}_{\star,j}}{n}, \tfrac{y_j}{n} ) -h( \tfrac{\hat{x}_{\star,j}}{n}, \tfrac{x_j}{n} ) = 0$ when $\min \{ |x_j|, |y_j| \} \geq K_h n$, $h$  being uniformly bounded by $1$ and $| h( \tfrac{ \hat{y} }{n} ) - h( \tfrac{\hat{x}}{n} )| \leq L_h n^{-1} | \hat{y} - \hat{x}|$. The expression in \eqref{profbound1b1} is bounded from above by
\begin{align*}
& \frac{ \Theta(n)}{dn^{\gamma}} \sum_{j=1}^d \frac{1}{n^{d-1}} \sum_{ |\hat{x}_{\star,j}| < K_h n }  \frac{2 c_{\gamma}}{n^2} \Big[ \sum_{|y_j| \leq  K_h n} \; \sum_{x_j \geq 2 K_h n}  \Big( \frac{x_j-y_j}{n} \Big)^{-\gamma-1} + \sum_{|x_j| \leq 2 K_h n} \; \sum_{y_j \geq 3 K_h n}  \Big( \frac{y_j-x_j}{n} \Big)^{-\gamma-1} \Big]  \lesssim \frac{ \Theta(n)}{n^{\gamma}},
\end{align*}
that goes to zero as $n \rightarrow \infty$, due to \eqref{timescale}. It remains to treat \eqref{profbound1b2} and \eqref{profbound1b3}. From $|x_j-y_j| \leq 5 K_h n$ in the double sums over $x_j$ and $y_j$, the sum of \eqref{profbound1b2} and \eqref{profbound1b3} is bounded from above by a constant times
\begin{align*}
   (K_h)^{d-1} L_h \sum_{x = -2 K_h n+1}^{2 K_h n-1}  \sum_{ z=1 }^{ 5 K_h n}  \frac{\Theta(n)}{n^3}  z^{-\gamma+1} \leq 4 K_h^d L_h  \frac{\Theta(n)}{n^2} \sum_{ z=1 }^{ 5 K_h n} z^{-\gamma+1}.
\end{align*}
From \eqref{timescale}, {the} last display can be rewritten as
\begin{align*}
\begin{dcases}
\frac{4 K_h^d L_h}{\log(n)} \sum_{ z=1 }^{ 5 K_h n}  z^{-1}  \leq \frac{4 K_h^d L_h}{\log(n)} [1 + \log( 5 K_h n)] \leq \frac{4 K_h^d L_h}{\log(n)} + 4 K_h^d L_h  \frac{\log(5 K_h)}{\log(n)} , & \quad \gamma =2, \\
4 K_h^d L_h \sum_{ z=1 }^{ 5 K_h n}   z^{-\gamma+1} \leq 4 K_h^d L_h\sum_{ z=1 }^{ \infty}  z^{-\gamma+1}, & \quad \gamma >2,
\end{dcases}
\end{align*}
hence the sum of \eqref{profbound1b1} and \eqref{profbound1b2} is bounded from above by a positive constant depending only {on $h$.} 
\end{proof}
Next we present a result which is very useful to estimate Radon-Nikodym derivatives of the form $d \nu_h^n ( \eta ) / d \nu_h^n ( \eta^{\hat{x}, \hat{y}} ) $. {The next} result is a direct consequence of the definition of $\nu_h^n$ in Definition \ref{defberprod}, therefore we omit its proof.
\begin{lem} \label{lemaux1}
\textit{
Let} $h \in \textrm{Ref}$\textit{, $\eta \in \Omega$, $x \neq y \in \mathbb{Z}$. Then \begin{align*}
 \frac{\nu_h^n ( \eta )}{\nu_h^n ( \eta^{\hat{x},\hat{y}} )} \leq \Big( \frac{b_h}{a_h} \Big)^2  \quad \text{and} \quad \Big| \frac{ \nu_h^n (\eta^{\hat{x},\hat{y}})}{ \nu_h^n (\eta)} -1 \Big| \leq \frac{| h( \tfrac{\hat{y}}{n} ) - h( \tfrac{\hat{x}}{n} )|}{(a_h )^2 }.
\end{align*}
In particular, if $f$ is a density with respect to $\nu_h^n$, we have
\begin{align*} 
\int_{\Omega} f(\eta) \, d \nu_h^n + \int_{\Omega} f(\eta^{\hat{x}, \hat{y}}) \, d \nu_h^n \leq 1 + \Big( \frac{b_h}{a_h} \Big)^2.
\end{align*}
}
\end{lem}
{The next} result is an adaptation of Lemma 5.1 in \cite{byrondif} to our model.
\begin{lem} \label{lemaux2}
\textit{
For every} $h \in \textrm{Ref}$ \textit{there exists $M_{2,h}>0$ depending only on $h$ such that
\begin{equation} \label{eqlemaux2lip}
 \int_{\Omega}  \big[ \sqrt{f(\eta^{\hat{x},\hat{y}})} - \sqrt{f(\eta)} \big] \sqrt{f(\eta)} \, d \nu_h^n \leq  - \frac{1}{4}  I_{ \hat{x} , \hat{y}} (\sqrt{f}, \nu_h^n)     +M_{2,h} [ h( \tfrac{\hat{y}}{n} ) - h( \tfrac{\hat{x}}{n} )]^2, 
\end{equation}
for every $f$ density with respect to $\nu_h^n$.
}
\end{lem}
\begin{proof}
Here, we follow the proof of Lemma 5.1 in \cite{byrondif}. The left-hand side of \eqref{eqlemaux2lip} can be rewritten as 
\begin{align*} 
& -  \frac{1}{2}  \int_{\Omega}   [  \sqrt{f(\eta^{\hat{x}, \hat{y}})} - \sqrt{f(\eta)} ]^2 \,   d \nu_h^n + \frac{1}{2} \int_{\Omega} [  \sqrt{f(\eta^{\hat{x}, \hat{y}})}]^2  \Big[ 1 - \frac{\nu_h^n ( \eta^{\hat{x},\hat{y}} ) }{\nu_h^n (\eta)} \Big] \, d \nu_h^n  \\
=& -  \frac{1}{2}  I_{ \hat{x} , \hat{y}} (\sqrt{f}, \nu_h^n)+ \frac{1}{4} \int_{\Omega}  [  \sqrt{f(\eta^{\hat{x},\hat{y}})} - \sqrt{f(\eta)} ]  [  \sqrt{f(\eta^{\hat{x},\hat{y}})} + \sqrt{f(\eta)} ] \Big[ 1 - \frac{\nu_h^n ( \eta^{\hat{x},\hat{y}} ) }{\nu_h^n (\eta)} \Big] \, d \nu_h^n. 
\end{align*}
Above we performed the change of variables $\eta \rightarrow \eta^{\hat{x},\hat{y}}$. Now we prove \eqref{eqlemaux2lip}. Applying Young's inequality, {the} last display is bounded from above by 
\begin{align*}
-  \frac{1}{2}  I_{ \hat{x} , \hat{y}} (\sqrt{f}, \nu_h^n) + \frac{1}{4}  I_{ \hat{x} , \hat{y}} (\sqrt{f}, \nu_h^n) + \frac{1}{16} \int_{\Omega}  [  \sqrt{f(\eta^{\hat{x}, \hat{y}})} + \sqrt{f(\eta)} ]^2   \Big[ 1 - \frac{\nu_h^n ( \eta^{\hat{x}, \hat{y}} ) }{\nu_h^n (\eta)} \Big]^2 \, d\nu_h^n(\eta).
\end{align*}
Combining Lemma \ref{lemaux1} with the fact that $(u+v)^2 \leq 2(u^2+v^2)$, the rightmost term in {the} last display is bounded from above by
\begin{align*}
 \frac{[h ( \tfrac{\hat{y}}{n} ) - h ( \tfrac{\hat{x}}{n} ) ]^2}{16 (a_h)^4}  \int_{\Omega}  2 [  \sqrt{f(\eta^{\hat{x},\hat{y}})} +  \sqrt{f(\eta)} ]^2 \, d \nu_h^n (\eta)  \leq  \frac{[h ( \tfrac{\hat{y}}{n} ) - h ( \tfrac{\hat{x}}{n} ) ]^2}{8 (a_h)^4} \Big[ 1  +  \Big( \frac{b_h}{a_h} \Big)^2  \Big],
\end{align*}
leading to \eqref{eqlemaux2lip} and ending the proof.
\end{proof} 
Finally we are ready to show Proposition \ref{bound}.
\begin{proof}[Proof of Proposition \ref{bound}]
From \eqref{eqlemaux2lip}, $\langle \mcb {L}_{n} \sqrt{f} , \sqrt{f} \rangle_{\nu_h^n}$ is bounded from above by
\begin{align*}
 & \frac{1}{2} \sum_{\{ \hat{x}, \hat{y} \} \in \mcb F} p(\hat{y} - \hat{x}) \Big( - \frac{1}{4}  I_{\hat{x}, \hat{y}} (\sqrt{f}, \nu_h^n)   + M_{2,h} [h ( \tfrac{\hat{y}}{n} ) - h ( \tfrac{\hat{x}}{n} ) ]^2  \Big) \\
+ & \frac{\alpha}{2n^{\beta}} \sum_{\{ \hat{x}, \hat{y} \} \in \mcb S} p(\hat{y}-\hat{x}) \Big( - \frac{1}{4}  I_{\hat{x},\hat{y}} (\sqrt{f}, \nu_h^n)   + M_{2,h} [h ( \tfrac{\hat{y}}{n} ) - h ( \tfrac{\hat{x}}{n} ) ]^2 \Big) \\
\leq & - \frac{1}{4}  D_n (\sqrt{f}, \nu_h^n) +\frac{\alpha+1}{2} M_{2,h}  \sum_{\hat{x}, \hat{y}} p(\hat{y}-\hat{x}) [ h( \tfrac{\hat{y}}{n} ) - h( \tfrac{\hat{x}}{n} )]^2.
\end{align*} 
Multiplying {the} last inequality by $\Theta(n) / n^d$, we get from Lemma \ref{lemaux0} that
\begin{align*}
\frac{\Theta(n)}{n^d} \langle \mcb{L}_{n} \sqrt{f} , \sqrt{f} \rangle_{\nu_h^n} \leq - \frac{\Theta(n)}{4n^d}  D_n (\sqrt{f}, \nu_h^n ) +\frac{\alpha+1}{2} M_{1,h} M_{2,h},
\end{align*}
and \eqref{boundlip} follows, ending the proof.
\end{proof}

\section{Proof of Lemma \ref{lemtight} } \label{prooftight} 

In this section we prove Lemma \ref{lemtight}. We begin by observing that the space $\mcb {M}^{+}$ defined in Definition \ref{radon} is a Polish space, due to Theorem 31.5 of \cite{bauer}. This means that we are done if we can prove that $(\mathbb{Q}_n)_{n \geq 1}$ satisfies both conditions of Theorem 1.3 in Chapter 4 of \cite{kipnis1998scaling}. 

We begin with the first condition. Keeping this in mind, define $H \subset \mcb {M}^{+}$ as
\begin{align*}
H:= \Big \{ \pi \in  \mcb {M}^{+}: \forall G \in C_c(\mathbb{R}^d), \quad \Big| \int G \, d \pi \Big| \leq 3 (b_G)^d \|G \|_{\infty} \Big\}.
\end{align*} 
From Definition 31.1 and Theorem 31.2 of \cite{bauer}, $H$ is vaguely relatively compact, therefore its closure $\bar{H}$ is vaguely compact. Now from \eqref{intmedemp} and \eqref{exc1part}, we have that $| \langle \pi_t^n, G \rangle | \leq 3 (b_G)^d \|G \|_{\infty}$, where $b_G$ is given in \eqref{defbg}. This leads to 
\begin{align*}
\forall n \geq 1, \forall t \in [0,T], \quad \mathbb{Q}_n \big( \pi_t^n \in \bar{H} \big) =1.
\end{align*} 
Therefore, choosing $K(t,\varepsilon)=H$ in the first condition of Theorem 1.3 in Chapter 4 of \cite{kipnis1998scaling}, we have 
\begin{equation} \label{1condlemtight}
 \forall t \in [0,T], \forall \varepsilon >0, \quad \sup_{n \geq 1} \mathbb{Q}_n \big( \pi_t^n \notin \bar{H} \big) = 0 < \varepsilon.
\end{equation}
In order to obtain the second condition of Theorem 1.3 in Chapter 4 of \cite{kipnis1998scaling}, we adapt the strategy used in the proof of Proposition 1.7 in Chapter 4 of \cite{kipnis1998scaling} to our context. First, we need to introduce a clever metric in $\mcb {M}^{+}$. Following the arguments in the proof of Theorem 31.5 in \cite{bauer}, there exists a countable set $D \subset C_c^0(\mathbb{R}^d)$ such that the mapping $\tilde{\delta}: \mcb {M}^{+} \times \mcb {M}^{+} \rightarrow [0, \infty)$ given by
\begin{equation} \label{metr1}
\tilde{\delta}(\pi, \hat{\pi}):= \sum_{k=1}^{\infty} \frac{1}{2^k} \min \Big\{ 1, \Big| \int f_k \, d \pi - \int f_k \, d \tilde{\pi} \Big| \Big\} 
\end{equation} 
determines a topology that is exactly the vague topology in $\mcb {M}^{+}$. In {the} last display, $\{f_1, f_2, \ldots\}$ is any fixed enumeration of the elements of $D$. Now, motivated by equation (1.1) in Chapter 4 of \cite{kipnis1998scaling}, we define the mapping $\delta: \mcb {M}^{+} \times \mcb {M}^{+} \rightarrow [0, \infty)$ as
\begin{equation} \label{metr2}
\delta(\pi, \hat{\pi}):= \sum_{k=1}^{\infty} \frac{1}{2^k} \frac{\Big| \int f_k \, d \pi - \int f_k \, d \tilde{\pi} \Big|}{1 + \Big| \int f_k \, d \pi - \int f_k \, d \tilde{\pi} \Big|}.
\end{equation} 
From \eqref{metr1} and \eqref{metr2}, we have that $\delta(\pi, \hat{\pi}) \leq \tilde{\delta}(\pi, \hat{\pi}) \leq 2 \delta(\pi, \hat{\pi})$, for every $\pi$, $\tilde{\pi}$ in $\mcb {M}^{+}$. In particular, $\delta$ determines a topology in $\mcb {M}^{+}$ which also coincides with the vague topology.

In the same way as it is done in Chapter 4 of \cite{kipnis1998scaling}, given $\pi \in \mcb {M}^{+}$, we define
\begin{align*}
\forall \lambda >0, \quad & \omega_{\pi}(\lambda):= \inf_{\{t_j\}_{0 \leq j \leq r }} \; \; \max_{0 \leq j < r} \; \; \sup_{t_j \leq s < t < t_{j+1}} \delta(\pi_s, \pi_t ), \\
\forall f \in C_c^0(\mathbb{R}^d), \; \forall \lambda >0, \quad & \omega_{\langle \pi, f \rangle}(\lambda):= \inf_{\{t_j\}_{0 \leq j \leq r }} \; \; \max_{0 \leq j < r} \; \; \sup_{t_j \leq s < t < t_{j+1}} \Big| \int f \, d \pi_s - \int f \, d \pi_t \Big|.
\end{align*}  
Above, the infima are taken over all the partitions $\{t_0, t_1, \ldots t_r\}$ of $[0,T]$ such that $0 = t_0 < t_1 < \ldots < t_r = T$ and $t_j - t_{j-1} > \epsilon$, for any $j \in \{1, \ldots, r\}$. Therefore, following the second condition of Theorem 1.3 in Chapter 4 of \cite{kipnis1998scaling}, we are done if we can prove that
\begin{equation} \label{2condlemtight}
  \forall \varepsilon >0, \quad \lim_{\lambda \rightarrow 0^+} \varlimsup_{n \rightarrow \infty} \mathbb{Q}_n \big(  \pi^n:  \omega_{\pi^n}(\lambda) > \varepsilon \big) = 0.
\end{equation}
Above, $\pi^n:=(\pi_t^n)_{0 \leq t \leq T}$. In order to be able to apply \eqref{T1sdif} (which is stated for functions in $C_c^2(\mathbb{R}^d)$ instead of $C_c^0(\mathbb{R}^d)$, such as the elements of $D$), we observe that for every $f \in C_c^0(\mathbb{R}^d) \subsetneq L^1(\mathbb{R}^d)$, there exists some $(G_i^f)_{i \geq 1}\subset C_c^{\infty}(\mathbb{R}^d) \subsetneq C_c^{2}(\mathbb{R}^d)$ such that $(G_i^f)_{i \geq 1}$ converges to $f$ in $L^1(\mathbb{R}^d)$. Plugging this with \eqref{intmedemp} and \eqref{exc1part}, we get
\begin{align*}
\forall \varepsilon >0, \quad \lim_{i \rightarrow \infty} \mathbb{Q}^n \big( | \langle \pi^n_s, f - G_i^f \rangle - \langle \pi^n_t, f - G_i^f \rangle | > \varepsilon )=0, 
\end{align*}
for any $s, t \in [0,T]$. Combining {the} last display with \eqref{T1sdif}, we get
\begin{equation} \label{T2sdif}
\displaystyle \lim _{\lambda \rightarrow 0^+} \varlimsup_{n \rightarrow\infty} \sup_{\tau  \in \mathcal{T}_{T}, \;  { 0 \leq } t \leq \lambda} {\mathbb{Q}}_{n}\big( \left\vert \langle\pi^{n}_{{T \wedge( \tau+ t)}},f\rangle-\langle\pi^{n}_{\tau}, f\rangle\right\vert > \ve \big)  =0, 
\end{equation}
for any $f \in C_c^0(\mathbb{R}^d)$. Above, $\mathcal{T}_T$ is {given in Definition \ref{deftaut}. Finally,} in order to obtain \eqref{2condlemtight}, we reproduce the arguments of the second part of the proof of Proposition 1.7 in Chapter 4 of \cite{kipnis1998scaling}. More exactly, fix $\varepsilon >0$. For every $f_k$ element of $D$, we have that $\langle \pi^n, f_k \rangle  \in \mcb D([0,T], \mathbb{R})$. Thus, combining \eqref{T2sdif} with Proposition 1.7 in Chapter 4 of \cite{kipnis1998scaling}, we get that
\begin{equation} \label{limrealpro}
\forall k \geq 1, \quad \lim_{\lambda \rightarrow 0^+} \varlimsup_{n \rightarrow \infty} \mathbb{Q}_n \Big( \omega_{\langle \pi^n, f_k \rangle } (\lambda) > \frac{\varepsilon}{2}  \Big) =0.
\end{equation} 
Now fix $k_{\varepsilon} \geq 1$ such that $2^{1-k_{\varepsilon}} < \varepsilon$. Thus, from \eqref{metr2}, we have
\begin{align} \label{boundmod}
\forall \lambda >0, \forall n \geq 1,  \quad \omega_{\pi^n}(\lambda) \leq \frac{\varepsilon}{2} + \sum_{k=1}^{k_{\varepsilon}} \frac{1}{2^k} \omega_{\langle \pi^n, f_k \rangle } (\lambda).
\end{align}
Next, fix $\beta >0$. From \eqref{limrealpro}, there exists $\lambda_0 >0$ such that
\begin{align*}
\forall \lambda \in (0, \lambda_0], \; \forall k \in \{1, \ldots, k_{\varepsilon}\}, \quad \varlimsup_{n \rightarrow \infty} \mathbb{Q}^n \Big( \omega_{ \langle \pi^n, f_k \rangle } > \frac{\varepsilon}{2} \Big) \leq \frac{\beta}{2^k}.
\end{align*}
This leads to
\begin{align*}
\forall \lambda \in (0, \lambda_0], \; \forall k \in \{1, \ldots, k_{\varepsilon}\}, \quad \varlimsup_{n \rightarrow \infty} \mathbb{Q}^n \Big( \sum_{k=1}^{k_{\varepsilon}} \frac{1}{2^k} \omega_{ \langle \pi^n, f_k \rangle } > \frac{\varepsilon}{2} \Big) \leq \beta.
\end{align*}
Combining {the} last display with \eqref{boundmod}, we get
\begin{align*}
\forall \lambda \in (0, \lambda_0], \quad \varlimsup_{n \rightarrow \infty} \mathbb{Q}^n \Big(  \omega_{  \pi^n } \geq \varepsilon  \Big) \leq \beta.
\end{align*} 
Since $\beta >0$ is arbitrary, we conclude that \eqref{2condlemtight} holds, ending the proof. 

\section{Proof of Proposition \ref{dynkform}} \label{miscdynk}

In this section we present the proof of Proposition \ref{dynkform}. Most of the arguments are exactly the same as the one used to obtain Lemma 5.1 in Appendix 1.5 of \cite{kipnis1998scaling}. The only difference is that, in this reference, the assumption $G$ being  of class $C^2$ regarding the time variable was required \textit{only} to prove that
\begin{equation} \label{claimdynk}
\lim_{h \rightarrow 0^+} \frac{1}{h} \int_t^{t+h} dr E_x \Big[F_r'(X_{t-s+h}) - F_t'(X_{t-s+h}) \Big]=0, 
\end{equation}
i.e., the term in the first line of page 331 in \cite{kipnis1998scaling} converges to zero as $h \rightarrow 0^{+}$. In \eqref{claimdynk}, we used the notation described in Lemma 5.1 of Appendix 1.5 in \cite{kipnis1998scaling}. In order to obtain an equivalent result in our setting, we will make use of the following lemma.
\begin{lem} \label{lemdynk}
\textit{
Let $n \geq 1$, $\gamma \geq 2$, $d \geq 1$,} $G \in \mcb{S}_{\textrm {Disc}}$ \textit{and $T >0$. Then for every $t \in [0, \; T]$, it holds
\begin{align} \label{limlemdynk}
\lim_{h \rightarrow 0^+} \frac{1}{h} \int_t^{t+h} \Bigg\{ \frac{1}{n^d}    \sum_{\hat{x}}  \big| \partial_s G(r, \tfrac{\hat{x}}{n}  ) -  \partial_s G(t, \tfrac{\hat{x}}{n}  ) \big| \Bigg\} dr = 0.
\end{align}
}
\end{lem}
\begin{proof}
Let $G^{-}, G^{+} \in \mcb S_{\textrm{Dif}}$ be such that $G=\mathbbm{1}_{\{ [0, \infty)  \times \mathbb{R}_{-}^{d*} \} } G^{-} + \mathbbm{1}_{ \{[0, \infty) \times \mathbb{R}_{+}^{d*} \}} G^{+}$ and define $b:=\max \{b_{G^{-}}, \; b_{G^{+}} \}$, where $b_{G^{-}}, b_{G^{+}}$ are given by \eqref{defbg}. Then, the limit in \eqref{limlemdynk} is bounded from above by
\begin{align*}
& \lim_{h \rightarrow 0^+} \frac{1}{h} \int_t^{t+h} \Bigg\{ \frac{1}{n^d}    \sum_{|\hat{x}| \leq b n} \big[  \big| \partial_s G^{-}(\tilde{t}, \tfrac{\hat{x}}{n}  ) -  \partial_s G^{-}(t, \tfrac{\hat{x}}{n}  ) \big| ] + \big[  \big| \partial_s G^{+}(\tilde{t}, \tfrac{\hat{x}}{n}  ) -  \partial_s G^{+}(t, \tfrac{\hat{x}}{n}  ) \big| \big]  \Bigg\} d \tilde{t} \\
\leq & \lim_{h \rightarrow 0^+} \frac{1}{h} \int_t^{t+h} \Bigg\{ \frac{1}{n^d}    \sum_{|\hat{x}| \leq b n}  \sup_{\hat{u} \in \mathbb{R}^d, \; t, r \geq 0: \; t \leq  r  \leq t+h} \big\{ |\partial_{s} G^{-}(r, \hat{u} ) -  \partial_{s} G^{-} (t, \hat{u} ) |  + |\partial_{s} G^{+}(r, \hat{u} ) -  \partial_{s} G^{+} (t, \hat{u} ) | \big\} \Bigg\} d \tilde{t} \\
\lesssim & \lim_{h \rightarrow 0^+} \sup_{\hat{u} \in \mathbb{R}^d, \; t, r \geq 0: \; t \leq  r  \leq t+h} \big\{ |\partial_{s} G^{-}(r, \hat{u} ) -  \partial_{s} G^{-} (t, \hat{u} ) |  + |\partial_{s} G^{+}(r, \hat{u} ) -  \partial_{s} G^{+} (t, \hat{u} ) | \big\} =0.
\end{align*}
In the last line, we used the fact that $\partial_{s} G^{-}$ and $\partial_{s} G^{+}$ are both uniformly continuous.
\end{proof}
Since Proposition \ref{dynkform} is a direct consequence of \eqref{exc1part} and Lemma \ref{lemdynk}, the proof ends!

\quad

\thanks{ {\bf{Acknowledgements: }}
P.C. was funded by the Deutsche Forschungsgemeinschaft (DFG, German Research Foundation) under Germany’s Excellence Strategy – EXC-2047/1 – 390685813. P.G. thanks  FCT/Portugal for financial support through the projects UIDB/04459/2020 and UIDP/04459/2020. B.J.O. thanks  Universidad Nacional de Costa Rica  for sponsoring the participation in  this article. {The authors also thank the referees for their careful work and great feedback.} This project has received funding from the European Research Council (ERC) under  the European Union's Horizon 2020 research and innovative programme (grant agreement   n. 715734).} 


\begin{thebibliography}{10}

\bibitem{B12}
C.~Bahadoran.
\newblock Hydrodynamics and Hydrostatics for a Class of Asymmetric Particle Systems with Open Boundaries.
\newblock {\em Communications in Mathematical Physics}, 310(1):1--24, 2012.


\bibitem{baldasso}
R.~Baldasso, O.~Menezes, A.~Neumann, and R.~Souza.
\newblock Exclusion process with slow boundary.
\newblock {\em J. Stat. Phys.}, 167(5):1112--1142, 2017.

\bibitem{bauer}
Heinz Bauer.
\newblock {\em Measure and integration theory}, volume~26 of {\em De Gruyter
  Studies in Mathematics}.
\newblock Walter de Gruyter \& Co., Berlin, 2001.
\newblock Translated from the German by Robert B. Burckel.

\bibitem{stefano}
C.~Bernardin, P.~Cardoso, P.~Gon{\c{c}}alves, and S.~Scotta.
\newblock Hydrodynamic limit for a boundary driven super-diffusive symmetric
  exclusion.
\newblock {\em arXiv preprint arXiv:2007.01621}, 2021.

\bibitem{byrondif}
C.~Bernardin, P.~Gon\c{c}alves, and B.~Jim\'{e}nez-Oviedo.
\newblock Slow to fast infinitely extended reservoirs for the symmetric
  exclusion process with long jumps.
\newblock {\em Markov Process. Related Fields}, 25(2):217--274, 2019.

\bibitem{byronsdif}
C.~Bernardin, P.~Gon\c{c}alves, and B.~Jim\'{e}nez-Oviedo.
\newblock A microscopic model for a one parameter class of fractional
  {L}aplacians with {D}irichlet boundary conditions.
\newblock {\em Arch. Ration. Mech. Anal.}, 239(1):1--48, 2021.

\bibitem{brezis2010functional}
H.~Brezis.
\newblock {\em Functional analysis, Sobolev spaces and partial differential
  equations}.
\newblock Springer Science \& Business Media, 2010.

\bibitem{casodif}
P.~Cardoso, P.~Gon{\c{c}}alves, and B.~Jim\'{e}nez-Oviedo.
\newblock Hydrodynamic behavior of long-range symmetric exclusion with a slow
  barrier: diffusive regime.
\newblock {\em to appear in Annales de l'IHP Probab. Stat.}, 2023+.

\bibitem{superdif}
P.~Cardoso, P.~Gon{\c{c}}alves, and B.~Jim\'{e}nez-Oviedo.
\newblock Hydrodynamic behavior of long-range symmetric exclusion with a slow
  barrier: superdiffusive regime.
\newblock {\em to appear in Annali della Scuola Normale Superiore di Pisa,
  Classe di Scienze}, 2023+.

\bibitem{evans1998partial}
L.~C. Evans.
\newblock Partial differential equations.
\newblock {\em Graduate studies in mathematics}, 19(2), 1998.

\bibitem{franco2015phase}
T.~Franco, P.~Gon{\c{c}}alves, and A.~Neumann.
\newblock Phase transition of a heat equation with robin’s boundary
  conditions and exclusion process.
\newblock {\em Transactions of the American Mathematical Society},
  367(9):6131--6158, 2015.

\bibitem{tertuaihp}
T.~Franco, P.~Gon\c{c}alves, and A.~Neumann.
\newblock Hydrodynamical behavior of symmetric exclusion with slow bonds.
\newblock {\em Ann. Inst. Henri Poincar\'{e} Probab. Stat.}, 49(2):402--427,
  2013.

\bibitem{tertumariana}
T.~Franco and M.~Tavares.
\newblock Hydrodynamic limit for the {SSEP} with a slow membrane.
\newblock {\em J. Stat. Phys.}, 175(2):233--268, 2019.

\bibitem{stefano_patricia}
P.~Gon{\c{c}}alves and S.~Scotta.
\newblock Diffusive to super-diffusive behavior in boundary driven exclusion.
\newblock {\em Markov Process. Related Fields}, 28(1):149--178, 2022.

\bibitem{GPV}
M.~Z. Guo, G.~C. Papanicolaou, and S.~R.~S. Varadhan.
\newblock Nonlinear diffusion limit for a system with nearest neighbor
  interactions.
\newblock {\em Comm. Math. Phys.}, 118(1):31--59, 1988.

\bibitem{jara2009hydrodynamic}
M.~Jara.
\newblock Hydrodynamic limit of particle systems with long jumps.
\newblock {\em arXiv preprint arXiv:0805.1326}, 2008.

\bibitem{kipnis1998scaling}
C.~Kipnis and C.~Landim.
\newblock {\em Scaling limits of interacting particle systems}, volume 320.
\newblock Springer Science \& Business Media, 1998.

\bibitem{MSV}
M.~Mourragui, E.~Saada, and S.~Velasco.
\newblock Hydrodynamic and hydrostatic limit for a generalized contact process
  with mixed boundary conditions.
\newblock {\em arXiv preprint arXiv:2212.07762}, 2022.

\bibitem{SS18}
S.~Sethuraman and D.~Shahar.
\newblock Hydrodynamic limits for long-range asymmetric interacting particle systems
\newblock {\em Electronic Journal of Probability}, 23:1--54, 2018.


\bibitem{spitzer}
F.~Spitzer.
\newblock Interaction of {M}arkov processes.
\newblock {\em Advances in Math.}, 5:246--290 (1970), 1970.

\bibitem{X22a}
L.~Xu.
\newblock Hydrodynamic limit for asymmetric simple exclusion with accelerated boundaries.
\newblock {\em arXiv preprint arXiv:2108.09345}, 2022.

\bibitem{X22b}
L.~Xu.
\newblock Hydrodynamics for One-Dimensional {ASEP} in Contact with a Class of Reservoirs.
\newblock {\em J. Stat. Phys.}, 189(1):1--26, 2022.

\bibitem{zeidler}
E.~Zeidler.
\newblock {\em Nonlinear functional analysis and its applications. {II}/{A}}.
\newblock Springer-Verlag, New York, 1990.
\newblock Linear monotone operators, Translated from the German by the author
  and Leo F. Boron.




\end{thebibliography}

\vspace{1 cm}

\Addresses

\end{document}